\documentclass[10pt, reqno]{amsart}

\usepackage[T1]{fontenc}
\usepackage{lmodern}
\usepackage{microtype}
\usepackage{amsmath,amssymb,mathtools}
\usepackage{enumitem}
\usepackage{booktabs}
\usepackage{needspace}
\usepackage{tikz}

\usepackage{xcolor}
\usepackage{listings}

\lstdefinestyle{pythonappendix}{
  language=Python,
  basicstyle=\ttfamily\scriptsize,
  numbers=left,
  numberstyle=\tiny\color{gray},
  numbersep=8pt,
  frame=single,
  columns=fullflexible,
  keepspaces=true,
  showstringspaces=false,
  breaklines=true,
  breakatwhitespace=false,
  tabsize=4
}

\usetikzlibrary{decorations.pathmorphing}
\usetikzlibrary{positioning,shapes,arrows,decorations.markings}
\usetikzlibrary{decorations.pathreplacing,calligraphy}
\usetikzlibrary{positioning,arrows.meta}
\usetikzlibrary{fit}
\definecolor{mystery}{rgb}{0,0,0.65}
\usepackage[
  colorlinks=true,
  linkcolor=mystery,
  citecolor=mystery,
  urlcolor=mystery
]{hyperref}

\usepackage{cleveref}

\newcommand{\R}{\mathbb R}
\newcommand{\C}{\mathbb C}
\newcommand{\E}{\mathbb E}
\newcommand{\one}{\mathbf 1}
\newcommand{\per}{\operatorname{per}}
\newcommand{\cyc}{\operatorname{cyc}}
\newcommand{\tr}{\operatorname{tr}}

\newcommand{\diag}{\operatorname{diag}}

\newcommand{\pf}{\operatorname{pf}}
\newcommand{\op}{\mathrm{op}}
\newcommand{\mx}{\mathrm{max}}
\newcommand{\capacity}{\mathrm{cap}}

\newcommand{\norm}[1]{\lVert#1\rVert}

\theoremstyle{plain}
\newtheorem{theorem}{Theorem}[section]
\newtheorem{lemma}[theorem]{Lemma}
\newtheorem{proposition}[theorem]{Proposition}

\newtheorem{fact}[theorem]{Fact}

\theoremstyle{definition}
\newtheorem{definition}[theorem]{Definition}
\newtheorem{remark}[theorem]{Remark}

\title[Hamiltonian Cycles in Tournaments]{The Maximum Number of Hamiltonian\\ Cycles in Tournaments}

\author{Alexander Divoux}
\address{Program in Applied and Computational Mathematics, Princeton University}
\email{adivoux@princeton.edu}

\author{Shouda Wang}
\address{Department of Mathematics, KTH Royal Institute of Technology}
\email{shoudawang@princeton.edu}

\date{August 10, 2026}

\begin{document}
\maketitle

\begin{abstract}
    Let $C(n)$ denote the maximum number of directed Hamiltonian cycles in a tournament on $n$ vertices. We prove that, as $n\to\infty$,
    \[
        C(n) = \left(\frac{\cosh 1}{\cos 1}\pm o(1)\right)\frac{(n-1)!}{2^n}.
    \]
    The bound is asymptotically attained by the regular tournament with vertices $\mathbb Z/n\mathbb Z$ where $(i,j)$ is a directed edge if $j - i \pmod n< n/2$. This answers a question of Alon and proves a conjecture of Wormald. Among the other implications of these results is asymptotically determining the maximum number of Hamiltonian paths in a tournament on $n$ vertices. We provide our own exposition of these proofs, which were originally found by GPT-5.6.

\end{abstract}

\section{Introduction}

A \textit{tournament} $T$ is an orientation of a complete graph. A \textit{Hamiltonian cycle} in $T$ is a directed cycle containing all vertices of $T$, and a \textit{Hamiltonian path} in $T$ is a directed path containing all vertices of $T$. We denote by $C(T)$ the number of Hamiltonian cycles in $T$. Let $C(n) = \max_T C(T)$, where the maximum is taken over all tournaments on $n$ vertices, and let $P(n)$ denote the     analogous maximum for Hamiltonian paths. The quantities $C(n)$ and $P(n)$ are intimately related. Indeed, 
\[
    \frac14P(n-1) \le C(n) \le \frac1nP(n),
\]
where the first inequality follows from adding a vertex with uniformly random adjacencies to a tournament containing $P(n-1)$ Hamiltonian paths, and the second inequality is clear by ignoring edges one by one in each Hamiltonian cycle.

More than 80 years ago, Szele \cite{szele1943} used the linearity of expectation to show that $P(n) \ge n!/2^{n-1}$. The same method also gives $C(n)\ge(n-1)!/2^n$. His proof is considered by many as the first application of the probabilistic method in combinatorics (see, e.g., \cite{AlonSpencer2016, Spencer1994, ErdosSpencer1974}). 

In his 1943 paper, Szele also gave the upper bound $P(n) = O(n!~\!2^{-3n/4})$. This was improved by Alon \cite{Alon1990}, who showed $C(n) = O(n^{3/2}(n-1)!/2^n)$, and subsequently by Friedgut and Kahn \cite{FriedgutKahn2005}, who showed $C(n) = O(n^\alpha(n-1)!/2^n)$ for a constant $\alpha\approx1.249$. A \textit{$1$-factor} in a tournament $T$ is a collection of vertex-disjoint directed cycles whose union contains all vertices of $T$. (Thus, any Hamiltonian cycle is a $1$-factor.) Both of the previous improvements arise from cleverly bounding the number of $1$-factors in a tournament $T$ via permanent estimates for the adjacency matrix of $T$. 

For odd $n$, let $T_n$ be the regular tournament with vertex set $\mathbb Z/ n\mathbb Z$, with a directed edge from $i$ to $j$ when $j-i\pmod n\in\{1,2,\dots,\frac{n-1}2\}$. See Figure \ref{fig:T_7} for a depiction of $T_7$. The original lower bound of $C(n) \ge (n-1)!/2^n$ was first improved by Adler, Alon, and Ross \cite{AdlerAlonRoss2001} to $C(n) \ge(e-o(1))(n-1)!/2^n$. Wormald \cite{Wormald} then gave examples of tournaments based on $T_n$ which show that, for infinitely many values of $n$,
\[
    C(n) \ge (2.85588+ o(1))\frac{(n-1)!}{2^n}.
\]
He also conjectured that this should be close to the correct answer for $C(n)$.

\begin{figure}[!]
    \centering
        \resizebox{4cm}{!}{
        \begin{tikzpicture}[scale=.6,auto=left]
        \tikzset{
          vertex/.style={
            circle,
            fill=black,
            draw=black,
            inner sep=1.8pt
          },
          midarrow/.style args={#1}{
            postaction={decorate},
            decoration={
              markings,
              mark=at position #1 with {
                \arrow{Stealth[length=2.5mm,width=2mm]}
              }
            }
          }
        }
        \foreach \i in {0,1,2,3,4,5,6} {
          \node[vertex] (v\i) at ({90-360*\i/7}:4) {};
        }
        \foreach \i in {0,1,2,3,4,5,6} {
          \pgfmathtruncatemacro{\j}{mod(\i+1,7)}
          \draw[midarrow={.55},line width=1.2pt]
            (v\i) -- (v\j);
        }
        \foreach \i in {0,1,2,3,4,5,6} {
          \pgfmathtruncatemacro{\j}{mod(\i+2,7)}
          \draw[midarrow={.54},line width=1.2pt]
            (v\i) -- (v\j);
        }
        \foreach \i in {0,1,2,3,4,5,6} {
          \pgfmathtruncatemacro{\j}{mod(\i+3,7)}
          \draw[midarrow={.76},line width=1.2pt]
            (v\i) -- (v\j);
        }
        \end{tikzpicture}}
    \caption{The regular tournament $T_7$.} \label{fig:T_7}
\end{figure}
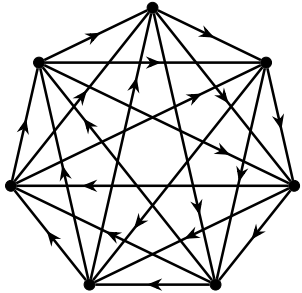

A tournament on $n$ vertices is \textit{regular} if every vertex has in-degree and out-degree exactly $(n-1)/2$ (so, a tournament is regular only if $n$ is odd). The related question of counting Hamiltonian cycles in regular tournaments was considered by Cuckler \cite{Cuckler2007}, who proved that every regular tournament on $n$ vertices contains at least $n!/(2+o(1))^n$ Hamiltonian cycles. This partially answers a conjecture of Friedgut and Kahn \cite{FriedgutKahn2005}, and answers an earlier question of Thomassen \cite{Thomassen1982}. The result was generalized to dense regular digraphs by Ferber, Krivelevich, and Sudakov \cite{FerberKrivelevichSudakov2017}.

Despite $T_n$ being a promising candidate for the tournament attaining $C(n)$, it has been surprisingly difficult to compute the asymptotics of $C(T_n)$ in practice. Indeed, the most natural probabilistic approach simply retrieves the universal bound of Cuckler for regular tournaments, which fails even to beat Szele's original lower bound on $C(n)$. On the other hand, more evidence that $T_n$ is a suitable maximizer for this setting can be observed from work on the related problem of determining, for a fixed constant $k$, the maximum number of directed $k$-cycles in a tournament on $n$ vertices. In this context, a long line of research \cite{KendallSmith1940, Colombo1964, BeinekeHarary1965, Savchenko2016, KomarovMackey2017, Grzesik2023} has recently culminated in a result of Fang and Huang \cite{FangHuang2026} establishing that an asymptotic maximizer is $T_n$ whenever $k$ is divisible by $4$, and is otherwise the random tournament, confirming a conjecture of Bartley and Day \cite{Bartley2018, Day2017}.

An alternate lens to view the question of determining whether $C(n)$ is close to the value $(n-1)!/2^n$ obtained by the random tournament is that of Sidorenko-type inequalities for directed graphs. Roughly speaking, a directed graph $H$ is \textit{tournament Sidorenko} if the random tournament minimizes the occurrences of $H$ among tournaments on a fixed number of vertices, and $H$ is \textit{tournament anti-Sidorenko} if the random tournament instead maximizes this quantity. This generalization of the usual Sidorenko setting from undirected graphs to tournaments is due to Fox, Himwich, Mani, and Zhou \cite{FoxHimwichManiZhou2025, FoxHimwichManiZhou2026}, and has been studied for various families of directed graphs \cite{ZhaoZhou2020, Grzesik2023, SahSawhneyZhao2023, HeManiNieTungWei2025, ChenClemenNoel2026}. Showing that $C(n) = \Theta((n-1)!/2^n)$ and $P(n) = \Theta(n!/2^{n-1})$ would thus establish a weak form of the tournament anti-Sidorenko property for Hamiltonian cycles and paths, respectively.

In this paper, we address the problem of closing the polynomial gap between the upper and lower bounds for $C(n)$ and $P(n)$. Interest in this question can be traced back to Alon \cite{Alon1990}, where it was deemed plausible that $C(T_n) = (1+o(1))C(n)$, and this conjectural picture was further supported in the aforementioned work of Wormald. The following is the main result of this paper, which confirms these predictions.

\begin{theorem}\label{thm:main}
    As $n\to\infty$, the maximum number of Hamiltonian cycles in a tournament on $n$ vertices is
    \[
        C(n) = \left(\frac{\cosh 1}{\cos 1}\pm o(1)\right)\frac{(n-1)!}{2^n}.
    \]
\end{theorem}

Theorem \ref{thm:main} follows immediately from Theorems \ref{thm:upper} and \ref{thm:lower}, which establish an upper bound and lower bound on $C(n)$, respectively.

\begin{theorem}\label{thm:upper}
    As $n\to\infty$, every tournament $T$ of order $n$ satisfies
    \[
        C(T)\le\left(\frac{\cosh 1}{\cos 1}+o(1)\right)\frac{(n-1)!}{2^n}.
    \]
\end{theorem}

For odd $n$, let $R_n = T_n$, and for even $n$, let $R_n$ be the tournament obtained by deleting an arbitrary vertex from $T_{n+1}$.

\begin{theorem}\label{thm:lower}
    As $n\to\infty$,
    \[
        C(R_n)\ge\left(\frac{\cosh 1}{\cos 1}-o(1)\right)\frac{(n-1)!}{2^n}.
    \]
\end{theorem}

By the above inequality relating $P(n)$ to $C(n)$, Theorem \ref{thm:main} immediately determines $P(n)$ up to a multiplicative constant, $P(n) = \Theta(n!/2^{n-1})$. A short analysis of the Hamiltonian paths in $R_n$ will allow us to extend the lower bound for $C(R_n)$ to match the upper bound given by $P(n) \le 4C(n+1)$. In particular, this determines the asymptotic constant to be $\cosh(1)/\cos(1)$ as well:

\begin{theorem}\label{thm:Hamiltonian paths}
    As $n \to \infty$, the maximum number of Hamiltonian paths in a tournament on $n$ vertices is
    \[
        P(n) = \left(\frac{\cosh1}{\cos1}\pm o(1)\right)\frac{n!}{2^{n-1}}.
    \]
\end{theorem}

\subsection{Proof Outline.} We give a high level overview of the proof, beginning with the upper bound, Theorem \ref{thm:upper}. Let $T$ be a tournament on $n$ vertices. We may assume that $C(T) \ge (n-1)!/2^n$, else there is nothing to show. The fact that $T$ contains at least as many Hamiltonian cycles as the random tournament gives structural information about $T$. In particular, using Bregman's theorem, Proposition \ref{prop:polylog set} shows that $C(T)\ge(n-1)!/2^n$ implies that the total squared difference between the outdegrees and indegrees of the vertices of $T$ is small. Thus, we may remove a small set $S$ of only polylogarithmically many vertices so that every vertex in the remaining tournament has small remaining degree imbalance, akin to a regular tournament. Setting $G = T - S$ and $N = |V(G)|$, this means that the adjacency matrix of $G$ normalized by $2/N$ has row sums and column sums close to one.

At this point, it would be nice to say something like: we can think of $G$ as behaving similarly to a regular tournament, and its normalized adjacency matrix is nearly doubly stochastic; and so we can hope to use some tools from linear algebra in order to count the Hamiltonian cycles inside this well-behaved tournament. But these are Hamiltonian cycles of $G$ rather than $T$, which neglects the fact that Hamiltonian cycles of $T$ still use all the vertices in $S$.

To address this concern, our strategy will be to partition the set of Hamiltonian cycles of $T$ based on the ``pattern'' they form in $S$. Consider a Hamiltonian cycle $C$ of $T$. The edges of $C$ with both endpoints in $S$ form a directed linear forest on $S$, with components consisting of maximal directed paths with vertices in $S\cap V(C)$. Such a component $v_1\to v_2\to\dots\to v_\ell$ appears in $C$ as part of a directed path $x\to v_1\to\dots\to v_\ell\to y$ for some attachment endpoints $x,y\in V(G)\cap V(C)$. For the purpose of counting this cycle, we will replace this entire path by a weighted transition in a modified adjacency matrix of $G$.

More specifically, after fixing a directed linear forest in $S$, as well as one attachment endpoint (in $G$) for each of its components, the possible choices of the other attachment endpoints for each component are encoded by replacing a small number of rows and columns of the normalized adjacency matrix of $G$ by uniform probability vectors, whose support consists of the appropriate destination vertices. This produces an auxiliary matrix $Q$ whose rows and columns are indexed by $V(G)$, such that the weighted sum over the cyclic permutations of $Q$,
\[
    (\text{appropriate weights})\times\sum_{\substack{\sigma\in S_N\\\sigma\text{ cyclic}}}\prod_{i=1}^NQ_{i,\sigma(i)},
\]
will upper-bound the number of Hamiltonian cycles of $T$ which induce this pattern in $S$. Proposition \ref{prop:insertions} organizes these computations.

The factor on the right hand side, which we denoted by $\cyc Q$, differs from the permanent of $Q$ only in that the sum is restricted to those permutations $\sigma\in S_N$ consisting of a single cycle. Let us ignore this discrepancy for the present moment, and pretend that we would be happy to work with $\per Q$ instead. In light of the permanent methods employed by previous work on this problem, we now find ourselves in familiar territory, and would like to use some permanent estimates like Van der Waerden's or Bregman's theorem. 

The first challenge to overcome is that universal bounds like Van der Waerden's are not precise enough for our application, and so we will instead use the permanent estimate in Proposition \ref{prop:permanent bound}, which is due to McCullagh \cite{McCullagh2014}. This permanent estimate measures how far the permanent of a doubly stochastic matrix $D$ differs from the permanent of $J_N$, the uniform matrix whose entries are all $1/N$ which achieves the minimum permanent $N!/N^N$, in terms of the difference matrix $B :=D - J_N$. (It requires some assumptions on $D$, and we also need a uniformity statement, so we prove a version suited for our application in Section \ref{sec:permanent bounds}.)

Another technical caveat is that this permanent bound takes as an input a doubly stochastic matrix; and while our matrices $Q$ are close to being doubly stochastic, in general they are not actually doubly stochastic. Thankfully, there is a standard tool for this: via Sinkhorn scaling, we can find a doubly stochastic matrix $D$ such that $D = \diag(e^x)Q\diag(e^y)$. Since every weighted Hamiltonian cycle in the sum uses exactly one entry from each row and column, diagonal scaling changes all their weights by the same factor. We can therefore estimate $\cyc D$ and then account for the scaling factor separately. The relevant statement for this tool is Proposition \ref{prop:sinkhorn}. In effect, whenever we have a tournament which is very close to being regular, translating to the relevant matrix and then applying Sinkhorn scaling gives us a way to argue as if it really were. So, it will suffice to compute $\cyc D$, and then multiply by the Sinkhorn scaling factors to retrieve $\cyc Q$.

Next we reconcile the difference between $\cyc D$ and permanents. An identity of Liu \cite{Liu1991} expresses $\cyc D$ as an alternating sum of determinants and permanents of its principal submatrices (one can view this as some form of inclusion-exclusion):
\[
    \cyc D = \frac1N\sum_{\emptyset\neq U\subseteq[N]}|U|(-1)^{|U|-1}\det D[U]\per D[U^c].
\]
Combining this with the permanent estimate and Sinkhorn scaling factors, it will follow in Proposition \ref{prop:modifications} that $\cyc Q$ satisfies
\[
    \frac{N^N\cyc Q}{(N-1)!}\le (1+o(1))\underbrace{\exp\Big(1+\sum_{1\le i,j\le N}Q_{ij}-N\Big)}_{\text{Sinkhorn scaling factors}}\underbrace{\frac{\det(I-K)}{\det(I-iK)}}_{Z(K)},
\]
where $K = (A - A^T)/N$ and $A$ is the adjacency matrix of $G$. Note that the factor denoted $Z(K)$ on the right is independent of $Q$.

Now that we understand the contribution $\cyc Q$ from each pattern, it remains to sum over all possible directed forest patterns $F$ on $S$. One possibility for such a forest arises when no two vertices of the Hamiltonian cycle $C$ are consecutive in $S$. In this case, every component of $F$ is a single vertex, and for each such vertex its possible predecessors and successors partition $V(G)$ in two. This property allows us to show that the sum of the contributions from these patterns is at most $(Z(K)+o(1))\frac{(n-1)!}{2^n}$. On the other hand, if $F$ contains $k\ge 1$ edges, then the corresponding factor will be at most $(O(N^{-1}))^k\cdot Z(K)\frac{(n-1)!}{2^n}$. In fact, even after summing over all forests with at least one edge, their total contribution is still only $o((n-1)!/2^n)$. Heuristically, this is because $|S|\sim\mathrm{polylog(n)}$ is very small, and hence a Hamiltonian cycle which enters $S$ is unlikely to continue to another vertex in $S$, but rather return to $V(G)$. Consequently, we will have shown that
\[
    C(T) \le (Z(K)+o(1))\frac{(n-1)!}{2^n}.
\]

The final ingredient in the proof is a universal upper bound for the quantity $Z(K)$ over all tournaments of a fixed order $N$. Proposition \ref{prop:Z(K) bound} will show that
\[
    Z(K) \le \frac{\sum_{j=0}^{\lfloor N/2\rfloor}\binom{N}{2j}N^{-2j}}{\sum_{j=0}^{\lfloor N/2\rfloor}(-1)^j\binom{N}{2j}N^{-2j}},
\]
where the sequences in the numerator and denominator converge to $\cosh1$ and $\cos1$ as $N\to\infty$, respectively. This will be accomplished by first showing that the quantity $Z(K)$ is maximized for transitive tournaments at every order, and then simply computing its value for the transitive tournament. The former step will use a majorization argument of Fang and Huang \cite{FangHuang2026} for the comparison. Together with the inequality bounding $C(T)$ above, this completes the upper bound.

For the lower bound, we will apply the previous framework directly to the cyclic regular tournament $T_n$. Let $A$ be the adjacency matrix of $T_n$ and $K = (A-A^T)/n$. Since $T_n$ is already a regular tournament, we do not need to delete any set of exceptional vertices $S$, and this will give equality in the arguments above up to the point $C(T_n) = (Z(K) - o(1))\frac{(n-1)!}{2^n}$. Moreover, $T_n$ is closely related to the transitive tournament on $n$ vertices. Indeed, arranging the vertices of $T_n$ in the natural cyclic order and reversing the orientations of edges crossing a diametrical edge-cut transforms it into the transitive tournament. The effect of reversing these edges will not affect the value of $Z(K)$, and as the transitive tournament converges to the maximum value of $\cosh1/\cos1$, this shows
\[
    C(T_n) = \left(\frac{\cosh1}{\cos1}-o(1)\right)\frac{(n-1)!}{2^n}.
\]

When $n$ is even, we can say the same things for the tournament $R_n$ obtained by deleting a vertex from $T_{n+1}$, but will need to use a small Sinkhorn scaling as this tournament is no longer regular. Since $C(n) \ge C(R_n)$, this completes the lower bound Theorem \ref{thm:lower}.

Finally, a few short arguments will reduce the count for Hamiltonian paths $P(n)$ to the count for Hamiltonian cycles $C(n)$, which gives Theorem \ref{thm:Hamiltonian paths}. In brief, both directions will exploit the standard trick of adding a new vertex with convenient adjacencies to a tournament $T$, such that the Hamiltonian cycles in the new directed graph can be used to count the Hamiltonian paths they induce in $T$.

\subsection{Further Directions.} It would be interesting to determine the tournament which exactly obtains the maximum number of Hamiltonian cycles (or paths). While Theorem \ref{thm:lower} implies that $T_n$ obtains the maximum asymptotically, it's not true for small orders that $C(T_n) = C(n)$. For instance, when $n = 7$, the Paley tournament $Q_7$ with vertices $\mathbb Z/7\mathbb Z$ and an edge from $i$ to $j$ whenever $j-i\pmod 7\in\{1,2,4\}$ satisfies $C(Q_7) =24 > 17=C(T_7)$. On the other hand, a stronger conjecture of Wormald \cite{Wormald} asserts that $C(n) = C(T_n)$ for all sufficiently large odd $n$.

It seems plausible that some adaptation of the techniques in this paper can be used to asymptotically determine the maximum number of $1$-factors in a tournament on $n$ vertices, and we would conjecture that $T_n$ is also an asymptotic maximizer for this quantity.

\subsection{Organization of this paper.} The remainder of this paper is organized as follows. In Section \ref{sec:preliminaries} we introduce basic notation and establish some known facts which will be used in the sequel. Sections \ref{sec:permanent bounds} and \ref{sec:Sinkhorn} give proofs for the permanent estimate and a quantitative Sinkhorn scaling result. In Section \ref{sec:inclusion-exclusion} we apply these to derive an asymptotic formula for Hamiltonian cycles, and in Section \ref{sec:modifications} we extend these estimates to modified adjacency matrices. Section \ref{sec:Z bound} proves the upper bound $\cosh1/\cos1$ on the determinant ratio $Z(K)$. The main result, Theorem \ref{thm:main}, is proved in Sections \ref{sec:upper} and \ref{sec:lower}, which establish the upper and lower bounds, respectively. Finally, the corresponding result for Hamiltonian paths is derived in Section \ref{sec:Hamiltonian paths}.

\section{Preliminaries}\label{sec:preliminaries}

The aim of this section is to recall a number of basic notions from tournaments and linear algebra and to state some facts which will be used throughout this paper.

\subsection{Notation}

Let $T$ be a tournament and $v\in V(T)$. We denote by $N_T^+(v)$ and $N_T^-(v)$ the sets of out-neighbors and in-neighbors of $v$, respectively; and $d_T^+(v) = |N_T^+(v)|$ and $d_T^-(v) = |N_T^-(v)|$. Where there is no ambiguity about which tournament is being referred to we will sometimes drop the subscripts. The \textit{order} of $T$ is $|V(T)|$. Suppose $T$ has order $n$. The \textit{adjacency matrix} $A$ of $T$ is the $n\times n$ binary matrix with rows and columns indexed by $V(T)$, with $A_{ij} = 1$ if $ij$ is a directed edge of $T$. The \textit{Seidel matrix} of $T$ is the matrix $C = A-A^T$, and $K = C/n$ is its \textit{normalized Seidel matrix}.

If $n$ is an integer, then $[n] = \{1,\dots, n\}$. We denote by $\one_n$ the vector in $\mathbb R^n$ with every coordinate equal to one, and by $J_n = \one_n\one_n^T/n$ the $n\times n$ matrix with every entry equal to $1/n$. A \textit{probability vector} is a nonnegative vector whose entries sum to one. A matrix is \textit{doubly stochastic} if it is nonnegative and every row and column sums to one.

For an $n\times n$ matrix $M$ and subsets $I,J\subseteq[n]$, let $M[I,J]$ denote the submatrix of $M$ induced by the rows indexed by $I$ and the columns indexed by $J$. We abbreviate $M[U,U]$ to $M[U]$, and write $U^c = [n]\setminus U$. We denote the permanent, determinant, and Pfaffian of $M$ by $\per M$, $\det M$, and $\mathrm{pf}~\!\!M$, respectively. Let $S_n$ denote the symmetric group on $[n]$. A permutation $\sigma\in S_n$ is \textit{cyclic} if it consists of a single cycle. We define $\cyc M$ to be the quantity 
\[
    \cyc M = \sum_{\substack{\sigma\in S_n\\\sigma\text{ cyclic}}}\prod_{i=1}^n M_{i,\sigma(i)}.
\]

For vectors, $\norm\cdot_p$ denotes the usual $\ell_p$-norm, and for matrices $\norm\cdot_p$ denotes the induced matrix norm. In particular, $\norm M_\op = \norm M_2 = \sigma_1(M)$ is the \textit{spectral norm}, where $\sigma_1(M)\ge\dots\ge\sigma_n(M)$ are the singular values of $M$. When it is clear which matrix is being referred to, we will simply write $\sigma_1\ge\dots\ge\sigma_n$. We will also use the maximum entry norm $\norm M_\mx = \max_{1\le i,j\le n}|M_{ij}|$ and the Frobenius norm 
\[
    \norm M_F = \Big(\sum_{1\le i,j\le n}|M_{ij}|^2\Big)^{1/2}.
\]

Throughout this paper, $n$ and $N$ are asymptotic quantities tending to infinity which represent the order of the relevant tournaments. We will assume that $n$ and $N$ are sufficiently large for various asymptotic inequalities to hold. By abuse of notation, a statement such as ``let $D$ be an $N\times N$ matrix'' really means ``let $(D_N)_N$ be a sequence of matrices, where $D_N$ is an $N\times N$ matrix.''

\begin{remark}\label{remark:uniformity}
    All occurrences of $o(1)$ errors are uniform in their respective parameters, which are usually constants denoted $A_0$ and $\rho$ that bound properties of matrices. In particular, for these estimates, the $o(1)$ error terms will be uniform over all matrices satisfying those bounds, although the rate may depend on $A_0$ and $\rho$. This will be useful when we take the sum over a large set of matrices of some bound which comes with a multiplicative $1+o(1)$ error, whereupon uniformity allows us to pull out this common factor.
\end{remark}

\subsection{Permanent Bounds}

We recall the following standard bounds for permanents of nonnegative matrices. The first is Van der Waerden's conjecture, proved independently by Egorychev \cite{Egorychev1981} and Falikman \cite{Falikman1981}.

\begin{fact}\label{fact:egorychev-falikman}
    If $D$ is an $N\times N$ doubly stochastic matrix, then
    \[
        \per D\ge \frac{N!}{N^N}.
    \]
\end{fact}

The second is Minc's conjecture on the permanent of binary matrices, proved by Bregman \cite{Bregman1973}. It will be convenient in our applications to use the following generalization to nonnegative matrices whose entries are bounded, due to Samorodnitsky. For a proof, see \cite[5.2.3]{Barvinok2008}.

\begin{fact}\label{fact:bregman}
    If $Q$ is a nonnegative $N\times N$ matrix with all row sums equal to $1$, and there exist positive integers $r_1,\dots, r_N$ such that $Q_{ij} \le 1/r_i$ for every $i,j\in [N]$, then
    \[
        \per Q \le \prod_{i=1}^N\frac{(r_i!)^{1/r_i}}{r_i}.
    \]
\end{fact}

Normalizing the rows of a binary matrix and applying the above recovers the classical statement of Minc's conjecture.

We record here a universal bound on the spectral norm of a tournament Seidel matrix. The proof is a quick application of the following inequality of Pick \cite{Pick1922}.

\begin{fact}\label{fact:pick}
    Let $A$ be an $N\times N$ real matrix and let $g = \max_{1\le i,j\le N}\left|\frac{A_{ij}-A_{ji}}{2}\right|$. Then every eigenvalue $\lambda$ of $A$ satisfies $|\mathrm{Im}(\lambda)|\le g\cot(\frac{\pi}{2N})$.
\end{fact}

\begin{lemma}\label{lem:Seidel norm bound}
    If $C$ is the Seidel matrix of a tournament on $N$ vertices, then
    \[
        \norm C_\op\le\cot\left(\frac\pi{2N}\right).
    \]
\end{lemma}
\begin{proof}
    Since $C$ is the Seidel matrix of a tournament, we have
    \[
        \max_{1\le i,j\le N}\left|\frac{C_{ij}-C_{ji}}{2}\right| = \left|\frac{1-(-1)}{2}\right| = 1.
    \]
    Moreover, $C$ is real and skew-symmetric, so every nonzero eigenvalue of $C$ has real part zero. Therefore, by Pick's theorem, every eigenvalue $\lambda$ of $C$ satisfies $|\lambda| = |\mathrm{Im}(\lambda)|\le\cot(\frac\pi{2N})$. Since the singular values of a real skew-symmetric matrix are the absolute values of its eigenvalues, it follows that $\norm C_\op \le \cot(\frac\pi{2N})$.
\end{proof}

\section{Permanent bounds for doubly stochastic matrices}\label{sec:permanent bounds}

In this section, we rederive the following asymptotic estimate of McCullagh for the permanent of a doubly stochastic matrix.

\begin{proposition}\label{prop:permanent bound}
    Let $A_0 < \infty$ and $\rho < 1$. If $D = J_N + B$ is an $N\times N$ doubly stochastic matrix such that $\norm B_\mx \le A_0/N$ and $\norm B_\op \le \rho$, then as $N\to \infty$,
    \[
        \frac{N^N\per D}{N!} = (1+o(1))\det(I-B^TB)^{-1/2}.
    \]
\end{proposition}

If $D$ is an $N\times N$ doubly stochastic matrix, then $\frac{N!}{N^N}\le \per D$ by Van der Waerden's conjecture (Fact \ref{fact:egorychev-falikman}), with equality obtained for $J_N$. Proposition \ref{prop:permanent bound} gives an asymptotic estimate for how far $\per D$ deviates from this lower bound as a function of the matrix $B$, which in turn represents how far $D$ deviates from $J_N$. We require some basic assumptions about $B$, namely that every entry of $B$ is $O(N^{-1})$ and the spectral norm of $B$ is bounded away from $1$.

This permanent estimate is due to McCullagh \cite{McCullagh2014}, where it is stated under slightly weaker assumptions on $D$ and the error is explicitly found to be an additive $O(N^{-1})$. For our application, we take slightly stronger assumptions on $D = J_N+B$, namely that each entry is of order $O(N^{-1})$ and the spectral norm is less than one, which are more convenient for our application. This also allows us to ensure that the $1+o(1)$ error is uniform over all matrices $D$ satisfying these assumptions; that is, the error tends to zero at a rate depending only on the fixed constants $A_0$ and $\rho$ (in accordance with Remark \ref{remark:uniformity}). We will need uniformity later on in order to pull out this $1+o(1)$ error factor when summing over the permanents of many such matrices. For these reasons we give a self-contained proof in this section.

Let us briefly outline the proof of Proposition~\ref{prop:permanent bound}. In Section \ref{sec:3.1}, we express the permanent as an expectation
$$
\frac{N^N\per D}{N!}=\E\left[R_B(N/T)\right],
$$
where $R_B$ is the rook polynomial and $T$ is a Gamma random variable. As $N\to\infty$, $N/T$ converges to $1$ in probability, so we essentially need to understand the behavior of $R_B(s)$ for $s$ close to $1$. This is done in Section \ref{sec:3.2}, where we establish in Proposition~\ref{prop:rook function asymptotic} a uniform asymptotic formula for $R_B(s)$ in a neighborhood of $1$. Finally, in Section \ref{sec:3.3}, we combine these estimates with bounds on the contribution from the tails of $T$ to complete the proof of Proposition~\ref{prop:permanent bound}.

\subsection{A probabilistic identity for the permanent}\label{sec:3.1}
Let us recall Wick's formula for the moments of complex Gaussian random variables, which will be used repeatedly. For an integer $k\ge 1$ we denote by $\mathcal P_2([2k])$ the set of partitions of $[2k]$ into pairs. A centered complex Gaussian random variable (vector resp.) is the image of a standard real Gaussian random variable (vector resp.) under a complex linear transform.

\begin{fact}
    Let $Z_1,\dots, Z_m$ be jointly centered complex Gaussian random variables, and let each $W_j$ be one of the variables $Z_i$ or $\overline{Z_i}$. Then $\mathbb E[W_1\cdots W_{2k+1}]=0$, and
    \[
        \mathbb E[W_1\cdots W_{2k}]=\sum
        _{\pi\in\mathcal P_2([2k])}\prod_{\{a,b\}\in\pi}\mathbb E[W_aW_b].
    \]
\end{fact}

Fix an $N\times N$ matrix $B$. For $0\le k\le N$, let $r_k(B)$ be the sum of the permanents of all $k\times k$ submatrices $B[I,J]$, where $I$ and $J$ range over the $k$-subsets of $[N]$, and by convention $r_0(B)=1$. Let
\[
    R_B(s)=\sum_{k=0}^N s^kr_k(B)
\]
be the generating function of $r_k(B)$. We first express $R_B(s)$ as an expectation. 

\begin{lemma}\label{lem:3.3}
    Let $B$ be a real $N\times N$ matrix and let $s\ge0$. There are jointly complex Gaussian vectors $X,Y$ satisfying $\E[X_iY_j]=B_{ij}$ and $\E[X_iX_j] = \E[Y_iY_j]=0$ for every $1\le i,j\le N$, for which
    \[
        R_B(s)=\E\left[\prod_{i=1}^N(1+\sqrt sX_i)\prod_{j=1}^N(1+\sqrt sY_j)\right].
    \]
\end{lemma}

\begin{proof}
    Let $B = U\Sigma V^T$ where $\Sigma=\diag(\sigma_1,\dots,\sigma_r)$ be a singular value decomposition of $B$. Let $z_1,\dots, z_r$ be independent standard complex Gaussian variables. In particular, $\E[z_a\bar z_b]=\delta_{ab}$ and $\E[z_az_b]=0$ hold, where $\delta_{ab}$ is the Kronecker delta. Denote by $z$ the column vector $(z_1,\dots, z_r)^T$, and define $X = U\Sigma^{1/2}z$ and $Y = V\Sigma^{1/2}\bar z$. By construction,
    \[
        \E[XY^T]=U\Sigma^{1/2}\E[z\bar z^T]\Sigma^{1/2}V^T=U\Sigma V^T=B.
    \]
    Moreover, we have $\E[XX^T]=U\Sigma^{1/2}\E[zz^T]\Sigma^{1/2}U^T=0$ and similarly we have $\E[YY^T]=V\Sigma^{1/2}\E[\bar z\bar z^T]\Sigma^{1/2}V^T=0$. It follows immediately that $\E[X_iY_j]=B_{ij}$ and $\E[X_iX_j]=\E[Y_iY_j]=0$. 
    
    To prove the formula for $R_B$, expand the terms in the expectation as
    \[
        \E\left[\prod_{i=1}^N(1+\sqrt sX_i)\prod_{j=1}^N(1+\sqrt sY_j)\right]=\sum_{I,J\subseteq[N]}s^{(|I|+|J|)/2}\E\left[\prod_{i\in I}X_i\prod_{j\in J}Y_j\right].
    \]
    For the Gaussian moment on the right side, we may apply Wick's formula since $\E[X_iX_{i'}]=\E[Y_jY_{j'}]=0$. If $|I|\neq |J|$, then this Gaussian moment vanishes. Otherwise, setting $I = \{i_1,\dots, i_k\}$ and $J = \{j_1,\dots, j_k\}$, we have by Wick's formula that
    \[
        \E\left[\prod_{i\in I}X_i\prod_{j\in J}Y_j\right]=\sum_{\pi\in S_k}\prod_{a=1}^k\E[X_{i_a}Y_{j_{\pi(a)}}]=\sum_{\pi\in S_k}\prod_{a=1}^kB_{i_a, j_{\pi(a)}}=\per(B[I,J]).
    \]
    Therefore,
    \begin{align*}
        \E\left[\prod_{i=1}^N(1+\sqrt sX_i)\prod_{j=1}^N(1+\sqrt sY_j)\right] &= \sum_{k=0}^Ns^k\sum_{\substack{I,J\subseteq[N]\\|I|=|J|=k}}\per(B[I,J]) \\&= \sum_{k=0}^Nr_k(B)s^k=R_B(s),
    \end{align*}
    as desired.   
\end{proof}

Next, we express the permanent $\per D$ as the expectation of $R_B$ evaluated at a random variable.

\begin{lemma}\label{lem:3.4}
    If $D=J_N+B$ is an $N\times N$ matrix and $T$ has the Gamma distribution $\Gamma(N+1,1)$, then
    \[
     \frac{N^N\per D}{N!}=\E[R_B(N/T)],
    \]
    where the expectation is taken with respect to $T$.
\end{lemma}

\begin{proof}
    The density of $T$ is $t^Ne^{-t}/N!$, and thus for every $0\le k\le N$ we have
    \[
        \E[T^{-k}]=\frac{(N-k)!}{N!}.
    \]
    By a direct computation,
    \begin{align*}
        \frac{N^N}{N!}\per(J_N+B)&=\frac1{N!}\sum_{\pi\in S_N}\prod_{i=1}^N(1+NB_{i,\pi(i)})\\
        &= \frac1{N!}\sum_{\pi\in S_N}\sum_{k=0}^NN^k\sum_{\substack{J\subseteq [N]\\|J|=k}}\prod_{j\in J}B_{j,\pi(j)}\\
        &= \frac1{N!}\sum_{k=0}^NN^k(N-k)!\sum_{\substack{I,J\subseteq[N]\\|I|=|J|=k}}\per(B[I,J])\\
    \end{align*}\begin{align*}
        \phantom{aaaaaaaaaaa}&= \frac1{N!}\sum_{k=0}^NN^k(N-k)!r_k(B)\\
        &= \sum_{k=0}^NN^k\E[T^{-k}]r_k(B)=\E[R_B(N/T)],
    \end{align*}
    which completes the proof.
\end{proof}

\subsection{Asymptotics of the rook function}\label{sec:3.2}
Every complex symmetric matrix $\Omega$ can be diagonalized as $\Omega= U\Sigma U^T$, where $U$ is unitary and $\Sigma$ is diagonal and nonnegative. Thus $A = U\Sigma^{1/2}$ satisfies $AA^T=\Omega$, $\norm A_\op^2 = \norm \Omega_\op$, and $\ker A^T=\ker\Omega$. In particular, if $g\sim N(0,I)$, then $Ag$ is a centered complex Gaussian vector with $\mathbb E[(Ag)(Ag)^T]=\Omega$. In the following lemma, we will introduce the complex symmetric matrix $K = I+A^TA$. By $K^{-1/2}$ we refer to the inverse square root given by $K^{-1/2}=\sum_{m=0}^\infty\binom{-1/2}m(A^TA)^m$, which converges as $\norm {A^TA}_\op<1$.

\begin{lemma}\label{lem:3.5}
    Let $\Omega=U\Sigma U^T$ be a complex symmetric $d\times d$ matrix with $\norm\Omega_\op<1$. Let $A = U\Sigma^{1/2}$ and $K = I+A^TA$. For $g\sim N(0, I_d)$, define $X = Ag$ and $Y = AK^{-1/2}g$. Then $\mathbb E[XX^T]=\Omega$ and $\mathbb E[YY^T]=\Omega(I+\Omega)^{-1}$. Moreover, for every polynomial $P\in\mathbb C[\xi_1,\dots,\xi_d]$,\footnote{Here and throughout the remainder of this section, $\det(I+\Omega)^{1/2}$ is equal to $1$ for $\Omega = 0$ and defined by analytic continuation elsewhere.}
    \begin{equation*}\label{eq:1}
        \mathbb E\left[P(X)e^{-X^TX/2}\right]=\det(I+\Omega)^{-1/2}\mathbb E[P(Y)]
    \end{equation*}
    and
    \[
        \mathbb E[P(X)]=\det(I+\Omega)^{-1/2}\mathbb E\left[P(Y)e^{Y^TY/2}\right].
    \]
\end{lemma}

\begin{proof}
    First, we have
    \[
        \E[XX^T]=A\E[gg^T]A^T=AA^T=U\Sigma U^T=\Omega.
    \]
    Since $\norm{A^TA}_\op\le\norm A^2_\op<1$, the matrix $K^{-1}$ exists, and hence
    \[
        AK^{-1/2}(AK^{-1/2})^T=A(I+A^TA)^{-1}A^T=\Omega(I+\Omega)^{-1}.
    \]
    Therefore,
    \[
        \E[YY^T]=AK^{-1/2}(AK^{-1/2})^T=\Omega(I+\Omega)^{-1}.
    \]
    Let $S$ be any complex symmetric matrix whose Hermitian part (i.e. $(S+S^*)/2$) is positive definite. For every $v\in \mathbb C^d$, it holds that
    \begin{equation}\label{eq:2}
        (2\pi)^{-d/2}\int_{\R^d}e^{-x^TSx/2+v^Tx}\mathrm dx=\det(S)^{-1/2}e^{v^TS^{-1}v/2}.
    \end{equation}
    
    Indeed, since this is true when $S$ is real positive definite and $v$ is real, the general case follows by analytic continuation in $v$ and $S$. Applying this formula with $S=K$ and then with $S=K^{-1}$ shows that for every $u\in\C^d$,
    \[
        \E[e^{u^TX-X^TX/2}]=\det(K)^{-1/2}e^{u^T\Omega(I+\Omega)^{-1}u/2},
    \]
    \[
        \E[e^{u^TY+Y^TY/2}]=\det(K)^{1/2}e^{u^T\Omega u/2}.
    \]
    By $P(\partial_u)$ we mean replacing each variable in the polynomial $P$ by the corresponding partial derivative with respect to $u$. Applying $P(\partial_u)$ to both identities and evaluating at $u=0$, we obtain
    \begin{align*}
        \E[P(X)e^{-X^TX/2}]&=\det(K)^{-1/2}P(\partial_u)|_{u=0}e^{u^T\Omega(I+\Omega)^{-1}u/2}
        \\&= \det(K)^{-1/2}\E[P(Y)],
    \end{align*}
    where the last equality follows from the moment generating function of the Gaussian $\E[e^{u^TY}]=e^{u^T\Omega(I+\Omega)^{-1}u/2}$. Similarly,
    \[
        \E[P(Y)e^{Y^TY/2}]=\det(K)^{1/2}P(\partial_u)|_{u=0}e^{u^T\Omega u/2}=\det(K)^{1/2}\E[P(X)].
    \]
    These prove the desired identities because $\det(K)=\det(I+A^TA)=\det(I+AA^T)=\det(I+\Omega)$.
\end{proof}

Define an auxiliary function $r$ given by
\[
    r(w)=(1+w)e^{-w+w^2/2}.
\]
Note that in its expansion $r(w)=1+\sum_{d\ge3}a_dw^d$, the first and second order terms vanish.

\begin{lemma}\label{lem:3.6}
    Let $d=2N$ and let $\Gamma(t)$ be a complex symmetric $d\times d$ matrix whose entries are analytic near $t=0$, with $\Gamma(0)=0$. For $t$ sufficiently close to zero, choose complex matrices $L_t$ such that $L_0=0$, $L_tL_t^T=\Gamma(t)$, and $\norm{L_t}_\op\to0$ as $t\to0$. If $Z(t) = L_tg$ where $g\sim N(0,I_d)$, then the function $F(t)$ given by
    \[
        F(t) = \E\left[\prod_{a=1}^dr(Z_a(t))\right]
    \]
    is analytic near zero and independent of the choice of factorization.
\end{lemma}

\begin{proof}
    Let $g\sim N(0,I_d)$, and let $Z(t) = L_tg$. Since $\norm{L_t}_\op\to0$ as $t\to0$, we may assume $\norm{L_t}_\op<1$ for $|t|$ sufficiently small. For $u\in\C^d$, define
    \[
        H(t,u)=\E\left[\exp\left(u^TZ(t)+\frac12Z(t)^TZ(t)\right)\right].
    \]
    Since $I-L_t^TL_t$ has positive definite real part, \eqref{eq:2} implies
    \begin{align*}
        H(t,u) &= \frac1{(2\pi)^{d/2}}\int_{\R^d}\exp\left(-\frac12x^T(I-L_t^TL_t)x+u^TL_tx\right)\mathrm dx \\
        &= \det(I-L_t^TL_t)^{-1/2}\exp\left(\frac12u^TL_t(I-L_t^TL_t)^{-1}L_t^Tu\right) \\
        &= \det(I-\Gamma(t))^{-1/2}\exp\left(\frac12u^T\Gamma(t)(I-\Gamma(t))^{-1}u\right),
    \end{align*}
    where the last equality follows from $L_t(I-L_t^TL_t)^{-1}L_t^T=\Gamma(t)(I-\Gamma(t))^{-1}$, and the square root is taken in the analytic branch that is equal to $1$ at $t=0$.

    Now recall that $r(w) = (1+w)e^{-w+w^2/2}$. By differentiating under the expectation, we have
    \[
        F(t)=\E\left[\prod_{a=1}^dr(Z_a(t))\right]=\left.\prod_{a=1}^d(1+\partial_{u_a})H(t,u)\right|_{u=-\one_d}.
    \]
    Since $\Gamma(t)$ is analytic and $\Gamma(0)=0$, we know $H(t,u)$ is analytic in $(t,u)$ near $t=0$, and hence $F(t)$ is analytic near $t=0$. Moreover, we see that $F$ is given by $H$, which in turn is given by $\Gamma(t)=L_tL_t^T$. It follows that $F(t)$ is independent of the choice of $L_t$, as desired.
\end{proof}

\begin{lemma}\label{lem:3.7}
    Assume the same hypotheses as in Lemma \ref{lem:3.6}. Suppose in addition that for some constants $A$ and $\theta\in[0,1)$,
    \begin{equation}\label{eq:3}
        \max_{1\le i,j\le2N}|[t^m]\Gamma_{ij}(t)|\le\frac AN\theta^{m-1}
    \end{equation}
    holds for every $m\ge1$. Then we can write $F(t)=1+\sum_{k\ge1}f_kt^k$, and there exists a constant $C<\infty$ depending only on $A$, such that for every $k\ge1$,
    \begin{equation}\label{eq:4}
        |f_k|\le\sum_{E=1}^k\sum_{1\le v\le\lfloor 2E/3\rfloor}C^E\left(\frac EN\right)^{E-v}\binom{k-1}{E-1}\theta^{k-E}.
    \end{equation}
\end{lemma}

\begin{proof}
    We begin by justifying the expansion $F(t)=1+\sum_{k\ge1}f_kt^k$. Define the function $F_\lambda(t)=\E\left[\prod_{a=1}^{2N}r(\lambda Z_a(t))\right]$. Since $\norm{L_t}_\op\to0$, we may assume $\norm{L_t}_\op<1/4$ for $|t|$ sufficiently small. Applying the preceding argument with $L_t$ replaced by $\lambda L_t$ shows that $F_\lambda(t)$ is jointly analytic in $(\lambda,t)$ for $|\lambda|<2$ and $|t|$ sufficiently small. Consequently, its Taylor series in $\lambda$ converges locally uniformly in $t$ at $\lambda=1$.

    Now let $r(w) = 1+\sum_{d\ge3}a_dw^d$. By a straightforward expansion,
    \begin{equation}\label{eq:5}
        \prod_{a=1}^{2N}r(\xi_a)=\sum_{S\subseteq[2N]}\sum_{\substack{d_i\ge3\\i\in S}}\left(\prod_{i\in S}a_{d_i}\right)\prod_{i\in S}\xi_i^{d_i}.
    \end{equation}
    For each fixed total degree $m$, differentiating $F_\lambda(t)$ at $\lambda=0$ gives
    \[
        [\lambda^m]F_\lambda(t)=\sum_{S\subseteq[2N]}\sum_{\substack{d_i\ge 3~(i\in S)\\\sum_{i\in S}d_i=m}}\left(\prod_{i\in S}a_{d_i}\right)\E\left[\prod_{i\in S}Z_i(t)^{d_i}\right].
    \]
    By Wick's formula, the odd moments vanish, and a moment of total degree $2E$ is a polynomial of degree $E$ in the entries of $\Gamma(t)$, and hence is $O(t^E)$ since $\Gamma(0)=0$. The local uniform convergence of the expansion $F(t)=1+\sum_{k\ge1}f_kt^k$ established above therefore implies
    \[
        f_k=\sum_{E=1}^k\sum_{S\subseteq[2N]}\sum_{\substack{d_i\ge3~(i\in S)\\\sum_{i\in S}d_i=2E}}\left(\prod_{i\in S}a_{d_i}\right)[t^k]\E\left[\prod_{i\in S}Z_i(t)^{d_i}\right].
    \]
    For each monomial $\prod_{i\in S}Z_i(t)^{d_i}$ appearing in an expectation, let $v = |S|$, and list its $2E$ factors as $Z_{j_1}(t),\dots, Z_{j_{2E}}(t)$, with repeated indices allowed. Since $d_i \ge 3$ for every $i$, we have $3v\le 2E$. As $\E[Z_i(t)Z_j(t)]=\Gamma_{ij}(t)$, Wick's formula gives
    \[
        \E\left[\prod_{i\in S}Z_i(t)^{d_i}\right]=\sum_{\pi\in\mathcal P_2([2E])}\prod_{\{a,b\}\in\pi}\Gamma_{j_aj_b}(t).
    \]
    Thus, each pairing contributes a product of $E$ entries of $\Gamma(t)$. In the following steps, we first estimate the coefficient of $t^k$ in one such product, then estimate the product of the $a_i$'s, and finally combine these estimates to bound $f_k$.

    First consider the product $\prod_{\ell=1}^E\Gamma_{p_\ell q_\ell}(t)$. Since $\Gamma(0)=0$, every factor contributes a positive power of $t$. Hence
    \[
        [t^k]\prod_{\ell=1}^E\Gamma_{p_\ell q_\ell}(t)=\sum_{\substack{m_1,\dots, m_E\ge1\\m_1+\dots+m_E=k}}\prod_{\ell=1}^E[t^{m_\ell}]\Gamma_{p_\ell q_\ell}(t).
    \]
    In particular, this coefficient vanishes unless $k\ge E$. By \eqref{eq:3}, 
    \begin{align}\label{eq:6}
        \left|[t^k]\prod_{\ell=1}^E\Gamma_{p_\ell q_\ell}(t)\right|&\le\sum_{\substack{m_1,\dots, m_E\ge1\\ m_1+\dots+m_E=k}}\left(\frac AN\right)^E\theta^{\sum_{\ell=1}^E(m_\ell-1)}\nonumber\\
        &=\binom{k-1}{E-1}\left(\frac AN\right)^E\theta^{k-E}.
    \end{align}

    Next, we estimate the sum of the product of the $|a_i|$'s. To this end, define $G(z)=\sum_{d\ge3}|a_d|z^d$. For a fixed $v$ and $E$,
    \[
        [z^{2E}]G(z)^v = \sum_{\substack{d_1,\dots, d_v\ge3\\d_1+\dots+d_v=2E}}|a_{d_1}\cdots a_{d_v}|.
    \]
    For a power series $g(w) = \sum_{d\ge 0}g_dw^d$, we will define $\norm g_R=\sum_{d\ge0}|g_d|R^d$. Then $\norm{gh}_R\le \norm g_R\cdot\norm h_R$ holds, and therefore
    \begin{align*}
        G(R) &\le \sum_{d\ge0}|[w^d]r(w)|R^d\\
        &\le\norm{1+w}_R\cdot\norm{e^{-w}}_R\cdot\norm{e^{w^2/2}}_R\\
        &= (1+R)\left(\sum_{k\ge0}\frac{R^k}{k!}\right)\left(\sum_{k\ge0}\frac{R^{2k}}{2^kk!}\right)\\
        &= (1+R)e^Re^{R^2/2}.
    \end{align*}
    Since $R\ge 1$, we have $\log(1+R)+R+\frac{R^2}2\le3R^2$. Hence
    \[
        G(R)\le(1+R)e^Re^{R^2/2}\le e^{3R^2}.
    \]
    Since $G$ has nonnegative coefficients, $[z^{2E}]G(z)^v\le G(R)^v/R^{2E}$. Taking $R = \sqrt{E/v}\ge\sqrt{3/2}$ and using the previous two inequalities, it follows that $[z^{2E}]G(z)^v\le e^{3E}(v/E)^E$, which implies
    \begin{align}
        \binom{2N}v(2E-1)!![z^{2E}]G(z)^v &\le\frac{(2N)^v}{v!}(2E)^Ee^{3E}\left(\frac vE\right)^E\nonumber\\
        &\le C_1^E N^v v^{E-v}\le C_1^EN^v E^{E-v},\label{eq:7}
    \end{align}
    where we used basic inequalities for factorials and the fact that $3v\le 2E$. Here $C_1$ is an absolute constant.

    To conclude, we combine the previous bounds in order to prove the desired bound on $f_k$. Fix some $k\ge 1$. For $v\ge 1$ and $E\ge 1$, define $L_{v,E}$ to be
    \[
        L_{v,E} = \sum_{\substack{S\subseteq[2N]\\|S|=v}}\sum_{\substack{d_i\ge3~(i\in S)\\\sum_{i\in S}d_i=2E}}\left(\prod_{i\in S}|a_{d_i}|\right)\left|[t^k]\E\left[\prod_{i\in S}Z_i(t)^{d_i}\right]\right|.
    \]
    By the termwise expansion \eqref{eq:5} and the vanishing of the coefficient when $E>k$,
    \begin{equation}\label{eq:8}
        |f_k|\le\sum_{E=1}^k\sum_{1\le v\le\min\{2N,\lfloor 2E/3\rfloor\}}L_{v,E}.
    \end{equation}
    Using \eqref{eq:6} for each pairing and summing over the degree choices, we have
    \begin{align*}
        L_{v,E}&\le\binom{2N}v(2E-1)!![z^{2E}]G(z)^v\binom{k-1}{E-1}\left(\frac AN\right)^E\theta^{k-E}\\
        &\le (C_1A)^E\left(\frac EN\right)^{E-v}\binom{k-1}{E-1}\theta^{k-E},
    \end{align*}
    where the last inequality follows from \eqref{eq:7}. Summing this bound in \eqref{eq:8} and enlarging the range of $v$ to $1\le v\le\lfloor 2E/3\rfloor$ proves \eqref{eq:4} with the constant $C = C_1A$.
\end{proof}

We may now prove the following asymptotics for the rook polynomial.

\begin{proposition}\label{prop:rook function asymptotic}
    Assume the hypotheses of Proposition \ref{prop:permanent bound} and $\rho>0$. Let $I = [a,b]\subset(0,\infty)$ be a compact interval containing $1$ in its interior and satisfying $b<\rho^{-1}$. Then for every $s\in I$, we have $R_B(s)=(1+o(1))\cdot\det(I-s^2B^TB)^{-1/2}$.
\end{proposition}

\begin{proof}
    Let $F_s(t) = R_{tB}(s)\det(I-s^2t^2B^TB)^{1/2}$, so that it suffices to show $F_s(1)=1+o(1)$ holds uniformly for $s\in I$. The first step of the proof involves expressing $F_s(t)$ as an expectation involving a Gaussian vector, whose covariance can be computed explicitly. This reformulation will allow us to estimate the coefficients of $F_s(t)$ using Wick's formula. For $s\in I$, define the matrix
    \[
        \Omega_{s,t}=\begin{bmatrix}0&stB\\stB^T&0\end{bmatrix}.
    \]
    Fix some $t\in\C$ satisfying $|st|\norm B_\op<1$. As in Lemma \ref{lem:3.5}, we diagonalize $\Omega_{s,t}=U\Sigma U^T$ and set $A_{s,t}=U\Sigma^{1/2}$. Let $K_{s,t}=I+A_{s,t}^TA_{s,t}$, and for $g\sim N(0,I_{2N})$, define $X_{s,t}=A_{s,t}g$ and $Y_{s,t}=A_{s,t}K_{s,t}^{-1/2}g$. Then $\E[X_{s,t}X_{s,t}^T]=\Omega_{s,t}$ and $\E[Y_{s,t}Y_{s,t}^T]=\Omega_{s,t}(I+\Omega_{s,t})^{-1}$. By the definition of $\Omega_{s,t}$ and Wick's formula,
    \[
        \E\left[\prod_{a=1}^{2N}(1+X_{s,t,a})\right]=\sum_{k=0}^N(st)^k\sum_{\substack{I_0,J_0\subseteq[N]\\|I_0|=|J_0|=k}}\per B[I_0,J_0] = R_{tB}(s).
    \]
    Since $B\one_N=B^T\one_N=0$, we have $\Omega_{s,t}\one_{2N}=0$. By construction, $\ker A_{s,t}^T=\ker\Omega_{s,t}$, so $\one_{2N}^TY_{s,t}=0$ almost surely. For the polynomial $P(\xi)=\prod_{a=1}^{2N}(1+\xi_a)$, we have
    \[
        P(Y_{s,t})e^{Y_{s,t}^TY_{s,t}/2}=\prod_{a=1}^{2N}r((Y_{s,t})_a).
    \]
    The second identity in Lemma \ref{lem:3.5} therefore yields
    \[
        R_{tB}(s)=\det(I+\Omega_{s,t})^{-1/2}\E\left[\prod_{a=1}^{2N}r((Y_{s,t})_a)\right].
    \]
    Using Schur complement, one can now show that $\det(I+\Omega_{s,t})=\det(I-s^2t^2B^TB)$. Combining this with the above, we have
    \[
        F_s(t)=\E\left[\prod_{a=1}^{2N}r((Y_{s,t})_a)\right].
    \]
    Finally, setting $q = \norm{\Omega_{s,t}}_\op=|st|\norm B_\op$, observe that expanding $K_{s,t}^{-1/2}$ using binomial series implies that as $t\to0$,
    \begin{align*}
        \norm{A_{s,t}K_{s,t}^{-1/2}}_\op^2&=\left\|A_{s,t}\sum_{m=0}^\infty\binom{-1/2}m(A_{s,t}^TA_{s,t})^m\right\|_\op^2\\
        &\le q\left(\sum_{m=0}^\infty\left|\binom{-1/2}m\right|q^m\right)^2=\frac q{1-q}\longrightarrow0.
    \end{align*}
    Thus $A_{s,t}K_{s,t}^{-1/2}g$ has precisely the form needed to apply Lemma \ref{lem:3.6}, which will be used in the next step.

    The next step will involve estimating the coefficients in the expansion of the entries of $\widetilde\Omega_{s,t}=\Omega_{s,t}(I+\Omega_{s,t})^{-1}$. Note that $\widetilde\Omega_{s,t}$ can be explicitly computed by block matrix inversion. Indeed, if $H = stB$, then we have
    \[
        \widetilde\Omega_{s,t}=\begin{bmatrix}-HH^T(I-HH^T)^{-1}&H(I-H^TH)^{-1}\\H^T(I-HH^T)^{-1}&-H^TH(I-H^TH)^{-1}\end{bmatrix}.
    \]
    Therefore, the top left block of $\widetilde\Omega_{s,t}$ is $-\sum_{q\ge1}s^{2q}t^{2q}(BB^T)^q$, the top right block is $\sum_{q\ge0}s^{2q+1}t^{2q+1}B(B^TB)^q$, the bottom left block is $\sum_{q\ge0}s^{2q+1}t^{2q+1}B^T(BB^T)^q$, and the bottom right block is $-\sum_{q\ge1}s^{2q}t^{2q}(B^TB)^q$. 
    
    The assumption $\norm B_\mx\le A_0/N$ implies 
    \[
        \max_{1\le i\le N}\norm{e_i^TB}_2\le \frac{A_0}{\sqrt N},\quad \max_{1\le j\le N}\norm{Be_j}_2\le \frac{A_0}{\sqrt N}.
    \]
    For any integer $q\ge1$, this together with the assumption $\norm B_\op\le \rho$ implies that
    \[
        |e_i^T(BB^T)^qe_j|=|(B^Te_i)^T(B^TB)^{q-1}(B^Te_j)|\le\frac{A_0^2}N\rho^{2q-2}.
    \]
    Therefore, we have
    \[
        \max_{1\le i,j\le N}|[t^{2q}](\widetilde \Omega_{s,t})_{ij}|\le s^{2q}\frac{A_0^2}{N}\rho^{2q-2}\le\frac{A_0^2\rho^{-2}}{N}(b\rho)^{2q},
    \]
    where $[t^{2q}](\widetilde\Omega_{s,t})_{ij}$ means to extract the coefficient of $t^{2q}$ in the expansion of $(\widetilde\Omega_{s,t})_{ij}$. Running the same argument for the other blocks in $\widetilde\Omega_{s,t}$, we obtain
    \begin{equation}\label{eq:9}
        \max_{1\le i,j\le 2N}|[t^m](\widetilde\Omega_{s,t})_{ij}|\le\frac CN(s\rho)^{m-1},
    \end{equation}
    where $m\ge1$ and $C$ is some constant depending on $A_0$ and $\rho$.

    Let $F_s(t)=1+\sum_{k\ge1}f_k(s)t^k$ and recall that we want to prove $F_s(1)=1+o(1)$ uniformly for $s\in I$. In this step we estimate $f_k(s)$ for small $k$. Setting $\theta=b\rho<1$, as $s\le b$ it follows from \eqref{eq:9} that
    \[
        \max_{1\le i,j\le 2N}|[t^m](\Omega_{s,t}(I+\Omega_{s,t})^{-1})_{ij}|\le\frac CN\theta^{m-1}
    \]
    for every $m\ge 1$, uniformly for $s\in I$. In the forthcoming argument, the constant $C$ will always depend only on $A_0$, $\rho$, and $I$, but may change from line to line.

    Apply Lemma \ref{lem:3.7} with $\Gamma(t)=\Omega_{s,t}(I+\Omega_{s,t})^{-1}$ and $Z(t)=A_{s,t}K_{s,t}^{-1/2}g$. Recall that the required condition $\norm{A_{s,t}K_{s,t}^{-1/2}}_\op\to0$ was verified in the beginning of the proof. By the lemma,
    \[
        |f_k(s)|\le\sum_{E=1}^k\sum_{1\le v\le\lfloor 2E/3\rfloor}C^E\left(\frac EN\right)^{E-v}\binom{k-1}{E-1}\theta^{k-E}.
    \]
    Fix a positive constant $K$ whose value will be determined later. Summing this upper bound on $|f_k(s)|$ over $1\le k\le K\sqrt N$ and changing the order of summation, we obtain
    \begin{align*}
        \sum_{1\le k\le K\sqrt N}|f_k(s)| &\le \sum_{E\le K\sqrt N}\sum_{1\le v\le\lfloor 2E/3\rfloor}C^E\left(\frac EN\right)^{E-v}\sum_{k=E}^{K\sqrt N}\binom{k-1}{E-1}\theta^{k-E}\\
        &\le \sum_{E\le K\sqrt N}\sum_{1\le v\le\lfloor 2E/3\rfloor}C^E\left(\frac EN\right)^{E-v}\sum_{k=E}^\infty\binom{k-1}{E-1}\theta^{k-E}\\
        &= \sum_{E\le K\sqrt  N}\sum_{1\le v\le \lfloor 2E/3\rfloor}C^E\left(\frac EN\right)^{E-v}(1-\theta)^{-E}.
    \end{align*}
    Since $v\le 2E/3$, we have $E-v\ge E/3$. Therefore, after enlarging $C$, we have
    \[
        C^E(1-\theta)^{-E}\left(\frac EN\right)^{E-v}\le\left(\frac{CE}{N}\right)^{E-v}.
    \]
    It follows that,
    \begin{align*}
        \sum_{1\le k\le K\sqrt N}|f_k(s)| &\le \sum_{E\le K\sqrt N}\sum_{1\le v\le \lfloor 2E/3\rfloor}\left(\frac{CE}N\right)^{E-v} \\
        &\le \sum_{E\le K\sqrt N}\sum_{j\ge \lceil E/3\rceil}\left(\frac{CE}N\right)^j\\
        &=\sum_{E\le K\sqrt N}\frac{(CE/N)^{\lceil E/3\rceil}}{1-CE/N} \\
        &\le \sum_{E\le K\sqrt N}\frac{(CK/\sqrt N)^{\lceil E/3\rceil}}{1-CK/\sqrt N}\\
        &\le\frac1{1-CK/\sqrt N}\sum_{E\ge1}\left(\frac{CK}{\sqrt N}\right)^{E/3}\\
        &=\frac1{1-CK/\sqrt N}\frac{(CK/\sqrt N)^{1/3}}{1-(CK/\sqrt N)^{1/3}}=o(1).
    \end{align*}

    It remains to estimate those coefficients $f_k(s)$ with $k>K\sqrt N$. Choose $R>1$ such that $R\theta<1$. Since $F_s(t)=R_{tB}(s)\det(I-s^2t^2B^TB)^{1/2}$ and $sR\norm B_\op\le R\theta<1$, $F_s$ is analytic in a neighborhood of the closed disk $|t|\le R$ for every $s\in I$. Throughout what follows, we will restrict $|t|\le R$.

    Let $z = (z_1,\dots, z_r)^T$ be a vector of independent standard complex Gaussian random variables, and let $X,Y$ be defined from $z$ as in the proof of Lemma \ref{lem:3.3}. Then for every complex $t$,
    \begin{equation}\label{eq:10}
        R_{tB}(s)=R_B(st)=\E\left[\prod_{i=1}^N(1+\sqrt{st}X_i)\prod_{j=1}^N(1+\sqrt{st}Y_j)\right].
    \end{equation}
    Since $B\one_N=B^T\one_N=0$, the construction in Lemma \ref{lem:3.3} gives $\sum_{i=1}^NX_i=\sum_{j=1}^NY_j=0$. Using this and AM-GM, we have
    \[
        \prod_{i=1}^N|1+\sqrt{st}X_i|\le\left(\frac1N\sum_{i=1}^N|1+\sqrt{st}X_i|^2\right)^{N/2}=\left(1+\frac{s|t|}N\norm X_2^2\right)^{N/2},
    \]
    and the analogous estimate also holds for $Y_i$. Denote $q = z^*\Sigma z=\norm X_2^2=\norm Y_2^2$. It follows from \eqref{eq:10} that for all $s\in I$ and $|t|\le R$,
    \[
        |R_{tB}(s)|\le\E\left[\left(1+\frac{sRq}N\right)^N\right]\le\E[e^{sRq}]=\prod_\ell(1-sR\sigma_\ell)^{-1},
    \]
    where $\sigma_\ell$ are the singular values of $B$. Note that we used $\E[e^{az_i^*z_i}]=(1-a)^{-1}$ for $\mathrm{Re}(a)<1$, which is justified by $sR\sigma_\ell\le bR\rho<1$. By the elementary inequality $\log(\frac1{1-x})\le\frac x{1-x}$ for $0\le x<1$, we have
    \[
        -\log(1-sR\sigma_\ell)\le-\log(1-bR\sigma_\ell)\le C\sigma_\ell,
    \]
    where $C = bR/(1-bR\rho)$. Therefore,
    \[
        \log\prod_\ell(1-sR\sigma_\ell)^{-1}\le C\sum_\ell\sigma_\ell\le C\sqrt N\left(\sum_\ell\sigma_\ell^2\right)^{1/2}\le C\sqrt N\norm B_F\le C\sqrt N,
    \]
    where we used Cauchy-Schwarz and $\norm B_F^2\le N^2\norm B_\mx^2\le A_0^2$. We conclude that for every $s\in I$, we have $\max_{|t|=R}|R_{tB}(s)|\le e^{C\sqrt N}$ and
    \begin{align*}
        |\det(I-s^2t^2B^TB)^{1/2}|&=\prod_\ell|1-s^2t^2\sigma_\ell^2|^{1/2}\le\prod_\ell(1+s^2R^2\sigma_\ell^2)^{1/2}
        \\&\le\exp\left(\frac{s^2R^2}2\sum_\ell\sigma_\ell^2\right)=\exp\left(\frac{s^2R^2}2\norm B_F^2\right)\le e^C.
    \end{align*}
    Together, these show that $\max_{|t|\le R}|F_s(t)|\le e^{C\sqrt N+C}$. Now by Cauchy's integral formula, we have
    \[
        |f_k(s)|\le\frac1{2\pi}\int_{|t|=R}\frac{|F_s(t)|}{|t|^{k+1}}|\mathrm dt|\le\frac1{2\pi}\cdot2\pi R\cdot\frac{e^{C\sqrt N+C}}{R^{k+1}}=e^{C\sqrt N+C}R^{-k}.
    \]
    We now choose the constant $K$ sufficiently large so that $K\log R>C$. With this value of $K$,
    \[
        \sum_{k>K\sqrt N}|f_k(s)|\le\frac{e^{C\sqrt N+C}R^{-K\sqrt N}}{1-R^{-1}}=o(1).
    \]

    Aggregating the contributions from $k \le K\sqrt N$ and $k > K\sqrt N$, we have that $\sum_{k\ge 1}|f_k(s)|=o(1)$. In particular, $F_s(1)=1+o(1)$, as desired.
\end{proof}
\medskip

\subsection{Proof of Proposition~\ref{prop:permanent bound}}\label{sec:3.3}

The case when $\rho=0$ is immediate, since then $B = 0$. We may therefore assume that $0 < \rho < 1$. Choose constants $0 < a < 1 < b < \rho^{-1}$ and let $I=[a,b]$. Let $T\sim\Gamma(N+1,1)$, and define
\[
    \mathcal A_N:=\left\{\frac NT\in I\right\}=\left\{\frac Nb\le T\le \frac Na\right\}.
\]
Then $T/N\to1$ in probability and we have that $\mathbb P(\mathcal A_N)=1-o(1)$. For $s\in I$, denote $\Delta(s) = \det(I-s^2B^TB)^{-1/2}$. Let $\sigma_1,\dots, \sigma_r$ be the nonzero singular values of $B$. Since $\sigma_\ell\le\rho$ and $\sum_\ell\sigma_\ell^2=\norm B_F^2\le A_0^2$, by the elementary inequality $\log(\frac1{1-x})\le\frac x{1-x}$ for $0\le x<1$, we have
\begin{align*}
    0\le\log\Delta(s)&=-\frac12\sum_{\ell=1}^r\log(1-s^2\sigma_\ell^2)\le\frac12\sum_{\ell=1}^r\log\left(\frac1{1-b^2\sigma_\ell^2}\right)\\
    &\le\frac{b^2}{2(1-b^2\rho^2)}\sum_{\ell=1}^r\sigma_\ell^2\le C.
\end{align*}
Moreover,
\[
    \frac{\Delta'(s)}{\Delta(s)}=s\sum_{\ell=1}^r\frac{\sigma_\ell^2}{1-s^2\sigma_\ell^2},
\]
and hence
\[
    \sup_{s\in I}|\Delta'(s)|\le\sup_{s\in I}|\Delta(s)|\frac{bA_0^2}{1-b^2\rho^2}\le C.
\]
It follows that
\begin{align*}
    |\E[\one_{\mathcal A_N}\Delta(N/T)]-\Delta(1)| &\le |\E[\one_{\mathcal A_N}(\Delta(N/T)-\Delta(1))]|+\Delta(1)\mathbb P(\mathcal A_N^c)\\
    &\le C\E[\one_{\mathcal A_N}|N/T-1|] + o(1)=o(1),
\end{align*}
where we used $T/N\to1$ in probability. Thus by Proposition \ref{prop:rook function asymptotic}, we have
\begin{equation}\label{eq:11}
    \E[\one_{\mathcal A_N}R_B(N/T)]=\Delta(1)+o(1).
\end{equation}
We now claim $\E[\one_{\mathcal A_N^c}R_B(N/T)]=o(1)$. Again let $X, Y$ and $z$ be as in Lemma \ref{lem:3.3}, so that $z^*\Sigma z=\norm X_2^2=\norm Y_2^2$. As before, $B\one_N=B^T\one_N=0$ implies $\sum_{i=1}^NX_i=\sum_{j=1}^NY_j=0$. For every $u\ge0$, by AM-GM and Lemma \ref{lem:3.3}, we have
\begin{align*}
    |R_B(u)| &\le \E\left[\prod_{i=1}^N|1+\sqrt uX_i|\prod_{j=1}^N|1+\sqrt uY_j|\right]\\
    &\le\E\left[\left(\frac 1N\sum_{i=1}^N|1+\sqrt uX_i|^2\right)^{N/2}\left(\frac1N\sum_{j=1}^N|1+\sqrt uY_j|^2\right)^{N/2}\right]\\
    &=\E\left[\left(1+\frac{uz^*\Sigma z}{N}\right)^{N/2}\left(1+\frac{uz^*\Sigma z}{N}\right)^{N/2}\right]\\
    &=\E\left[\left(1+\frac{uz^*\Sigma z}{N}\right)^N\right].
\end{align*}
Hence,
\begin{align}\label{eq:12}
    \E_T[\one_{\mathcal A_N^c}|R_B(N/T)|]&\le\mathbb E_z\left[\frac1{N!}\int_{\{t>0:N/t\notin I\}}\left(1+\frac{z^*\Sigma z}{t}\right)^Nt^Ne^{-t}\mathrm dt\right]\nonumber\\
    &= \E_z\left[\frac1{N!}\int_{\{t>0:N/t\notin I\}}(t+z^*\Sigma z)^Ne^{-t}\mathrm dt\right].
\end{align}
Let $U$ be a random variable with distribution $\Gamma(N+1,1)$. For fixed $z$, a change of variable shows
\begin{align}\label{eq:13}
    \frac1{N!}\int_{\{t>0:N/t\notin I\}}&(t+z^*\Sigma z)^Ne^{-t}\mathrm dt\nonumber\\
    &=e^{z^*\Sigma z}\mathbb P\left(U-z^*\Sigma z\notin\left[\frac Nb,\frac Na\right],~U\ge z^*\Sigma z\right).
\end{align}

Choose a sufficiently small $\varepsilon$, and assume first that $z^*\Sigma z\le \varepsilon N$. Then on this event,
\[
    \left\{U-z^*\Sigma z\notin\left[\frac Nb,\frac Na\right]\right\}\subseteq\left\{U\le\left(\frac1b+\varepsilon\right)N\right\}\cup\left\{U\ge\frac Na\right\}.
\]
Since $a<1<b$, there exist $\delta_-,\delta_+>0$ such that $1/b+\varepsilon\le1-\delta_-$ and $1/a\ge1+\delta_+$ hold. For $0 < \theta<1$, using $\E[e^{\theta U}]=(1-\theta)^{-(N+1)}$ and the Chernoff bound, we have
\begin{align*}
    \mathbb P(U\ge(1+\delta_+)N) &\le e^{-\theta(1+\delta_+)N}\E[e^{\theta U}]\\
    &= e^{-\theta(1+\delta_+)N}(1-\theta)^{-(N+1)}\le e^{-c_+N}
\end{align*}
for a suitable $\theta>0$ and some $c_+>0$. Similarly,
\begin{align*}
    \mathbb P(U\le (1-\delta_-)N) &\le e^{\theta(1-\delta_-)N}(1+\theta)^{-(N+1)}\le e^{-c_-N}
\end{align*}
for some $c_->0$. Consequently,
\begin{equation}\label{eq:14}
    \mathbb P\left(U-z^*\Sigma z\notin\left[\frac Nb,\frac Na\right],~U\ge z^*\Sigma z\right)\le e^{-c_-N}+e^{-c_+N} \le e^{-cN}
\end{equation}
for some smaller constant $c>0$.

For the event $z^*\Sigma z\ge\varepsilon N$, choose $\lambda>1$ such that $\lambda\rho<1$. We have
\[
    \E_z\left[e^{z^*\Sigma z}\one_{\{z^*\Sigma z>\varepsilon N\}}\right]\le e^{-(\lambda-1)\varepsilon N}\E_z[e^{\lambda z^*\Sigma z}] = e^{-(\lambda-1)\varepsilon N}\prod_{\ell=1}^r(1-\lambda\sigma_\ell)^{-1}.
\]
Since $\lambda\sigma_\ell\le\lambda\rho<1$, as before we have
\[
    \log\prod_{\ell=1}^r(1-\lambda\sigma_\ell)^{-1}\le C\sum_{\ell=1}^r\sigma_\ell\le C\sqrt N\norm B_F\le C\sqrt N.
\]
It follows that
\[
    \E_z\left[e^{z^*\Sigma z}\one_{\{z^*\Sigma z>\varepsilon N\}}\right]\le\exp\left(-(\lambda-1)\varepsilon N+C\sqrt N\right).
\]
Since $\lambda>1$ and $z^*\Sigma z\ge0$, the same exponential moment estimate also gives $E_z[e^{z^*\Sigma z}]\le\E_z[e^{\lambda z^*\Sigma z}]\le e^{C\sqrt N}$. Combining this with \eqref{eq:12}, \eqref{eq:13}, and \eqref{eq:14},
\begin{align*}
    \E_T[\one_{\mathcal A_N^c}|R_B(N/T)|] &\le \E_z\left[e^{z^*\Sigma z}\mathbb P\left(U-z^*\Sigma z\notin\left[\frac Nb,\frac Na\right],~U\ge z^*\Sigma z\right)\right]\\
    &\le e^{-cN}\E_z\left[e^{z^*\Sigma z}\one_{\{z^*\Sigma z\le\varepsilon N\}}\right]+\E_z\left[e^{z^*\Sigma z}\one_{\{z^*\Sigma z>\varepsilon N\}}\right]\\
    &\le e^{-cN+C\sqrt N}+e^{-(\lambda-1)\varepsilon N+C\sqrt N}=o(1).
\end{align*}
Together with Lemma~\ref{lem:3.4} and \eqref{eq:11}, we now have
\[
    \frac{N^N\per D}{N!}=\E[R_B(N/T)]=\det(I-B^TB)^{-1/2}+o(1).
\]
Finally, since $1\le\det(I-B^TB)^{-1/2}\le e^C$, we may rewrite the additive error as a multiplicative error, i.e.,
\[
    \frac{N^N\per D}{N!}=(1+o(1))\det(I-B^TB)^{-1/2},
\]
as desired.\qed

\section{Sinkhorn scaling}\label{sec:Sinkhorn}

Let us recall that by Sinkhorn's theorem, for any $n\times n$ matrix $A$ with positive entries, there exist vectors $x,y\in\mathbb R^n$ such that $\diag(e^x) A \diag (e^y)$ is a doubly stochastic matrix, where $\diag(e^x)$ is the diagonal matrix with diagonal $(e^{x_1},e^{x_2},\dots, e^{x_n})$, and $\diag(e^y)$ is defined similarly.

The goal of this section is to prove a quantitative version of the statement that, when applying Sinkhorn's theorem to a matrix which is already close to being doubly stochastic, the vectors $x,y$ can be chosen to be small. This will be useful in our applications for the following reason. The adjacency matrix of a regular tournament is always a scalar multiple of a doubly stochastic matrix, and we can input such a doubly stochastic matrix into Proposition \ref{prop:permanent bound} to extract information about its permanent. The tournaments we must consider in the proof of the upper bound (Theorem \ref{thm:upper}) need not be regular, but they can be shown to be ``close'' to regular, and in particular their adjacency matrix is close to being a scalar multiple of a doubly stochastic matrix. In order to extract information about the permanents of these matrices via Proposition \ref{prop:permanent bound}, we will apply Sinkhorn scaling to produce a doubly stochastic matrix, and will need to show that the scaling matrices $\diag(e^x)$ and $\diag(e^y)$ are close to the identity matrix. To this end, we need:

\begin{proposition}\label{prop:sinkhorn}
   Let $A_0 < \infty$ and $\rho < 1$. Let $Q$ be a nonnegative $N\times N$ matrix with $\norm Q_\mx \le A_0/N$. Let $u,v\in\mathbb R^N$ be given by $u = Q\one_N-\one_N$ and $v = Q^T\one_N - \one_N$, and let $P = \diag(Q\one _N)^{-1/2} Q\diag(Q^T\one_N)^{-1/2}$. Assume $\sigma_2(P) \le \rho+o(1)$ and
    \[
        \norm u_\infty + \norm v_\infty + \frac{\norm u_2+\norm v_2}{\sqrt N} = o(1).
    \]
    Then there exist vectors $x,y\in\mathbb R^N$ such that $\diag(e^x)Q\diag(e^y)$ is doubly stochastic, $\norm x_2 + \norm y_2 = O(\norm u_2 + \norm v_2)$, and $\norm x_\infty + \norm y_\infty=o(1)$.
\end{proposition}

The main ingredient in the proof of Proposition \ref{prop:sinkhorn} is an application of the Banach fixed-point theorem (also known as the contraction mapping theorem). We begin by recording some linear algebra facts in the following lemmas.

\begin{lemma}\label{lem:sinkhorn spectral 1}
    Let $Q$ be a nonnegative $N\times N$ matrix with positive row and column sums, and let $r = Q\one_N$ and $c = Q^T\one_N$. If $P = \diag(r)^{-1/2}Q\diag(c)^{-1/2}$, then $\norm P_\op=1$, $Pc^{1/2}=r^{1/2}$, and $P^Tr^{1/2}=c^{1/2}$, where square roots are taken coordinatewise.
\end{lemma}

\begin{proof}
    For $\alpha,\beta\in\mathbb R^N$, Cauchy-Schwarz gives
    \begin{align*}
        |\alpha^TP\beta|&=\left|\sum_{1\le i,j\le N} Q_{ij}\frac{\alpha_i}{\sqrt{r_i}}\frac{\beta_j}{\sqrt{c_j}}\right|\\&\le\left(\sum_{1\le i,j\le N}Q_{ij}\frac{\alpha_i^2}{r_i}\right)^{1/2}\left(\sum_{1\le i,j\le N}Q_{ij}\frac{\beta_j^2}{c_j}\right)^{1/2}=\norm\alpha_2\norm\beta_2.
    \end{align*}
    Thus $\norm P_\op\le1$. On the other hand, $P^TPc^{1/2}=P^Tr^{1/2}=c^{1/2}$, so $1$ is a singular value of $P$ and therefore $\norm P_\op=1$.
\end{proof}

\begin{lemma}\label{lem:sinkhorn spectral 2}
    Under the assumptions of Lemma \ref{lem:sinkhorn spectral 1}, suppose in addition that $\sigma_2(P)\le\rho<1$. If $\mathcal M = \left(\!\begin{smallmatrix}I&P\\P^T&I\end{smallmatrix}\right)$, then $\ker\mathcal M=\mathrm{span}\{(r^{1/2},-c^{1/2})\}$, and for all $\xi\in\mathbb R^{2N}$ satisfying $\xi\perp(r^{1/2},-c^{1/2})$ we have $\langle\xi,\mathcal M\xi\rangle\ge(1-\rho)\norm\xi_2^2$.
\end{lemma}

\begin{proof}
    Let $P= U\Sigma V^T$ be a singular value decomposition, where $\Sigma$ is the diagonal matrix $\diag(\sigma_1(P),\dots,\sigma_N(P))$. Consider the matrix $\widetilde{\mathcal M}$ given by
    \[
        \widetilde{\mathcal M} = \begin{pmatrix}U^T&0\\0&V^T\end{pmatrix}\mathcal M\begin{pmatrix}U&0\\0&V\end{pmatrix}=\begin{pmatrix}I&\Sigma\\\Sigma&I\end{pmatrix}.
    \]
    For each $1\le k\le N$, the subspace spanned by $(e_k,0)$ and $(0,e_k)$ is invariant under the matrix on the right, and the restriction of $\widetilde{\mathcal M}$ to this subspace is
    \[
        \begin{pmatrix}1&\sigma_k(P)\\\sigma_k(P)&1\end{pmatrix}.
    \]
    Thus its two eigenvalues are $1+\sigma_k(P)$ and $1-\sigma_k(P)$, and so the eigenvalues of $\mathcal M$ are $1+\sigma_k(P)$ and $1-\sigma_k(P)$ for $1\le k\le N$. 
    
    Since $\sigma_1(P)=1$ and $\sigma_2(P)\le\rho<1$, exactly one eigenvalue of $\mathcal M$ is zero. Thus the kernel of $\mathcal M$ is one-dimensional. The identities in Lemma \ref{lem:sinkhorn spectral 1} show that $(r^{1/2},-c^{1/2})$ belongs to this kernel. Since all other eigenvalues are at least $1-\rho$, it follows that $\langle\xi,\mathcal M\xi\rangle\ge(1-\rho)\norm\xi_2^2$, as desired.
\end{proof}

For the remainder of this section, we use an auxiliary norm $\norm{\cdot}_*$ defined for a pair of vectors $a,b$ as follows:
\[
    \norm{(a,b)}_* = \norm a_\infty+\norm b_\infty+\frac{\norm a_2+\norm b_2}{\sqrt N}.
\]

\begin{lemma}\label{lem:sinkhorn inverse}
    Under the assumptions of Lemma \ref{lem:sinkhorn spectral 2}, suppose in addition that $1/2\le r_i,c_j\le3/2$ and $\norm Q_\mx\le A_0/N$. If $f,h\in\mathbb R^N$ satisfy $\langle\one_N,f\rangle=\langle\one_N,h\rangle$, then there exists a unique pair $(a,b)$ satisfying
    \begin{equation}\label{eq:15}
        \begin{cases}
            \diag(r)a + Qb=f,\\
            Q^Ta+\diag(c)b=h,\\
            \langle r,a\rangle=\langle c,b\rangle.
        \end{cases}
    \end{equation}
    Moreover, there exists a constant $C = C(A_0,\rho)$ such that
    \begin{equation}\label{eq:16}
        \norm a_2+\norm b_2\le C(\norm f_2+\norm h_2)
    \end{equation}
    and
    \begin{equation}\label{eq:17}
        \norm{(a,b)}_*\le C\norm{(f,h)}_*.
    \end{equation}
\end{lemma}

\begin{proof}
    Let $\alpha=\diag(r)^{1/2}a$, $~\beta=\diag(c)^{1/2}b$, $~\phi=\diag(r)^{-1/2}f$, $~\psi=\diag(c)^{-1/2}h$, and $\kappa=(r^{1/2},-c^{1/2})$. With this notation, the first two equations in \eqref{eq:15} become $\alpha+P\beta=\phi$ and $P^T\alpha+\beta=\psi$. Since $\langle\one_N,f\rangle=\langle\one_N,h\rangle$ says $(\phi,\psi)\perp\kappa$, Lemma \ref{lem:sinkhorn spectral 2} shows that this system has a unique solution $(\alpha,\beta)\in\kappa^\perp$, which proves \eqref{eq:15}. Also by Lemma \ref{lem:sinkhorn spectral 2}, we have
    \[
        (\norm\alpha_2^2+\norm\beta_2^2)^{1/2}\le\frac1{1-\rho}(\norm\phi_2^2+\norm\psi_2^2)^{1/2}.
    \]
    Since all coordinates of $r$ and $c$ lie in the interval $[1/2,3/2]$, this proves \eqref{eq:16}. For the last estimate, note that $0\le P_{ij}\le2A_0/N$, and hence $\alpha+P\beta=\phi$ gives
    \[
        \norm\alpha_\infty\le\norm\phi_\infty+\frac{2A_0}N\norm\beta_1\le\norm\phi_\infty+\frac{2A_0}{\sqrt N}\norm\beta_2,
    \]
    and similarly,
    \[
        \norm\beta_\infty\le\norm\psi_\infty+\frac{2A_0}{\sqrt N}\norm\alpha_2.
    \]
    Combining this with \eqref{eq:16} and again using $r_i,c_j\in[1/2,3/2]$ now gives \eqref{eq:17}.
\end{proof}

We are now ready to prove Proposition~\ref{prop:sinkhorn}.

\begin{proof}[Proof of Proposition \ref{prop:sinkhorn}]
    Let $r = Q\one_N = \one_N+u$ and $c = Q^T\one_N=\one_N+v$, and let $\eta=\norm{(u,v)}_*$. By assumption, $\eta=o(1)$, so for all sufficiently large $N$ every coordinate of $r$ and $c$ lies in the interval $[1/2,3/2]$. Since $\sigma_2(P)\le\rho+o(1)$, we have $\sigma_2(P)\le\frac{1+\rho}2$ for all sufficiently large $N$. 

    We begin by defining, for $a,b\in\mathbb R^N$ the vector $\mathcal R(a,b)=(R_1(a,b),R_2(a,b))\in\mathbb R^{2N}$, where
    \[
        R_1(a,b)_i = \sum_{j=1}^N Q_{ij}(e^{a_i+b_j}-1-a_i-b_j)
    \]
    and
    \[
        R_2(a,b)_j=\sum_{i=1}^N Q_{ij}(e^{a_i+b_j}-1-a_i-b_j).
    \]
    Note that the sums of the coordinates of $R_1(a,b)$ and $R_2(a,b)$ are equal. We will record some estimates on $\mathcal R$ which will be used later. Let $\varphi(t) = e^t-1-t$ and $m =\norm a_\infty+\norm b_\infty$. For $m\le 1$, Taylor expansion of $\varphi(t)$ near $t = 0$ gives $|\varphi(a_i+b_j)|\le Cm(|a_i|+|b_j|)$ for some absolute constant $C$. It follows that $|R_1(a,b)_i|\le Cm(\diag(r)|a|+Q|b|)_i$ and $|R_2(a,b)_j|\le Cm(Q^T|a|+\diag(c)|b|)_j$. Since every row and column sum of $Q$ is at most $3/2$, Schur's test gives $\norm Q_\op\le3/2$. Therefore, for some absolute constant $C_R$,
    \begin{equation}\label{eq:18}
        \norm{R_1(a,b)}_2+\norm{R_2(a,b)}_2\le C_R(\norm a_\infty+\norm b_\infty)(\norm a_2+\norm b_2).
    \end{equation}
    Using $|\varphi(a_i+b_j)|\le Cm^2$, we have $\norm{R_1(a,b)}_\infty+\norm{R_2(a,b)}_\infty\le C_Rm^2$. Thus
    \begin{equation}\label{eq:19}
        \norm{\mathcal R(a,b)}_*\le C_Rm\norm{(a,b)}_*\le C_R\norm{(a,b)}_*^2
    \end{equation}
    whenever $\norm{(a,b)}_*\le 1$. Moreover, assuming both $\norm{(a,b)}_*$ and $\norm{(a',b')}_*$ are at most some $\delta\le1$, the mean value theorem gives $|\varphi(s)-\varphi(t)|\le C\delta|s-t|$ for all $|s|,|t|\le\delta$. So, repeating the preceding estimates yields
    \begin{equation}\label{eq:20}
        \norm{\mathcal R(a,b)-\mathcal R(a',b')}_*\le C_R\delta\norm{(a-a',b-b')}_*.
    \end{equation}

    The next step will involve constructing a contraction mapping in order to prove the existence of the desired $x,y$. Define $\mathcal D = \{(f,h)\in\mathbb R^N\times\mathbb R^N : \langle\one_N,f\rangle=\langle\one_N,h\rangle\}$ and $\mathcal H = \{(x,y)\in\mathbb R^N\times\mathbb R^N:\langle r,x\rangle=\langle c,y\rangle\}$. For $(f,h)\in\mathcal D$, let $\mathcal L(f,h)\in\mathcal H$ denote the unique pair given by Lemma \ref{lem:sinkhorn inverse}. The map $\mathcal L : \mathcal D\to\mathcal H$ is linear. Let $C_L$ be a constant for which both \eqref{eq:16} and \eqref{eq:17} hold.

    The matrix $\diag(e^x)Q\diag(e^y)$ is doubly stochastic if and only if $\diag(e^x)Qe^y=\one_N$ and $\diag(e^y)Q^Te^x=\one_N$. Since $r = \one_N+u$ and $c = \one_N+v$, we may rewrite these equations as
    \begin{equation}\label{eq:21}
        \begin{cases}
            \diag(r)x+Qy+u+R_1(x,y)=0,\\
            Q^Tx+\diag(c)y+v+R_2(x,y)=0.
        \end{cases}
    \end{equation}
    We are going to formulate these equations as a fixed point problem, then apply the contraction mapping theorem. Note that $\langle\one_N,u\rangle=\sum_{1\le i,j\le N}Q_{ij}-N=\langle\one_N,v\rangle$ and for any $x,y\in\mathbb R^N$, $\langle\one_N,R_1(x,y)\rangle=\langle\one_N,R_2(x,y)\rangle$. Thus the function $\Phi : \mathcal H\to\mathcal H$ given by
    \[
        \Phi(x,y)=-\mathcal L(u+R_1(x,y),v+R_2(x,y))
    \]
    is well defined. By the definition of $\mathcal L$, a pair $(x,y)\in\mathcal H$ is a fixed point of $\Phi$ if and only if it satisfies \eqref{eq:21}. Let $\mathcal B = \{(x,y)\in\mathcal H:\norm{(x,y)}_*\le 2C_L\eta\}$. Since $\eta=o(1)$, we have $2C_L\eta\le1$ for all sufficiently large $N$. If $(x,y)\in\mathcal B$, then Lemma \ref{lem:sinkhorn inverse} and \eqref{eq:19} imply
    \begin{align*}
        \norm{\Phi(x,y)}_* &\le C_L(\norm{(u,v)}_*+\norm{\mathcal R(x,y)}_*)\\
        &\le C_L\eta + C_LC_R\norm{(x,y)}_*^2\\
        &\le C_L\eta+C_LC_R(2C_L\eta)^2\\
        &= C_L\eta(1+4C_L^2C_R\eta).
    \end{align*}
    For all sufficiently large $N$, we have $C_L\eta(1+4C_L^2C_R\eta) \le 2C_L\eta$, and hence $\Phi(\mathcal B)\subseteq\mathcal B$.

    Now suppose $(x,y),(x',y')\in\mathcal B$. Since $\mathcal L$ is linear, we have by Lemma \ref{lem:sinkhorn inverse} that
    \begin{align*}
        \norm{\Phi(x,y)-\Phi(x',y')}_*&\le C_L\norm{\mathcal R(x,y)-\mathcal R(x',y')}_* \\&\le 2C_L^2C_R\eta\norm{(x-x',y-y')}_*,
    \end{align*}
    where the second inequality follows from the Lipschitz estimate in \eqref{eq:20} with $\delta=2C_L\eta$. Since $\eta=o(1)$, the Lipschitz constant $2C_L^2C_R\eta$ is at most $1/2$ for all sufficiently large $N$. 

    The set $\mathcal B$ is closed in $(\mathcal H,\norm \cdot_*)$ and is therefore complete. By the contraction mapping theorem, there is a unique fixed point $(x,y)\in\mathcal B$. Moreover, by \eqref{eq:21}, the matrix $\diag(e^x)Q\diag(e^y)$ is doubly stochastic and
    \[
        \norm x_\infty+\norm y_\infty\le\norm{(x,y)}_*\le 2C_L\eta=o(1).
    \]
    
    It remains to show that $\norm x_2+\norm y_2=O(\norm u_2+\norm v_2)$. By \eqref{eq:16} and then \eqref{eq:18},
    \begin{align*}
        \norm x_2+\norm y_2&\le C_L(\norm u_2+\norm v_2+\norm{R_1(x,y)}_2+\norm{R_2(x,y)}_2)\\
        &\le C_L(\norm u_2+\norm v_2)+C_LC_R(\norm x_\infty+\norm y_\infty)(\norm x_2+\norm y_2).
    \end{align*}
    Since $\norm x_\infty+\norm y_\infty=o(1)$, we have for sufficiently large $N$ that
    \[
        \norm x_2+\norm y_2\le \frac{C_L}{1-C_LC_R(\norm x_\infty+\norm y_\infty)}(\norm u_2+\norm v_2)=O(\norm u_2+\norm v_2),
    \]
    which completes the proof of Proposition \ref{prop:sinkhorn}.
\end{proof}

\section{Inclusion-exclusion}\label{sec:inclusion-exclusion}

The aim of this section is to prove the following.

\begin{proposition}\label{prop:cyc bound}
    Let $A_0 < \infty$ and $\rho < 1$. If $D = J_N + B$ is an $N\times N$ doubly stochastic matrix with $\norm D_\mx \le A_0/N$ and $\norm B_\op\le\rho$, then
    \[
        \frac{N^N\cyc D}{(N-1)!} = (1+o(1))\frac{\det(I-B)}{\sqrt{\det(I-B^TB)}}.
    \]
\end{proposition}

Let us first recall that if $M$ is the adjacency matrix of a tournament $T$, then $\per M$ counts the number of $1$-factors of $T$, and $\cyc M$ counts the number of Hamiltonian cycles of $T$. Since our goal is to count the number of Hamiltonian cycles, but the estimates of Section \ref{sec:permanent bounds} are stated in terms of permanents, it will be helpful to rewrite $\cyc M$ in terms of determinants and permanents. For this purpose, we will use the following result of Liu \cite{Liu1991}, which we reformulate in our setting for directed graphs. This identity, which may be viewed as a form of inclusion-exclusion for $\cyc M$ over the principal minors of $M$, has also been rederived in the setting of Hamiltonian cycle polynomials \cite{SawczukGnang2025}. We include its short proof for completeness.

\begin{lemma}[{\cite[Theorem 2]{Liu1991}}]\label{lem:inclusion-exclusion}
    For every $N\times N$ matrix $M$,
    \[
        N\cyc M = \sum_{\emptyset\neq U\subseteq [N]} |U|(-1)^{|U|-1}\det M[U]\per M[U^c].
    \]
\end{lemma}

\begin{proof}
    By expanding the expressions for $\det M[U]$ and $\per M[U^c]$ and collecting terms according to permutations $\sigma\in S_N$, the right hand side can be rewritten as
    \[
        \sum_{\sigma\in S_N}\Big(\sum_{\substack{\emptyset\neq U\subseteq[N]\\\sigma(U)=U}} |U|(-1)^{|U|-1}\mathrm{sgn}(\sigma|_U)\Big)\Big(\prod_{i=1}^N M_{i,\sigma(i)}\Big).
    \]
    Suppose that some fixed $\sigma\in S_N$ has a cycle decomposition of the form $C_1, \dots, C_k$ for some $k \ge 1$, and let $|C_i| = \ell_i$ for $i\in[k]$. The condition that $\sigma(U) = U$ implies that $U$ is a union of disjoint cycles of the permutation $\sigma$, and hence $U = \bigcup_{i\in S}C_i$ for some nonempty $S\subseteq[k]$. For such a choice of $U$, we have $\mathrm{sgn}(\sigma|_U) = (-1)^{|U|-|S|}$, so its contribution $|U|(-1)^{|U|-1}\mathrm{sgn}(\sigma|_U)$ to the coefficient of the monomial $\prod_{i=1}^NM_{i,\sigma(i)}$ is $(\sum_{i\in S}\ell_i)(-1)^{|S|-1}$. It follows that the coefficient of $\prod_{i=1}^NM_{i,\sigma(i)}$ is
    \[
        \sum_{\emptyset\neq S\subseteq[k]}\Big(\sum_{i\in S}\ell_i\Big)(-1)^{|S|-1} = \sum_{i=1}^k \ell_i\sum_{T\subseteq[k]\setminus\{i\}}(-1)^{|T|},
    \]
    which is $N$ when $k = 1$ (and hence $\sigma$ is counted by $\cyc M$) and zero when $k > 1$. 
\end{proof}

Given any $D$ as in Proposition \ref{prop:permanent bound}, the idea of the proof is now to apply Lemma \ref{lem:inclusion-exclusion} with $M = D$, and analyze the contributions from each of the principal minors of $D$ corresponding to some nonempty subset $U\subseteq [N]$. We do this in two separate cases based on the size of $U$. As we will see, the significant case is when $|U|$ is small (i.e. polylogarithmic in $N$). In this case, $\per D[U^c]$ is nearly the permanent of $D$, so our strategy will be to use Sinkhorn scaling to adjust $D[U^c]$ into a doubly stochastic matrix, to which we can then apply Proposition \ref{prop:permanent bound}. In what follows, we denote by $o_\op(1)$ an (appropriately sized) matrix with spectral norm tending to $0$ as $N$ tends to infinity.

\begin{lemma}\label{lem:small U per}
    Under the assumptions of Proposition \ref{prop:permanent bound}, any subset $U\subseteq[N]$ with $|U| \le \log^2 N$ satisfies $\per D[U^c] = (1+o(1))\per D$.
\end{lemma}

\begin{proof}
    Let $M = N - |U|$ and consider the nonnegative $M\times M$ matrix $Q = D[U^c]$. We begin by verifying that $Q$ satisfies the hypotheses of Proposition \ref{prop:sinkhorn}. Since $Q$ is a submatrix of $D$ and $\norm B_\mx \le A_0/N$, we have $\norm Q_\mx \le \norm D_\mx \le (A_0+1)/N \le (A_0+1)/M$. Recall that $u,v\in\mathbb R^M$ are the vectors satisfying $Q\one_M = \one_M+u$ and $Q^T\one_M = \one_M + v$. For each $i\in U^c$, the corresponding coordinate of $u$ is $u_i = -\sum_{j\in U} D_{ij}$, since every row of $D$ sums to one. Therefore, $\norm u_2 \le (A_0+1)|U|/\sqrt N$ and $\norm u_\infty \le (A_0+1)|U|/N$. Similarly, $v_i = -\sum_{j\in U}D_{ji}$, so the same bounds hold for $v$. Since $|U| \le \log^2 N$, it follows that $\norm u_2 + \norm v_2 = o(1) =  o(\sqrt M)$ and
    \[
        \norm u_\infty+\norm v_\infty + \frac{\norm u_2 + \norm v_2}{\sqrt M} = o(1).
    \]
    It remains to check that $P = \diag(\one_M+u)^{-1/2} Q \diag(\one_M+v)^{-1/2}$ satisfies the condition $\sigma_2(P) \le \rho + o(1)$. We proceed in two steps, first showing that $\sigma_2(P) \le \sigma_2(Q) + o(1)$, and then upper-bounding $\sigma_2(Q)$ by $\rho$. 
    
    Since $\norm u_\infty = o(1)$, we have $\diag(\one_M + u)^{-1/2} = I + o_\op(1)$, and similarly $\diag(\one_M + v)^{-1/2} = I + o_\op(1)$. For convenience, denote these by $L$ and $R$, respectively. Since $Q$ is a submatrix of the doubly stochastic matrix $D$, we have by Schur's test that $\norm Q_\op \le \sqrt{\norm Q_1\norm Q_\infty} \le 1$. It follows from $P = LQR$ and $\norm Q_\op \le 1$ that
    \[
        \norm{P-Q}_\op \le \norm{L-I}_\op\cdot\norm Q_\op\cdot\norm R_\op + \norm Q_\op\cdot\norm{R - I}_\op = o(1),
    \]
    where we have used that the spectral norm is submultiplicative. Recalling that singular values are $1$-Lipschitz with respect to the spectral norm (cf. \cite[3.3.16(a)]{HornJohnson1991}), this implies $\sigma_2(P)\le\sigma_2(Q)+o(1)$. It remains to bound $\sigma_2(Q)$. For this, note that $Q = J_N[U^c] + B[U^c]$, where $J_N[U^c]$ has rank one. Viewing $J_N[U^c]$ as a rank one approximation of $Q$, the Eckart-Young theorem gives
    \[
        \sigma_2(Q) = \min_{\mathrm{rank}(A) \le 1}\norm{Q-A}_\op \le \norm {Q - J_N[U^c]}_\op = \norm{B[U^c]}_\op\le \norm B_\op \le \rho,
    \]
    so $\sigma_2(P) \le \rho + o(1)$. Therefore, $Q$ satisfies the hypotheses of Proposition \ref{prop:sinkhorn} with constants $A_0 + 1$ and $\rho$. Applying the proposition to $Q$ gives $x,y\in\mathbb R^M$ such that
    \[
        \widetilde D = \diag(e^x) Q \diag(e^y)
    \]
    is doubly stochastic, $\norm x_2+\norm y_2 = o(1)$, and $\norm x_\infty+\norm y_\infty = o(1)$.
    
    Next, we show that $\widetilde D$ satisfies the hypotheses of Proposition \ref{prop:permanent bound}. Let $\widetilde D = J_M + \widetilde B$. Since $\norm x_\infty + \norm y_\infty = o(1)$, we have that $e^{x_i} = 1+o(1)$ and $e^{y_i} =1 + o(1)$ for every $i\in[M]$. Therefore, using that $\norm Q_\mx \le (A_0+1)/N$, we have
    \[
        \norm{\widetilde B}_\mx \le \frac{1}{M} + \norm{\widetilde D}_\mx \le \frac1M + (1+o(1))\norm Q_\mx \le \frac{2A_0 + 3}{M}.
    \]
    Since $\diag(e^x) = I + o_\op(1)$, and similarly for $\diag(e^y)$, our earlier argument shows that $\norm{\widetilde D - Q}_\op \le o(1)$, and consequently $\sigma_2(\widetilde D)\le\sigma_2(Q) + o(1) \le \rho+o(1)$. Since $\widetilde D$ is doubly stochastic, we now have (for all sufficiently large $N$)
    \[
        \norm{\widetilde B}_\op = \sigma_2(\widetilde D) \le \rho + o(1) \le \frac{\rho+1}{2} < 1,
    \]
    so $\widetilde D$ satisfies the hypotheses of Proposition \ref{prop:permanent bound} with constants $2A_0+3$ and $(\rho+1)/2$. Applying the proposition to $\widetilde D$ gives
    \[
        \frac{M^M}{M!}\per\widetilde D = (1+o(1))\det(I - \widetilde B^T\widetilde B)^{-1/2}.
    \]

    We now show that the right hand side is approximately $\det(I - B^TB)^{-1/2}$. In the forthcoming comparison, we view $\widetilde B$ as an $N\times N$ matrix by padding it with zero rows and zero columns on $U$, which does not change $\det(I - \widetilde B^T\widetilde B)$. We first note that $\norm{\widetilde B-B}_F = o(1)$. Indeed, since $\norm B_\mx \le A_0/N$, deleting the rows and columns indexed by $U$ changes $B$ by $O(\sqrt{|U|/N}) = o(1)$. On $U^c\times U^c$, we have
    \[
        \widetilde B - B[U^c] = (\widetilde D - Q) + (J_N[U^c] - J_M).
    \]
    The second term satisfies $\norm{J_N[U^c]-J_M}_F = |U|/N = o(1)$, while the bounds $\norm Q_\mx = O(1/N)$, $\norm x_2 + \norm y_2 = o(1)$, and $|e^{x_i+y_j}-1| = O(|x_i|+|y_j|)$ imply $\norm{\widetilde D - Q}_F = o(1)$. The triangle inequality thus gives $\norm{\widetilde B - B}_F = o(1)$. Moreover, we also have the trivial observation that $\norm B_F + \norm{\widetilde B}_F = O(1)$, since both matrices have $N^2$ entries of size $O(1/N)$.
    
    Now let $H = B^TB$ and $\widetilde H = \widetilde B^T\widetilde B$. Since $\norm B_\op,\norm{\widetilde B}_\op < (\rho + 1)/2$, we may choose $r = (\rho+1)/2 < 1$ so that $\norm H_\op, \norm{\widetilde H}_\op \le r^2$. Thus every eigenvalue of $H$ and $\widetilde H$ lies in the interval $[0, r^2]$, so the eigenvalues of $I - H$ and $I - \widetilde H$ are all at least $1 - r^2 > 0$. Let $H_t = (1-t)H + t\widetilde H$ for $0\le t\le 1$ be the interpolation between $H$ and $\widetilde H$ (note that $I - H_t$ also has all eigenvalues at least $1-r^2$). Applying Jacobi's identity, we have
    \[
        \frac{d}{dt}\log\det(I-H_t) = -\tr\left((I-H_t)^{-1}(\widetilde H - H)\right).
    \]
    Since every eigenvalue of $I - H_t$ is at least $1 - r^2$, we have $\norm{(I-H_t)^{-1}}_\op \le (1-r^2)^{-1}$. By Cauchy-Schwarz, the magnitude of the right hand side is
    \begin{align*}
        \left|\tr\left((I-H_t)^{-1}(\widetilde H - H)\right)\right| &\le \left|\tr\left((I-H_t)^{-1}(\widetilde B-B)^T\widetilde B\right)\right| \\&\qquad+ \left|\tr\left((I-H_t)^{-1}B^T(\widetilde B - B)\right)\right|\\
        &\le \norm{(I-H_t)^{-1}}_\op\cdot\norm{\widetilde B - B}_F\cdot(\norm{\widetilde B}_F+\norm B_F)\\
        &= o(1),
    \end{align*}
    uniformly over all $0\le t\le 1$. Finally, integrating Jacobi's identity from $t = 0$ to $t = 1$ yields
    \begin{align*}
        \log\det(I-\widetilde H) - \log\det(I - H) = -\int_0^1 \tr\left((I-H_t)^{-1}(\widetilde H-H)\right)\mathrm dt = -o(1),
    \end{align*}
    and exponentiation gives $\det(I-\widetilde B^T\widetilde B)^{-1/2} = (1+o(1)) \det(I - B^TB)^{-1/2}$.
    
    It remains to account for the effect of Sinkhorn scaling on $\per Q$. Define the quantity $s = \sum_{i=1}^M(x_i+y_i)$, so that $\per \widetilde D = e^s\per Q$. Since $D$ is doubly stochastic and $\sum_{i,j\in U}D_{ij} = O(|U|^2/N) = o(1)$, the sum of the entries of $Q$ is
    \[
        \sum_{i,j\in U^c}Q_{ij} = N - 2|U| + \sum_{i,j\in U}D_{ij} = M - |U| + o(1).
    \]
    On the other hand, as $\widetilde D$ is doubly stochastic, the sum of its entries is $M$. Over all $i,j\in U^c$, expanding $e^{x_i + y_j}$ using Taylor series (and some computation) now gives
    \[
        M = \sum_{i,j\in U^c}Q_{ij} + s + o(1),
    \]
    where the error in the linear term was $O((\norm u_2 + \norm v_2)(\norm x_2 + \norm y_2)) = o(1)$ and the quadratic remainder was $O(\norm x_2^2 + \norm y_2^2) = o(1)$. Therefore, $s = |U| - o(1)$.

    Finally, applying Proposition \ref{prop:permanent bound} to $D$ and using the determinant comparison above, we obtain
    \[
        \frac{\per Q}{\per D} = (1+o(1)) e^{-s}\frac{M!/M^M}{N!/N^N} = (1+o(1))e^{-|U|+o(1)}e^{|U|-o(1)} = 1+o(1),
    \]
    where the second equality follows from Stirling's formula and $|U| \le \log^2 N$. This concludes the proof of the lemma.
\end{proof}

We now work to bound the sum of the contributions from all larger $U$. To do so, we will use standard bounds on determinants and permanents in the form of Hadamard's inequality, Fact \ref{fact:egorychev-falikman}, and Fact \ref{fact:bregman}.

\begin{lemma}\label{lem:big U per}
    Under the assumptions of Proposition \ref{prop:permanent bound}, there exists a constant $C = C(A_0)$ such that, for every $1 \le k \le N$,
    \[
        \frac{1}{\per D}\sum_{|U| = k}k|\det D[U]|\per D[U^c]\le N^C\left(\frac{C}{\sqrt k}\right)^k.
    \]
    Moreover,
    \[
        \sum_{|U|=k}k|\det D[U]|\le k\left(\frac{A_0+1}{\sqrt k}e\right)^k.
    \]
\end{lemma}

\begin{proof}
    We begin with the simple proof of the second statement. Since $\norm B_\mx \le A_0/N$, we have $\norm{D[U]}_\mx \le (A_0+1)/N$. Therefore, by Hadamard's inequality,
    \[
        \sum_{|U|=k}k|\det D[U]| \le \binom Nk k \left(\sqrt k\frac{(A_0+1)}{N}\right)^k \le k\left(\frac{A_0+1}{\sqrt k}e\right)^k,
    \]
    where we have used $\binom Nk \le (\frac{Ne}k)^k$. 

    We now prove the first statement. Fix some $U\subseteq [N]$ with $|U| = k$, and let $M = N - k$. The case where $M = 0$ is clear, so assume $M \ge 1$. We first claim that
    \[
        \per D[U^c] \le N^{C_1}e^{-M}
    \]
    for some $C_1 = C_1(A_0)$. Let $r = \min\{M, \lfloor\frac{N}{A_0+1}\rfloor\}$, where we may assume $N$ is sufficiently large so that the second term is positive. Since $\norm{D[U^c]}_\mx \le (A_0+1)/N$ and $D$ is doubly stochastic, all entries of $D[U^c]$ are at most $1/r$ and every row sum of $D[U^c]$ is at most $1$.

    Our choice of $r$ implies that we may increase some of the entries in each row of $D[U^c]$, while preserving the upper bound $1/r$, until each row sums to exactly one. Indeed, if some row currently has sum $s \le 1$, then its total remaining ``capacity'' is $M/r - s \ge 1 - s \ge 0$. Hence we may choose an $M\times M$ matrix $Q$ indexed by $U^c$ such that every row of $Q$ sums to $1$, and for every $i,j\in U^c$, we have $Q_{ij} \ge D[U^c]_{ij}$ and $Q_{ij} \le 1/r$. The monotonicity of the permanent for nonnegative matrices together with Fact \ref{fact:bregman} now imply
    \[
        \per D[U^c] \le \per Q \le \left(\frac{(r!)^{1/r}}{r}\right)^{M}.
    \]

    We estimate this expression in two cases. If $r = M$, then this is simply $M!/M^M$, which by Stirling's formula is at most $N^{C_1}e^{-M}$ for some absolute constant $C_1$. Otherwise, $r = \lfloor\frac{N}{A_0+1}\rfloor$ is of order $N$. Applying Stirling's formula again and using $M \le N$, we have
    \[
        M\log\left(\frac{(r!)^{1/r}}{r}\right) = -M + O(\log N).
    \]
    It follows once more that $\per D[U^c] \le N^{C_1}e^{-M}$ for some constant $C_1 = C_1(A_0)$.

    On the other hand, since $D$ is doubly stochastic, Fact \ref{fact:egorychev-falikman} gives $\per D \ge N!/N^N$. Thus, by Stirling's formula,
    \[
        \frac{\per D[U^c]}{\per D} \le N^{C_1}e^{-M}\frac{N^N}{N!} \le N^{C_2}e^k
    \]
    for some $C_2 = C_2(A_0)$. Combining this estimate with the lemma's second statement proved above, we have
    \begin{align*}
        \frac{1}{\per D}\sum_{|U|=k}k|\det D[U]|\per D[U^c] &\le N^{C_2}e^k\sum_{|U|=k}k|\det D[U]| \\&\le N^{C_2}k\left(\frac{A_0+1}{\sqrt k}e^2\right)^k.
    \end{align*}
    Since $k \le N$, we may absorb the additional factor of $k$ into the polynomial factor in $N$. Enlarging the constant once more, there exists a constant $C = C(A_0)$ such that
    \[
        \frac{1}{\per D}\sum_{|U|=k}k|\det D[U]|\per D[U^c] \le N^C\left(\frac{C}{\sqrt k}\right)^k,
    \]
    as desired.
\end{proof}

We will also need the following simple lemma.

\begin{lemma}\label{lem:det(I-B)}
    If $D = J_N + B$ is an $N\times N$ doubly stochastic matrix, then
    \[
        \sum_{\emptyset\neq U\subseteq[N]}|U|(-1)^{|U|-1}\det D[U] = \det(I-B).
    \]
\end{lemma}

\begin{proof}
    Let $f(t) = \det(I - tD)$. By principal minor expansion, this determinant is
    \[
        f(t) = \sum_{U\subseteq[N]}(-t)^{|U|}\det D[U].
    \]
    Note that after differentiating and evaluating at $t = 1$, we have
    \[
        -f'(1) = \sum_{\emptyset \neq U\subseteq[N]}|U|(-1)^{|U|-1}\det D[U],
    \]
    so it suffices to show $f'(1)=-\det(I-B)$. Since $D$ is doubly stochastic, we have $B\one = B^T\one = 0$, and hence $B$ vanishes on $\mathrm{span}\{\one\}$, and $\one^\perp$ is invariant under $B$. On the other hand, $J_N$ acts as the identity on $\mathrm{span}\{\one\}$ and vanishes on $\one^\perp$. It follows that $D$ acts as the identity on $\mathrm{span}\{\one\}$, while on $\one^\perp$ it agrees with $B$. Therefore, with respect to the decomposition $\mathbb R^N = \mathrm{span}\{\one\}\oplus\one^\perp$, the matrices $D$ and $B$ have blocks $D = 1 \oplus B|_{\one^\perp}$ and $B = 0\oplus B|_{\one^\perp}$. It follows that $\det(I-tD) = (1-t)\det(I-tB)$. Differentiating via the product rule and evaluating at $t = 1$ now gives $f'(1) = -\det(I-B)$, as desired.
\end{proof}

We can now prove Proposition \ref{prop:cyc bound}.

\begin{proof}[Proof of Proposition \ref{prop:cyc bound}]
    Since $D$ is doubly stochastic and $\norm D_\mx\le A_0/N$, we have $\norm B_\mx \le A_0/N$, so Proposition \ref{prop:permanent bound} and the preceding lemmas apply. By Lemma \ref{lem:inclusion-exclusion},
    \[
        \frac{N\cyc D}{\per D}=\sum_{\emptyset\neq U\subseteq[N]}|U|(-1)^{|U|-1}\det D[U]\frac{\per D[U^c]}{\per D}.
    \]
    We will split this sum into contributions from the regimes $|U| \le \log^2 N$ and $|U| > \log^2 N$. When $|U| \le \log^2 N$, Lemma \ref{lem:small U per} gives
    \[
        \frac{\per D[U^c]}{\per D} = 1+o(1),
    \]
    which we would like to substitute into the preceding expression. In order to control the sum of the $o(1)$ errors accumulated by this replacement, note that Lemma \ref{lem:big U per} implies
    \[
        \sum_{\emptyset\neq U\subseteq[N]}|U||\det D[U]|\le\sum_{k=1}^\infty k\left(\frac{A_0+1}{\sqrt k}e\right)^k,
    \]
    which converges. It follows that the contribution from all $U$ with $|U|\le\log^2 N$ is
    \[
        \sum_{\substack{\emptyset\neq U\subseteq[N]\\|U|\le\log^2 N}}|U|(-1)^{|U|-1}\det D[U] + o(1).
    \]
    Now consider the $U$ with $|U| > \log^2 N$. Lemma \ref{lem:big U per} gives, for some constant $C = C(A_0)$,
    \[
        \frac{1}{\per D}\sum_{\substack{\emptyset\neq U\subseteq[N]\\|U|>\log^2 N}} |U||\det D[U]|\per D[U^c] \le \sum_{k > \log^2 N}N^C\left(\frac{C}{\sqrt k}\right)^k = o(1),
    \]
    so the contribution from these $U$ is negligible. Moreover, we also have
    \[
        \sum_{\substack{\emptyset\neq U\subseteq[N]\\|U|>\log^2 N}}|U||\det D[U]| \le \sum_{k=\lceil \log^2 N\rceil}^\infty k\left(\frac{(A_0+1)}{\sqrt k}e\right)^k = o(1).
    \]
    Consequently, by Lemma \ref{lem:det(I-B)},
    \[
        \frac{N\cyc D}{\per D}=\sum_{\emptyset\neq U\subseteq[N]}|U|(-1)^{|U|-1}\det D[U] + o(1) = \det(I-B)+o(1).
    \]
    We briefly remark that the Taylor series for $\log\det(I-B)$, together with $\norm B_\op\le\rho<1$, $\norm B_F = O(1)$, and $\norm D_\mx\le A_0/N$, imply $\det(I - B) = \Omega(1)$. So, the $o(1)$ error above may be taken multiplicatively, and hence
    \[
        \frac{N\cyc D}{\per D} = (1+o(1))\det(I-B).
    \]
    Finally, Proposition \ref{prop:permanent bound} applied to $D$ gives
    \[
        \frac{N^N\per D}{N!} = (1+o(1))\det(I-B^TB)^{-1/2}.
    \]
    Combining this with the previous equation, we have
    \[
        \frac{N^N\cyc D}{(N-1)!} = \frac{N^N\per D}{N!}\cdot\frac{N\cyc D}{\per D} = (1+o(1))\frac{\det(I-B)}{\sqrt{\det(I-B^TB)}},
    \]
    which concludes the proof.
\end{proof}

\section{Modified adjacency matrices}\label{sec:modifications}

We begin with a definition:

\begin{definition}\label{def:sparse modifications}
    \textit{Let $Q_0$ be a nonnegative $N\times N$ matrix and let $A_0 < \infty$. We say that $Q$ is an $A_0$-sparse modification of $Q_0$ if there exist sets $R,C\subseteq[N]$ with $|R| + |C| = O(\log^4N)$, probability vectors $p^{(i)}\in \mathbb R_{\ge0}^N$ for $i\in R$, and probability vectors $q^{(j)}\in\mathbb R_{\ge0}^N$ for $j\in C$, all having coordinates at most $A_0/N$, such that for every $1 \le i,j \le N$,
    \[
        Q_{ij} = 
        \begin{cases}
            (Q_0)_{ij} & \text{if }i\notin R,~j\notin C,\\
            p^{(i)}_j & \text{if }i\in R,~j\notin C,\\
            q^{(j)}_i & \text{if }i\notin R,~j\in C,\\
            0 & \text{if }i\in R,~j\in C.
        \end{cases}
    \]}
\end{definition}

In other words, an $A_0$-sparse modification of $Q_0$ is a matrix obtained from $Q_0$ by replacing $O(\log^4N)$ rows and columns by probability vectors with entries at most $A_0/N$, and at the intersections of these rows and columns we simply put zero. 

The aim of this section is to prove the following proposition regarding $A_0$-sparse modifications.

\begin{proposition}\label{prop:modifications}
    Let $A_0 < \infty$ and $\rho < 1$. Let $T$ be a tournament on $N$ vertices with normalized Seidel matrix $K$, and suppose $\norm{K\one_N}_\infty\cdot\norm{K\one_N}_2^2=o(1)$ and $\norm K_\op\le\rho$. If $Q$ is an $A_0$-sparse modification of $J_N - I/N + K$ and $\diag Q = 0$, then
    \[
        \frac{N^N\cyc Q}{(N-1)!} \le (1+o(1))\exp\Big(1 + \sum_{1\le i,j\le N} Q_{ij} - N\Big)Z(K).
    \]
\end{proposition}

To help facilitate this proof, we will introduce Gurvits' notion of the \textit{capacity} of a matrix. For a nonnegative $N\times N$ matrix $Q$, define its capacity by
\[
    \capacity(Q) = \inf_{z}\frac{\prod_{i=1}^N (Qz)_i}{\prod_{i=1}^N z_i},
\]
where the infimum is taken over all $z = (z_1,\dots, z_N)^T\in\mathbb R^N_{>0}$. Capacity was introduced by Gurvits \cite{Gurvits2008} to give a common generalization of the Van der Waerden and Schrijver-Valiant conjectures (whose proof does not rely on the Alexandrov-Fenchel inequality). The quantity has since had a significant afterlife, and is known to behave nicely under several matrix and operator transformations (see, e.g., \cite{Wigderson2020}).

In our application, capacity will simply serve as a convenient way of packaging the cost of applying Sinkhorn scaling on various matrices which are close to being doubly stochastic. To this end, we need the following lemma, which is a specialization of \cite[Lemma 4.2.11]{Kwok18} to the context of Sinkhorn scaling.
\begin{lemma}[{\cite[Lemma 4.2.11]{Kwok18}}]\label{lem:capacity sinkhorn}
    If $D = \diag(e^x)Q\diag(e^y)$ is an $N\times N$ doubly stochastic matrix where $x,y\in\mathbb R^N$, then
    \[
        \capacity(Q) = \exp\big(-\sum_{i=1}^N(x_i+y_i)\big).
    \]
\end{lemma}

\begin{proof}
    The version of capacity in \cite{Kwok18} is defined to be $N$ times the $N$th root of ours for $N\times N$ matrices. Thus, raising both sides of their lemma to the $N$th power and canceling the factor $N^N$ gives the identity $\capacity (LXR) = \det (L)\det(R)\capacity(X)$. In our setting $D$ is doubly stochastic, which implies $\capacity(D) = 1$. Indeed, by the weighted AM-GM inequality,
    \[
        \prod_{i=1}^N(Dz)_i = \prod_{i=1}^N\sum_{j=1}^ND_{ij}z_j\ge\prod_{i=1}^N\prod_{j=1}^Nz_j^{D_{ij}}=\prod_{j=1}^N z_j,
    \]
    and hence $\capacity(D)\ge 1$. On the other hand, taking $z = \one_N$ shows $\capacity(D) \le 1$. The lemma now follows from $1 = \det(\diag(e^x))\det(\diag(e^y))\capacity(Q)$.
\end{proof}

We will also need the following lemma, which records an elementary determinant identity for real skew-symmetric matrices. (This is where the quantity $Z(K)$ emerges from the determinant expressions arising in applications of Proposition \ref{prop:permanent bound}.) Recall that the nonzero eigenvalues of a real skew-symmetric matrix occur in conjugate pairs with real part zero.

\begin{lemma}\label{lem:skew-sym det}
    Let $K$ be a real skew-symmetric matrix with $\norm K_\op<1$, and let $\pm i\lambda_1,\dots,\pm i\lambda_r$ be its nonzero eigenvalues, counted with multiplicity. Then
    \[
        \sqrt{\det(I-K^TK)}=\det(I-iK)
    \]
    and
    \[
        Z(K)=\prod_{j=1}^r\frac{1+\lambda_j^2}{1-\lambda_j^2}\ge1.
    \]
\end{lemma}

\begin{proof}
    Since $\norm K_\op < 1$, we have $0 < \lambda_j<1$ for every $j\in[r]$. Since $K^TK=-K^2$,
    \[
        \det(I-K^TK)=\det(I+K^2)=\det(I-iK)\det(I+iK) = \det(I-iK)^2,
    \]
    where the last equality follows from $I+iK=(I-iK)^T$. Moreover,
    \[
        \det(I-iK)=\prod_{j=1}^r(1-\lambda_j^2)>0,
    \]
    so $\sqrt{\det(I-K^TK)}=\det(I-iK)$, which proves the first statement. For the second statement, we similarly have $\det(I-K)=\prod_{j=1}^r(1+\lambda_j^2)$, and therefore
    \[
        Z(K)=\frac{\det(I-K)}{\det(I-iK)}=\prod_{j=1}^r\frac{1+\lambda_j^2}{1-\lambda_j^2}\ge1,
    \]
    as desired.
\end{proof}

\begin{proof}[Proof of Proposition \ref{prop:modifications}]
    Let $Q_0 = J_N - I/N+K$ and define $R,C\subseteq[N]$ so that $Q$ is an $A_0$-sparse modification of $Q_0$ whose entries may differ from $Q_0$ only in rows $R$ and columns $C$ (thus $|R|+|C|=O(\log^4N)$). 
    
    We first show that $Q$ satisfies the properties of Proposition \ref{prop:sinkhorn}. From the definition of an $A_0$-sparse modification and $\norm{Q_0}_\mx \le 2/N$, we have that $\norm Q_\mx\le(A_0+2)/N$. Recall that $u,v\in\mathbb R^N$ are the vectors $u = Q\one_N-\one_N$ and $v = Q^T\one_N-\one_N$. For any $1 \le i \le N$, if $i\notin R$, then the only entries in row $i$ which may differ between $Q_0$ and $Q$ occur in the columns indexed by $C$, so
    \[
        u_i = (K\one_N)_i - \frac1N + O\left(\frac{|C|}N\right).
    \]
    On the other hand, if $i\in R$, then the entries outside of the columns indexed by $C$ agree with the probability vector $p^{(i)}$ (whose coordinates sum to one), and hence $u_i = O(|C|/N)$. A symmetric argument gives $v_j = -(K\one_N)_j - 1/N + O(|R|/N)$ if $j\notin C$ while $v_j = O(|R|/N)$ if $j\in C$ for any $1\le j\le N$. Since $|R|+|C|=O(\log^4N)$, it follows that
    \[
        \norm u_\infty+\norm v_\infty\le 2\norm{K\one_N}_\infty + O\left(\frac{1+|R|+|C|}{N}\right) = 2\norm{K\one_N}_\infty + o(1) = o(1),
    \]
    where we have used $\norm{K\one_N}^3_\infty \le \norm{K\one_N}_\infty\cdot\norm{K\one_N}_2^2=o(1)$. Since $\norm w_2\le\sqrt N\norm w_\infty$ for every $w\in\mathbb R^N$, this also gives $(\norm u_2+\norm v_2)/\sqrt N=o(1)$, so it only remains to check the hypothesis about $\sigma_2(P)$. We proceed in a similar fashion as in Lemma \ref{lem:small U per}. Since $\norm u_\infty=o(1)$, we have $\diag(\one_N+u)^{-1/2}=I+o_\op(1)$, and similarly $\diag(\one_N+v)^{-1/2}=I+o_\op(1)$. For convenience, let us denote these by $L'$ and $R'$, respectively. Since $\norm Q_\mx \le (A_0+2)/N$, we have $\norm Q_1,\norm Q_\infty\le A_0+2$, so by Schur's test $\norm Q_\op\le\sqrt{\norm Q_1\norm Q_\infty} =O(1)$. It follows from $P = L'QR'$ and $\norm Q_\op=O(1)$ that
    \[
        \norm{P-Q}_\op\le\norm{L'-I}_\op\cdot\norm Q_\op\cdot \norm {R'}_\op+\norm Q_\op\cdot\norm{R'-I}_\op = o(1).
    \]
    Every entry of $Q$ and $Q_0$ is $O(1/N)$, so $\norm{Q-Q_0}_\mx =O(1/N)$. The two matrices $Q$ and $Q_0$ differ on at most $N(|R|+|C|)$ entries. Therefore,
    \[
        \norm{Q-Q_0}_\op\le\norm{Q-Q_0}_F\le \sqrt{N(|R|+|C|)O(N^{-2})} = O\left(\sqrt{\frac{|R|+|C|}{N}}\right)=o(1),
    \]
    where we have used that $\norm A_\op\le\norm A_F$ for every matrix $A$. The triangle inequality therefore gives $\norm{P-Q_0}_\op=o(1)$. As singular values are $1$-Lipschitz with respect to the spectral norm, this implies $\sigma_2(P)\le\sigma_2(Q_0)+o(1)$, so it remains to bound $\sigma_2(Q_0)$. Viewing $J_N$ as a rank one approximation of $Q_0$, by the Eckart-Young theorem
    \[
        \sigma_2(Q_0)\le\norm{Q_0-J_N}_\op = \norm{K-I/N}_\op\le\norm K_\op+\norm{I/N}_\op \le \rho+o(1),
    \]
    so $\sigma_2(P)\le\rho+o(1)$. Therefore, $Q$ satisfies the hypotheses of Proposition \ref{prop:sinkhorn} with constants $A_0+2$ and $\rho$. Applying the proposition to $Q$ gives $x,y\in\mathbb R^N$ such that
    \[
        D = \diag(e^x)Q\diag(e^y)
    \]
    is doubly stochastic, $\norm x_2+\norm y_2=o(\sqrt N)$, and $\norm x_\infty+\norm y_\infty=o(1)$.

    Next, we compare $D$ with $J_N+K$ in order to apply Proposition \ref{prop:cyc bound} to $D$. Since $\norm x_\infty+\norm y_\infty=o(1)$, we have $\max_{1\le i,j\le N}|e^{x_i+y_j}-1|=o(1)$. Therefore, 
    \[
        \norm{D-Q}_F\le\norm Q_F\max_{1\le i,j\le N}|e^{x_i+y_j}-1| \le N\norm Q_\mx\max_{1\le i,j\le N}|e^{x_i+y_j}-1|=o(1),
    \]
    where the final inequality comes from $\norm Q_\mx=O(1/N)$. Setting $B = D - J_N$, it follows that $\norm{B-K}_F=o(1)$, since
    \[
        \norm{B-K}_F=\norm{D-Q_0-I/N}_F\le\norm{D-Q}_F+\norm{Q-Q_0}_F+\norm{I/N}_F= o(1).
    \]
    Moreover, since $\rho < 1$, we have for all sufficiently large $N$ that
    \[
        \norm B_\op\le \norm{B-K}_\op+\norm K_\op\le \norm{B-K}_F+\norm K_\op\le \rho+o(1)<\frac{\rho+1}{2}.
    \]
    Finally, we have $\norm D_\mx \le \norm Q_\mx(\max_{1\le i,j\le N}e^{x_i+y_j}) \le (2A_0+4)/N$, so $D$ satisfies the hypotheses of Proposition \ref{prop:cyc bound} with constants $2A_0+4$ and $(\rho+1)/2$. Applying the proposition to $D$, we have
    \[
        \frac{N^N\cyc D}{(N-1)!}=(1+o(1))\frac{\det(I-B)}{\sqrt{\det(I-B^TB)}}.
    \]

    We now record the cost of the Sinkhorn scaling. Since every term contributing to $\cyc Q$ contains exactly one entry from each row and each column, we have $\cyc D =\exp\big(\sum_{i=1}^N(x_i+y_i)\big)\cyc Q$. By Lemma \ref{lem:capacity sinkhorn}, this gives $\cyc Q = \cyc (D)\capacity(Q)$. We can bound $\capacity(Q)$ by evaluating the infimum (in the capacity definition) at $z = \one_N$:
    \[
        \capacity(Q) \le \prod_{i=1}^N(1+u_i)\le\exp\left(\sum_{i=1}^Nu_i\right)=\exp\left(\sum_{1\le i,j\le N}Q_{ij}-N\right).
    \]

    It remains to show that
    \[
        \frac{\det(I-B)}{\sqrt{\det(I-B^TB)}}=(1+o(1))eZ(K).
    \]
    To do so, we compare the numerator with $e\det(I-K)$ and the denominator with $\sqrt{\det(I-K^TK)}$. This step will involve two applications of Jacobi's identity, in a similar fashion to that of Lemma \ref{lem:small U per}. Note that $\norm{B-K}_F=o(1)$ and $\norm K_F=O(1)$ imply that we also have $\norm B_F=O(1)$. Moreover, for all sufficiently large $N$, both $\norm B_\op$ and $\norm K_\op$ are at most $r = (\rho+1)/2<1$.

    We begin with the numerator. For $0\le t\le 1$, let $B_t=tB+(1-t)K$. Then $\norm {B_t}_\op\le r$, so every eigenvalue $\lambda$ of $B_t$ satisfies $|\lambda|<1$ and $I-B_t$ is invertible. Moreover, the real eigenvalues of $I-B_t$ are positive, while its complex eigenvalues occur in conjugate pairs whose products are positive, so $\det(I-B_t)>0$. By Jacobi's identity,
    \[
        \frac d{dt}\log\det(I-B_t)=-\tr\left((I-B_t)^{-1}(B-K)\right).
    \]
    Using $(I-B_t)^{-1}=I+B_t(I-B_t)^{-1}$, this breaks into
    \[
        \frac d{dt} \log\det(I-B_t)=-\tr(B-K)-\tr\left(B_t(I-B_t)^{-1}(B-K)\right).
    \]
    Since $B = D - J_N$ and $\diag Q = \diag D = 0$, we have $\tr(B)=-1$. As $K$ is a normalized Seidel matrix, $\tr(K)=0$, so the first term above is simply $1$. To bound the second term, we use the fact that for matrices $X$ and $Y$, $|\tr(XY)|\le\norm X_F\norm Y_F$, which follows from Cauchy-Schwarz. Applying this inequality to $\tr\left(B_t(I-B_t)^{-1}(B-K)\right)$,
    \[
        |\tr\left(B_t(I-B_t)^{-1}(B-K)\right)|\le\norm{B_t(I-B_t)^{-1}}_F\cdot\norm{B-K}_F.
    \]
    Let $w_i$ denote the $i$th row of $B_t$. By the definition of the spectral norm,
    \[
        \norm{B_t(I-B_t)^{-1}}_F^2=\sum_{i=1}^N\norm{w_i(I-B_t)^{-1}}_2^2 \le \norm{(I-B_t)^{-1}}^2_\op\cdot\norm{B_t}_F^2.
    \]
    Since $\norm{B_t}_\op\le r < 1$, we have by the Neumann series that
    \[
        \norm{(I-B_t)^{-1}}_\op = \left\|\sum_{i=0}^\infty B_t^i\right\|_\op\le\sum_{i=0}^\infty\norm{B_t}_\op^i \le \sum_{i=0}^\infty r^i=\frac{1}{1-r}. 
    \]
    Now we know that $\norm{B_t}_F\le \norm B_F+\norm K_F=O(1)$ and $\norm{(I-B_t)^{-1}}_\op\le(1-r)^{-1}$, so $\norm{B_t(I-B_t)^{-1}}_F=O(1)$. Combining this with $\norm{B-K}_F=o(1)$ implies $|\tr\left(B_t(I-B_t)^{-1}(B-K)\right)|=o(1)$, and hence $\frac d{dt}\log\det(I-B_t)=1+o(1)$. (This is uniform over all $0\le t\le 1$.) Integrating this expression from $t=0$ to $t=1$,
    \[
        \log\det(I-B)-\log\det(I-K)=\int_0^1(1+o(1))\mathrm dt = 1+o(1),
    \]
    so $\det(I-B)=(1+o(1))e\det(I-K)$.

    Next, we compare the denominator with $\sqrt{\det(I-K^TK)}$. Let $H = K^TK$ and $\widetilde H = B^TB$. Since $\norm K_\op,\norm B_\op\le r$, we have $\norm H_\op,\norm{\widetilde H}_\op\le r^2$. We may now apply the argument involving $H$ and $\widetilde H$ from Lemma \ref{lem:small U per} verbatim to see that $\det(I-B^TB)^{-1/2}=(1+o(1))\det(I-K^TK)^{-1/2}$.

    Combining the previous two comparisons with Lemma \ref{lem:skew-sym det}, we have
    \[
        \frac{\det(I-B)}{\sqrt{\det(I-B^TB)}} = (1+o(1))e\frac{\det(I-K)}{\sqrt{\det(I-K^TK)}} = (1+o(1))e\frac{\det(I-K)}{\det(I-iK)},
    \]
    which is $(1+o(1))eZ(K)$. Putting everything together,
    \[
        \frac{N^N\cyc Q}{(N-1)!} = \capacity(Q)\frac{N^N\cyc D}{(N-1)!} \le (1+o(1))\exp\Big(1+\sum_{1\le i,j\le N}Q_{ij}-N\Big)Z(K),
    \]
    which completes the proof of Proposition \ref{prop:modifications}.
\end{proof}

\section{Upper bound on $Z(K)$}\label{sec:Z bound}

Let $K$ be the normalized Seidel matrix of an arbitrary tournament. The aim of this section is to prove the following upper bound on the quantity $Z(K) = \det(I-K)/\det(I-iK)$. 

\begin{proposition}\label{prop:Z(K) bound}
    For $N$ sufficiently large, if $T$ is a tournament on $N$ vertices with normalized Seidel matrix $K$, then
    \[
        \frac{\det(I-K)}{\det(I-iK)}\le\frac{\sum_{j=0}^{\lfloor N/2\rfloor}\binom{N}{2j}N^{-2j}}{\sum_{j=0}^{\lfloor N/2\rfloor}(-1)^j\binom{N}{2j}N^{-2j}}\overset{N}{\longrightarrow}\frac{\cosh1}{\cos1}
    \]
\end{proposition}

Let $U=U_N$ be the Seidel matrix of the transitive tournament, i.e. the $N\times N$ matrix with $U_{ij}=1$ for $i<j$ and $U_{ij}=-1$ for $i>j$, and zero diagonal. The proof proceeds in two steps. We will first show that $Z(K)\le Z(U/N)$, in that the transitive tournament attains the maximum value for $Z(K)$ (this holds for every $N$). Then, we will evaluate $Z(U/N)$ to be the ratio of two sums, which tend to $\cosh1$ and $\cos1$ as $N\to\infty$, respectively. 

For a real matrix $A$, define the matrix $|A|$ to be $|A|=(A^TA)^{1/2}$. Thus, $\tr|A|^{2m}$ is the sum of the $2m$-th powers of the singular values of $A$. The main step in showing $Z(K)\le Z(U/N)$ is to show that $\tr|C|^{2m}\le\tr|U|^{2m}$ for every tournament Seidel matrix $C$ of order $N$ and every positive integer $m$. The reason for this comparison is that $\log Z(K)$ can be expanded as a convergent series of such terms with positive coefficients, and so bounding each term by its value for the transitive tournament will in turn bound $Z(K)$ by $Z(U_N/N)$.

The argument for this comparison is due to Fang and Huang \cite[Section 3]{FangHuang2026}, who showed that $U_N$ maximizes $\tr(C^{4k})$ among tournament Seidel matrices $C$ of order $N$. For our application, we need this same comparison for sums of other even powers, which also follows straightforwardly from their proof. We first introduce the notation needed to reconstruct this argument.

For $x\in \mathbb R^N$, let $x^{\downarrow}\in\mathbb R^N$ be the vector consisting of the absolute value of the coordinates of $x$, rearranged in nonincreasing order. Let $\mathcal D_N$ denote the set of non-increasing vectors in $\mathbb R^N$ with nonnegative entries. Recall that the \textit{weak majorization order} $\preceq_w$ is a partial order on the vectors in $\mathcal D_N$, which is defined as follows. For $x,y\in\mathcal D_N$, we say $x\preceq_w y$ if $\sum_{j=1}^r x_j \le \sum_{j=1}^r y_j$ for every $1\le r \le N$.

It will also be helpful to introduce the matrix $M$, which allows us to compare the powers of $C$ and $U$. Define the $N\times N$ symmetric matrix $M = (M_{rs})_{r,s\in[N]}$ by
\[
    M_{rs}=\begin{cases}1 & r+s\le N, \\ 0 & r + s = N + 1, \\  (-1)^{r+s-N-2}&r+s\ge N+2.\end{cases}
\]
(This is called $K$ in \cite{FangHuang2026}.) For instance, if $N = 5$ then
\[
    M = \begin{bmatrix}
        1 & 1 & 1 & 1 & 0 \\
        1 & 1 & 1 & 0 & 1 \\
        1 & 1 & 0 & 1 & -1 \\
        1 & 0 & 1 & -1 & 1 \\
        0 & 1 & -1 & 1 & -1
    \end{bmatrix}.
\]

\begin{lemma}\label{lemma:trace comparison}
    For every tournament Seidel matrix $C$ of order $N$ and every integer $m\ge 1$,
    \[
        \tr|C|^{2m}\le\tr|U_N|^{2m}.
    \]
\end{lemma}
\begin{proof}
    The case $N = 1$ is immediate, so assume $N\ge2$. By \cite[Proposition 3.1]{FangHuang2026} (Karamata's inequality) applied to $t\mapsto t^2$, we have that $u\preceq_w v$ implies $\norm u_2\le \norm v_2$. Moreover, applying \cite[Corollary 3.4]{FangHuang2026} with $A_1 = \dots = A_m = C$, we have $(C^mx)^\downarrow\preceq_w M^mx^\downarrow$ for every $x\in\mathbb R^N$. In particular, the vector $M^mx^\downarrow$ is nonnegative and nonincreasing. Denoting by $e_j$ the $j$th standard basis vector, we have $e_j^\downarrow = e_1$. Applying the Euclidean norm inequality and using $M^T = M$, we have for every $1 \le j \le N$ that
    \[
        \norm{C^me_j}_2^2\le\norm{M^m e_1}_2^2 = (M^{2m})_{11}.
    \]
    It remains to calculate the value of $(M^{2m})_{11}$. To do this, we directly apply two observations from the proof of \cite[Theorem 1.4]{FangHuang2026}. First, there exists an orthogonal signed permutation matrix $S$ such that $Se_1 = e_1$ and $-U^2 = S^TM^2S$. Namely $S = DP$, where $D$ is the diagonal matrix with $D_{ii}=(-1)^{i+1}$ and $P$ is the permutation matrix corresponding to the permutation $(\sigma(1),\sigma(2),\dots,\sigma(N))=(1, N, 2, N-1, 3, N-2,\dots)$, where $P_{i,\sigma(i)}=1$. By taking $m$th powers and comparing the first diagonal entries, we have $(M^{2m})_{11} = ((-U^2)^m)_{11}$. Second, the signed cyclic shift sending $e_j$ to $e_{j-1}$ for $j > 1$ and $e_1$ to $-e_N$ commutes with $U$, and therefore also commutes with $(-U^2)^m$. Conjugation by this shift cycles the diagonal entries of $(-U^2)^m$, so all diagonal entries of $(-U^2)^m$ are equal. Since $U^TU = -U^2$, it follows that $N(M^{2m})_{11} = \tr|U|^{2m}$. 

    Finally, since $C^T = -C$, we have $(C^m)^TC^m = (-C^2)^m=|C|^{2m}$. Summing the inequality $\norm{C^m e_j}_2^2 \le (M^{2m})_{11}$ over the standard basis vectors, we have
    \[
        \tr|C|^{2m} = \sum_{j=1}^N\norm{C^me_j}_2^2\le N(M^{2m})_{11} = \tr|U|^{2m},
    \]
    as desired.
\end{proof}

We now prove Proposition \ref{prop:Z(K) bound} by expanding $Z(K)$ using Lemma \ref{lem:skew-sym det} and then applying Lemma \ref{lemma:trace comparison} to each term in the expansion.

\begin{proof}[Proof of Proposition \ref{prop:Z(K) bound}]
    Let $K = C/N$. The sum of the absolute values in each row and each column of $K$ is $(N-1)/N$, so $\norm K_\op\le (N-1)/N < 1$, and the same holds for $U/N$. Since $K$ is real and skew-symmetric, its nonzero eigenvalues can be listed as $\pm i\lambda_1,\dots,\pm i\lambda_r$, with multiplicity. By Lemma \ref{lem:skew-sym det}, applied to $K$ and its eigenvalues, we have
    \[
        \log Z(K)=\sum_{j=1}^r\log\frac{1+\lambda_j^2}{1-\lambda_j^2}=\sum_{s=0}^\infty\frac{\tr|K|^{4s+2}}{2s+1}.
    \]
    Since $\norm K_\op < 1$, this series converges absolutely. By Lemma \ref{lemma:trace comparison} applied with $m = 2s + 1$, each numerator in the rightmost summation is at most $\tr|U/N|^{4s+2}$. Since all coefficients are positive, comparing this expansion of $\log Z(K)$ with the expansion of $\log Z(U/N)$ gives $\log Z(K) \le \log Z(U/N)$, and hence $Z(K)\le Z(U/N)$. 

    It remains to evaluate $Z(U/N)$. Every principal submatrix of $U$ is again the Seidel matrix of a transitive tournament. We claim that the Pfaffian $\pf(U_{2j})=1$ for every $j\ge 0$, with the empty Pfaffian defined as one. For $j\ge 1$, expansion along the first row gives
    \[
        \pf(U_{2j})=\sum_{\ell=2}^{2j}(-1)^\ell\pf(U_{2j-2})=\pf(U_{2j-2}),
    \]
    because the alternating sum has an odd number of terms and begins and ends with coefficient $+1$. The claim now follows from induction. Since $\det(U_{2j})=\pf(U_{2j})^2$, every even order principal determinant of $U$ is one. On the other hand, every odd principal determinant is zero, because the corresponding principal submatrix is skew-symmetric. The expansion of a determinant into principal minors therefore gives, for every $z\in\mathbb C$,
    \[
        \det(I-zU/N) = \sum_{j=0}^{\lfloor N/2\rfloor}\binom{N}{2j}(z/N)^{2j}.
    \]
    Taking $z=1$ and $z = i$, we have
    \[
        Z(U/N) = \frac{\sum_{j=0}^{\lfloor N/2\rfloor}\binom{N}{2j}N^{-2j}}{\sum_{j=0}^{\lfloor N/2\rfloor}(-1)^j\binom{N}{2j}N^{-2j}}.
    \]
    By the binomial theorem, the numerator of this ratio is $\frac12((1+N^{-1})^N+(1-N^{-1})^N)$, which tends to $\cosh 1$ as $N\to\infty$. The denominator is $\mathrm{Re}(1+i/N)^N$, which tends to $\cos(1)$ as $N\to\infty$. As the upper bound depends only on $N$, we have $Z(K) = Z(C/N) \le Z(U/N) \le  \cosh1/\cos1+o(1)$ uniformly over all normalized tournament Seidel matrices $K$ of order $N$, as desired.
\end{proof}

\section{Proof of the upper bound}\label{sec:upper}

In this section, we prove Theorem \ref{thm:upper}. We first show using Bregman's theorem (Fact \ref{fact:bregman}) that if a tournament $T$ has more Hamiltonian cycles than the random tournament, then aside from a small exceptional set of vertices the rest of $T$ is nearly regular.

\begin{proposition}\label{prop:polylog set}
    Let $T$ be a tournament on $n$ vertices with $C(T)\ge(n-1)!/2^n$. There exists a set $S\subseteq V(T)$ with $|S| = O(\log^4n)$ such that, if $G = T - S$ is a subtournament on $N$ vertices with normalized Seidel matrix $K$, then $\norm{K\one_N}_\infty\cdot\norm{K\one_N}_2^2 = o(1)$ and $\norm K_\op \le 2/\pi$.
\end{proposition}

\begin{proof}
    Let $A$ be the adjacency matrix of $T$, that is, the $n\times n$ matrix with rows and columns indexed by $V(T)$ where $A_{ij}=1$ if $ij$ is a directed edge in $T$. For each $v\in V(T)$, let $d_T^+(v)$ denote the outdegree of $v$, which equals the sum of the entries in row $v$ of $A$. Let $m = (n-1)/2$; we would hope that for $T$ close to a regular tournament, each vertex has indegree and outdegree approximately $m$. The permanent of $A$ counts the $1$-factors of $T$, and every Hamiltonian cycle of $T$ is a $1$-factor. Therefore, applying Fact \ref{fact:bregman} after normalizing the rows of $A$ gives
    \[
        \log C(T)\le \log \per A \le \sum_{v\in V(T)}\frac{\log(d^+_T(v)!)}{d^+_T(v)}.
    \]
    (Since $C(T)\ge(n-1)!/2^n$, there is at least one Hamiltonian cycle in $T$, so each vertex has positive indegree and outdegree; thus the above normalization is well-defined.)
    
    In the forthcoming computation, it will be convenient to make this summation symmetric in the indegrees and outdegrees of each vertex. To this end, note that reversing the orientation of every edge of $T$, which is equivalent to considering the adjacency matrix $A^T$, clearly preserves the number of Hamiltonian cycles. Since also $\per A = \per A^T$ and every vertex of $T$ has sum of indegree and outdegree $2m$, applying Fact \ref{fact:bregman} in the same way for the reversed tournament implies
    \[
        \log C(T) \le \frac12\sum_{v\in V(T)}\left(\frac{\log(d^+_T(v)!)}{d^+_T(v)}+\frac{\log((2m-d^+_T(v))!)}{2m-d^+_T(v)}\right).
    \]
    We will estimate each term in this sum using the bound $\log(a!)\le a\log a -a+\log a + 1$, which is a simple consequence of Stirling's approximation and holds for all $a \ge 1$. For any integer $1\le d\le 2m-1$, we have by this bound that
    \begin{align*}
        \frac12\left(\frac{\log(d!)}d+\frac{\log((2m-d)!)}{2m-d}\right)&\le\log m-1+\frac12\log\left(1-\frac{(d-m)^2}{m^2}\right)\\&+\frac12\left(\frac{\log d+1}d+\frac{\log(2m-d)+1}{2m-d}\right).
    \end{align*}
    Suppose first that $|d-m|\le m/2$. Then both $d$ and $2m-d$ are at least $m/2$, so the last two terms are at most $2(\log(2m)+1)/m$. Moreover, since $\log(1-a)\le -a$ for every $0\le a<1$, the middle logarithmic term is at most $-(d-m)^2/(2m^2)$. In particular, when $|d-m|\le m/2$, the preceding expression is at most
    \[
        \log m-1-\frac{(d-m)^2}{8m^2}+\frac{2(\log(2m)+1)}{m}.
    \]
    (The middle term has denominator $8m^2$ instead of $2m^2$ for future consistency.) It remains to consider the case $|d-m|>m/2$. By the symmetry, we may assume $d < m/2$. Consider the function $f(x) = \log x+\frac{\log x+1}{x}$, which is concave on $[1,\infty)$. By concavity, the sum $f(d)+f(2m-d)$ is maximized, for $1\le d\le m/2$, at $d=m/2$. Therefore, substituting $d=m/2$ shows that when $|d-m|>m/2$,
    \[
        \frac12\left(\frac{\log d!}d+\frac{\log((2m-d)!)}{2m-d}\right)\le\log m-1+\frac12\log\frac34+o(1).
    \]
    Since $\frac12\log\frac34<-1/8$, this is at most $\log m-1-\frac18$ for all sufficiently large $n$. As $(d-m)^2/m^2 \le 1$, it follows that in this range we also have
    \[
        \frac12\left(\frac{\log(d!)}{d}+\frac{\log((2m-d)!)}{2m-d}\right)\le\log m-1-\frac{(d-m)^2}{8m^2}+\frac{2(\log(2m)+1)}{m}.
    \]

    Thus, this bound holds for every $1\le d \le 2m-1$, so we may apply it with $d = d^+_T(v)$ for each $v \in V(T)$ to upper bound $\log C(T)$. Summing over all the vertices,
    \[
        \log C(T)\le n(\log m-1)-\frac1{8m^2}\sum_{v\in V(T)}(d^+_T(v)-m)^2 + O(\log n).
    \]
    On the other hand, we know $C(T) \ge (n-1)!/2^n$. Together with the lower bound $\log (a!)\ge a\log a-a + 1$ (which again holds for all $a\ge 1$), we have
    \begin{align*}
        \log C(T)&\ge \log((n-1)!)-n\log 2 \ge (n-1)\log(n-1)-(n-1)+1-n\log 2\\
        &= n(\log m-1)+O(\log n).
    \end{align*}
    Combining this with the upper bound on $\log C(T)$ yields
    \[
        \sum_{v\in V(T)}(d^+_T(v)-m)^2=O(n^2\log n).
    \]

    Next, we choose the set $S$ to be the vertices $v\in V(T)$ where $d^+_T(v)$ differs from $m$ by a significant amount. In particular, let $S$ be given by
    \[
        S = \left\{v\in V(T) : |d^+_T(v)-m| > \frac{n}{\log^{3/2}n}\right\}.
    \]
    Since $\sum_{v\in V(T)}(d^+_T(v)-m)^2=O(n^2\log n)$, it follows that $|S|=O(\log^4n)$. Let $G = T - S$ be the tournament on $N = n - |S|$ vertices obtained by deleting the vertices in $S$, and let $K$ be the normalized Seidel matrix of $G$. Let $d^+_G(v)$ be the number of outneighbors of $v$ in $G$, and denote by $d^+_S(v)$ the number of outneighbors of $v$ in $S$. Deleting $S$ changes the outdegrees of each $v \in V(G)$ so that $d^+_T(v)-d_G^+(v) \le |S|$. Since $K$ is the normalized Seidel matrix of $G$, we have
    \[
        (K\one_N)_v=\frac{2d^+_G(v)-(N-1)}N = \frac{2(d^+_T(v)-m)+|S|-2d^+_S(v)}{N}.
    \]
    From the definition of $S$, each vertex $v\in V(G)$ satisfies $|d_T^+(v)-m| = O(n\log^{-3/2}n)$, and each of the other two terms in the numerator is $O(\log^4n)$. Moreover, since $|S|=O(\log^4n)$, we have $N = (1-o(1))n$, and therefore $\norm{K\one_N}_\infty = O(\log^{-3/2}n)$. Next, we bound $\norm{K\one_N}_2^2$. Using the inequality $(a+b)^2\le2a^2+2b^2$, we have
    \[
        \norm{K\one_N}_2^2\le\frac{8}{N^2}\sum_{v\in V(G)}(d_T^+(v)-m)^2+\frac{2|S|^2}{N}.
    \]
    We previously established that the first term is $O(\log n)$, even when summing over all $v\in V(T)$. Since $|S| = O(\log^4n)$ and $N = (1-o(1))n$, the second term is $o(1)$. Therefore, $\norm{K\one_N}_2^2 = O(\log n)$, and so $\norm{K\one_N}_\infty\cdot\norm{K\one_N}_2^2 = O(\log^{-1/2}n)=o(1)$, as desired.

    It only remains to bound the spectral norm of $K$. Let $C$ be the Seidel matrix of $G$, so that $K = C/N$. By Lemma \ref{lem:Seidel norm bound}, we have $\norm{C}_\op\le\cot{\frac\pi{2N}}$. Consequently, $\norm{K}_\op\le\frac1N\cot\frac\pi{2N} \le \frac2\pi$. This proves Proposition \ref{prop:polylog set}.
\end{proof}

We are now ready to prove Proposition \ref{prop:insertions}. The upper bound Theorem \ref{thm:upper} will then follow immediately by combining Propositions \ref{prop:Z(K) bound}, \ref{prop:polylog set}, and \ref{prop:insertions}.

\begin{proposition}\label{prop:insertions}
    Let $\rho < 1$. Let $T$ be a tournament on $n$ vertices and let $S\subseteq V(T)$ satisfy $|S| = O(\log^4 n)$. Let $G = T - S$ be a subtournament on $N$ vertices with normalized Seidel matrix $K$. If $\norm{K\one_N}_\infty\cdot\norm{K\one_N}_2^2=o(1)$ and $\norm K_\op\le\rho$, then
    \[
        C(T)\le \left(\frac{\det(I-K)}{\det(I-iK)}+o(1)\right)\frac{(n-1)!}{2^n}.
    \]
\end{proposition}

\begin{proof}
    Let $s = |S|$, so $n = N+s$ and $s = O(\log^4N)$. Let $A$ be the adjacency matrix of $G$, with rows and columns indexed by $V(G)$ as before, and let $Q_0 = J_N-I/N+K=2A/N$. Let $H$ be a Hamiltonian cycle of $T$. The edges of $H$ with both endpoints in $S$ induce a forest whose components are directed paths. Roughly speaking, we will partition the Hamiltonian cycles of $T$ according to their paths through $S$, apply Proposition \ref{prop:modifications} to each part, and then sum over the possibilities for these paths.

    Let $\mathcal F$ be the set of forests $F\subseteq T[S]$ whose vertex set is $S$ and whose components are directed paths (possibly trivial ones consisting of a single vertex). The restriction of every Hamiltonian cycle of $T$ to $S$ is such a forest. For $F\in\mathcal F$, let $c(F)$ denote its set of components. For each $P\in c(F)$, let $u_P, v_P$ be its first and last vertices, and let $a_P = |N^-_G(u_P)|$ and $b_P=|N^+_G(v_P)|$ be the numbers of possible predecessors and successors of $P$ in $G$.

    We will further distinguish Hamiltonian cycles with the same forest $F\in\mathcal F$ by recording one vertex of $G$ (either a predecessor or a successor) for each component $P\in c(F)$. If $a_P\le b_P$, we record the predecessor $x_P\in N_G^-(u_P)$, and we say that $P$ is \textit{represented} by row $x_P$. Otherwise, we record the successor $y_P\in N_G^+(v_P)$ and represent $P$ by column $y_P$. Thus we record a vertex from the smaller neighborhood and encode the larger neighborhood in a replacement vector. A \textit{pattern} consists of $F$ and these recorded vertices. It suffices to consider only those patterns for which the recorded predecessors are distinct and the recorded successors are distinct (but a predecessor and a successor are allowed to coincide). Note that all Hamiltonian cycles are consistent with these restrictions, since every vertex has a unique predecessor and successor along a cycle.

    Let $\mathcal P_F$ be the set of patterns with forest $F$, and let $\mathcal P=\bigcup_{F\in\mathcal F}\mathcal P_F$. For $\pi\in\mathcal P$, let $R_\pi, C_\pi$ be its sets of recorded predecessors and successors, respectively, and let $C(T,\pi)$ be the number of Hamiltonian cycles of $T$ with pattern $\pi$. Every Hamiltonian cycle determines exactly one pattern, so $C(T) = \sum_{\pi\in\mathcal P} C(T,\pi)$.

    Fix $F\in\mathcal F$ and $\pi\in\mathcal P_F$, and let $k = |c(F)|$. Our plan will be to encode the rows and columns $R_\pi$ and $C_\pi$ of $Q_0$ by replacement vectors to create a matrix $Q_\pi$, to which we will then apply Proposition \ref{prop:modifications}. In order to apply the proposition, these replacement vectors must have entries of order $O(1/N)$, so we must be careful depending on whether $\max\{a_P,b_P\}$ is indeed of order $\Omega(N)$. For a component $P\in c(F)$ represented by a row, the replacement vector is the uniform probability vector on $N_G^+(v_P)\setminus\{x_P\}$ if $b_P\ge N/2$, and on $V(G)\setminus\{x_P\}$ otherwise. If $P$ is represented by a column, the replacement vector is the uniform probability vector on $N_G^-(u_P)\setminus\{y_P\}$ if $a_P\ge N/2$, and on $V(G)\setminus\{y_P\}$ otherwise. Let $h_P=\max\{a_P,b_P\}$ if this maximum is at least $N/2$, and let $h_P=N$ otherwise. Each such vector has between $N/2-1$ and $h_P$ nonzero coordinates, so its positive entries are therefore at least $h_P^{-1}$ and at most $3/N$ for all sufficiently large $N$.

    Let us briefly foreshadow why the replacement vectors are chosen this way. We need each entry in these vectors to be $O(1/N)$. If $P$ is a single vertex, then $a_P+b_P = N$. By including a vertex of the smaller of $N_G^+(u_P)$ and $N_G^-(u_P)$ in the pattern, the other neighborhood ranges over a set of size at least $N/2$, and hence the corresponding replacement vector has entries around $2/N$. However, if $P$ contains some edge of $F$, then $N_G^-(u_P)$ and $N_G^+(v_P)$ need not partition $V(G)$. It is therefore possible that both $a_P$ and $b_P$ are smaller than $N/2$, potentially even $o(N)$. In this case, we resign ourselves to setting the replacement vector to be uniform over all of $V(G)\setminus\{x_P\}$ or $V(G)\setminus\{y_P\}$. This amounts to permitting transitions which do not correspond to Hamiltonian cycles of $T$ (we end up overcounting), and the associated price of this relaxation later will be scaling $\cyc Q$ by a factor of $N$. However, we will see that components containing edges of $F$ will contribute enough additional decay to absorb this loss (this can be viewed as a consequence of the fact that Hamiltonian cycles are unlikely to use an edge in $S$, since $|S|$ is very small).

    Let $Q_\pi$ be obtained from $Q_0$ by the prescribed replacements in rows $R_\pi$ and columns $C_\pi$, except that every entry in $R_\pi\times C_\pi$ is zero. Outside these rows and columns, let $Q_\pi$ agree with $Q_0$. Since $|R_\pi|+|C_\pi|=k\le s$, this is a $3$-sparse modification of $Q_0$ in the sense of Definition \ref{def:sparse modifications} with zero diagonal.

    Deleting $S$ from a Hamiltonian cycle with pattern $\pi$ leaves a cyclic ordering of $V(G)$, in which each path $P\in c(F)$ becomes a single transition. No transition in this ordering lies in $R_\pi\times C_\pi$. Indeed, such a transition from $x_P$ to $y_{P'}$, say, would have to represent both $P$ and $P'$. As a transition through $S$ contains exactly one component of $F$, this would force $P = P'$, contrary to the rule that a component is represented by either a row or a column, but not both. (This justifies setting the entries of $R_\pi\times C_\pi$ to be zero in $Q_\pi$.) Conversely, for a cyclic ordering arising from a Hamiltonian cycle with pattern $\pi$, the original Hamiltonian cycle is uniquely determined by $\pi$ and the cyclic order. In particular, every path $P\in c(F)$ is ``reinserted'' immediately after its recorded predecessor or immediately before its recorded successor to reconstruct the Hamiltonian cycle. The $k$ transitions through $S$ have weights at least $h_P^{-1}$ in $Q_\pi$, and each of the remaining $N - k$ transitions is an edge of $G$ outside the modified rows and columns, and hence has weight $2/N$. Therefore, the number of Hamiltonian cycles of $T$ with pattern $\pi$ satisfies
    \[
        C(T,\pi)\le\left(\frac N2\right)^{N-k}\left(\prod_{P\in c(F)}h_P\right)\cyc Q_\pi.
    \]
    
    Applying Proposition \ref{prop:modifications} to $Q_\pi$ with $A_0 = 3$, we have
    \[
        \cyc Q_\pi\le (1+o(1))\frac{(N-1)!}{N^N}\exp\left(1+\sum_{x,y\in V(G)} (Q_\pi)_{xy} - N\right)Z(K).
    \]
    We next evaluate the exponential factor in this expression. Let $\delta = K\one_N$. The sum of the entries of $Q_0$ is $N-1$, its row sums are $1 - 1/N + \delta_x$, and its column sums are $1 - 1/N-\delta_y$. Replacing a row of $Q_0$ by a probability vector changes the total sum of its entries by $1/N-\delta_x$, and replacing a column changes it by $1/N+\delta_y$. Summing these changes over all rows and columns in $R_\pi$ and $C_\pi$ accounts for every entry of $Q_\pi$, except those which lie in $R_\pi\times C_\pi$. There are at most $s^2$ such entries, and the discrepancy at each is $O(1/N)$, so 
    \[
        1 + \sum_{x,y\in V(G)}(Q_\pi)_{xy} - N = \frac kN-\sum_{x\in R_\pi}\delta_x + \sum_{y\in C_\pi}\delta_y + O(s^2/N).
    \]
    As $k/N + O(s^2/N) = o(1)$, plugging in the bound for $\cyc Q_\pi$ into the bound for $C(T,\pi)$ gives
    \[
        C(T,\pi)\le(1+o(1))Z(K)\frac{(N-1)!}{2^N}\exp\left(-\sum_{x\in R_\pi}\delta_x+\sum_{y\in C_\pi}\delta_y\right)\prod_{P\in c(F)}\frac{2h_P}N.
    \]
    Moreover, the $(1+o(1))$ error is uniform over $\mathcal P$, as all the modifications use at most $s$ rows and columns and the same bound $3/N$ on their entries.

    We now wish to compute $\sum_{\pi \in\mathcal P_F}C(T,\pi)$, the sum  of this bound over the patterns $\pi$ sharing a fixed forest $F$. The exponential factor $\exp(-\sum_{x\in R_\pi}\delta_x+\sum_{y\in C_\pi}\delta_y)$ need not be close to one for each pattern, so we retain it in the sum. Its contribution from a component represented by a row is $e^{-\delta_{x_P}}$, and from a component represented by a column it is $e^{\delta_{y_P}}$. For $P\in c(F)$, define its ``weight'' as
    \[
        w_P = \begin{cases}\frac{2h_P}N\sum_{x\in N_G^-(u_P)}e^{-\delta_x} & a_P\le b_P\\ \frac{2h_P}{N}\sum_{y\in N_G^+(v_P)}e^{\delta_y} & a_P>b_P \end{cases}.
    \]
    Now note that the only constraints between the recorded vertices of different components are the distinctness requirements. Relaxing these requirements can only increase the sum of our nonnegative bounds, and the choices then become independent between different components. Moreover, $h_P$ depends only on $P$, so we have
    \[
        \sum_{\pi \in\mathcal P_F}C(T,\pi)\le (1+o(1))Z(K)\frac{(N-1)!}{2^N}\prod_{P\in c(F)}w_P.
    \]
    To estimate the product of the weights $w_P$, let $L = \norm\delta_1$. The hypothesis gives $\norm\delta_\infty^3\le\norm\delta_\infty\norm\delta_2^2=o(1)$, while H\"older's inequality gives
    \[
        L^3\le N^2\sum_{x\in V(G)}|\delta_x|^3\le N^2\norm\delta_\infty\norm\delta_2^2=o(N^2).
    \]
    Thus $\norm\delta_\infty=o(1)$ and $L=o(N^{2/3})$, so in particular $sL/N=o(1)$. For any $U\subseteq V(G)$, the bound $|e^{\pm\delta_x}-1|\le2|\delta_x|$ for $N$ sufficiently large yields $\sum_{x\in U}e^{\pm\delta_x}\le |U|+2L$. If $P$ is a single vertex, then $h_P = \max\{a_P,b_P\}$ and $a_P+b_P=N$, so $w_P\le 2a_Pb_P/N+4L\le N/2+4L$. For every other component, the bounds $h_P\le N$ and $\min\{a_P,b_P\}\le N$ give $w_P\le 2N+4L$.

    Let $r = |E(F)|$. Since $F$ is a forest with $k$ components, we have $k = s - r$ and at most $r$ components of $F$ contain an edge. Each component consisting of a single vertex contributes at most $(N/2)(1+8L/N)$, and every other component contributes at most $4(N/2)(1+8L/N)$. Multiplying these bounds and using the approximation $1+8L/N\le e^{8L/N}$, we have
    \[
        \prod_{P\in c(F)}w_P\le\left(\frac N2\right)^k4^re^{8sL/N}=(1+o(1))\left(\frac N2\right)^s\left(\frac8N\right)^r,
    \]
    where the error is again uniform over $\mathcal F$.

    It remains to sum over the forests $F\in\mathcal F$. The component estimate above gives a factor $8/N$ for each edge of $F$. Let $m = \binom s2$, and for $0\le r\le m$ let $\mathcal F_r$ be the forests in $\mathcal F$ with $r$ edges. Since $T[S]$ has exactly $m$ directed edges and a forest on the fixed vertex set $S$ is uniquely determined by its edge set, we have $|\mathcal F_r|\le\binom mr$. Note that not every choice of $r$ edges corresponds to a forest where every component is a path, but again relaxing this restriction can only increase the count. When $r = 0$, there is exactly one forest, where every path is a singleton.

    Every pattern belongs to exactly one $\mathcal P_F$, and every forest belongs to exactly one $\mathcal F_r$. Summing the bounds for a fixed forest and then applying the component estimate above gives
    \begin{align*}
        C(T)&=\sum_{\pi\in\mathcal P}C(T,\pi)\\
        &=\sum_{r=0}^m\sum_{F\in\mathcal F_r}\sum_{\pi\in\mathcal P_F}C(T,\pi)\\
        &\le (1+o(1))Z(K)\frac{(N-1)!}{2^N}\left(\frac N2\right)^s\sum_{r=0}^m|\mathcal F_r|\left(\frac 8N\right)^r\\
        &\le (1+o(1))Z(K)\frac{(N-1)!N^s}{2^n}\sum_{r=0}^m\binom mr\left(\frac8N\right)^r.
    \end{align*}
    By the binomial theorem, the summation factor is $(1+8/N)^m$, and since $m = O(s^2)=o(N)$, we have $1\le (1+8/N)^m\le e^{8m/N}=1+o(1)$. (For $r =0$ the term is already $1$, so the terms with $r\ge 1$ in this sum therefore had total $o(1)$; thus the total contribution from forests with at least one edge was negligible, as predicted.)

    Finally, $(n-1)!/(N-1)!=\prod_{j=0}^{s-1}(N+j)\ge N^s$, and Proposition \ref{prop:Z(K) bound} gives $Z(K) = O(1)$. Therefore,
    \[
        C(T)\le(Z(K)+o(1))\frac{(n-1)!}{2^n} = \left(\frac{\det(I-K)}{\det(I-iK)}+o(1)\right)\frac{(n-1)!}{2^n},
    \]
    which completes the proof of the proposition. 
\end{proof}

\begin{proof}[Proof of Theorem \ref{thm:upper}.]
    Let $T$ be a tournament on $n$ vertices. If $C(T)<(n-1)!/2^n$ then the conclusion is immediate because $\cosh1/\cos1>1$. We may therefore assume that $C(T)\ge (n-1)!/2^n$. By Proposition \ref{prop:polylog set}, there exists a set $S\subseteq V(T)$ with $|S|=O(\log^4n)$ such that, if $G = T - S$ has $N$ vertices and normalized Seidel matrix $K$, we have $\norm{K\one_N}_\infty\norm{K\one_N}_2^2=o(1)$ and $\norm K_\op\le2/\pi$. Let $\rho$ be a fixed constant satisfying $2/\pi<\rho<1$. For all sufficiently large $n$, the hypotheses of Proposition \ref{prop:insertions} hold, and hence
    \[
        C(T)\le(Z(K)+o(1))\frac{(n-1)!}{2^n}.
    \]
    Proposition \ref{prop:Z(K) bound} gives $Z(K)\le\cosh1/\cos1+o(1)$. Therefore,
    \[
        C(T)\le\left(\frac{\cosh1}{\cos1}+o(1)\right)\frac{(n-1)!}{2^n},
    \]
    which proves Theorem \ref{thm:upper}.
\end{proof}

\section{Proof of the lower bound}\label{sec:lower}

Recall that for $n$ odd, $T_n$ is the tournament with vertex set $\mathbb Z/n\mathbb Z$ and a directed edge from $i$ to $j$ if $j-i\pmod n<n/2$. If $n$ is even, let $R_n$ be the tournament obtained from $T_{n+1}$ by deleting the vertex $0$; and if $n$ is odd, let $R_n = T_n$. In this section, we first prove Theorem \ref{thm:lower}, which counts the number of Hamiltonian cycles in the tournaments $R_n$. This count will asymptotically match the upper bound on $C(n)$ established in the previous section. The crucial property of $T_n$ which we exploit is that its normalized tournament Seidel matrix asymptotically attains the value of $\cosh1/\cos 1$ in Proposition \ref{prop:Z(K) bound}. The tournament $R_n$ for $n$ even is introduced solely because it also inherits this property from $T_{n+1}$. We remark that, if $K$ is the normalized Seidel matrix of the uniform random tournament on $n$ vertices, then one can show $\lim_{n\to\infty}\mathbb E[Z(K)] = e < \cosh1/\cos 1$. 

We will then adapt the counting strategy slightly in order to show that the quantity $P(n)$ is also asymptotically maximized by $P(R_n)$. That is, we will show $P(R_n) = (\frac{\cosh1}{\cos1}-o(1))\frac{n!}{2^{n-1}}$. Together with Theorem \ref{thm:upper} and the inequality $P(n)\le 4C(n+1)$, this completes the proof of Theorem \ref{thm:Hamiltonian paths}.

We begin with the following lemma.

\begin{lemma}\label{lem:9.1}
    If $K$ is the normalized Seidel matrix of $R_n$, then for every $z\in \mathbb C$,
    \[
        \det(I-zK)=\sum_{j=0}^{\lfloor n/2\rfloor}\binom{n}{2j}\frac{z^{2j}}{n^{2j}}.
    \]
    In particular, $Z(K)\to\cosh1/\cos1$ as $n\to\infty$.
\end{lemma}

\begin{proof}
    Our strategy will be to show that the Seidel matrix of $R_n$ is obtained from the Seidel matrix $U_n$ of the transitive tournament by conjugation with a signed permutation matrix. This process preserves the determinants in the statement, so our previous calculation for $Z(U_n/n)$ would complete the proof. In brief, we may justify the existence of such a signed permutation matrix as follows: arranging the vertices of $T_n$ in the (natural) cyclic order and then reversing the directed edges which cross a diametrical cut results in the transitive tournament on $n$ vertices. 

    More concretely, let $m\ge 0$, and consider $T_{2m+1}$ with $V(T_{2m+1})=\{0, 1, \dots, 2m\}$. Multiply the rows and columns of its Seidel matrix which are indexed by the set $\{m+1, \dots, 2m\}$ by $-1$. This reverses precisely the edges between this set of vertices and $\{0,\dots, m\}$. The resulting tournament is transitive, and has the transitive order
    \[
        0, m+1, 1, m+2, \dots, m-1, 2m, m.
    \]
    Indeed, edges within either set follow this order. For an edge between $i\in\{0,\dots, m\}$ and $m + j$ with $1 \le j \le m$, the original orientation is $i\to m+j$ whenever $j \le i$. After reversal, this now becomes $i \to m + j$ whenever $i < j$, which matches the above order.

    This proves the assertion when $n$ is odd. In the case that $n = 2m$ for some $m\ge1$, recall that $R_n$ is obtained from $T_{2m+1}$ by deleting vertex $0$. Deleting that vertex from the same construction above leaves a transitive tournament on $n$ vertices, and hence the assertion also holds when $n$ is even. 

    Multiplying the rows and columns of the Seidel matrix of $R_n$ as described above corresponds to conjugating this matrix by a signed permutation matrix. As this preserves determinants, we have $\det(I - zK) = \det(I - zU_n/n)$, and the result now follows from the proof of Proposition \ref{prop:Z(K) bound}. 
\end{proof}

We can now prove Theorem \ref{thm:lower} by applying Proposition \ref{prop:cyc bound} and using Lemma \ref{lem:9.1} in the final step. Unfortunately, for $n$ even, the tournament $R_n$ is not a regular tournament (and hence its adjacency matrix need not be a scalar multiple of a doubly stochastic one), so we will again need to apply Sinkhorn scaling and show that the corresponding scaling factors are negligible. This involves verifying the hypotheses for Propositions \ref{prop:sinkhorn} and \ref{prop:cyc bound}.

\begin{proof}[Proof of Theorem \ref{thm:lower}]
    Let $A$ be the adjacency matrix of $R_n$, and let $K = (A - A^T)/n$ be its normalized Seidel matrix. As before, set $Q = 2A/n = J_n-I/n+K$. We will apply Proposition \ref{prop:sinkhorn} to scale $Q$ to a doubly stochastic matrix.

    Let $u = Q\one_n-\one_n$ and $v = Q^T\one_n - \one_n$. If $n$ is odd, every row and column sum of $Q$ is $1-1/n$. If $n$ is even, these sums are either $1$ or $1 - 2/n$, since deleting a vertex from $T_{n+1}$ leaves every degree equal to $n/2$ or $n/2-1$. Therefore, $\norm u_\infty+\norm v_\infty = O(n^{-1})$ and $\norm u_2+\norm v_2 = O(n^{-1/2})$. 

    We now check the remaining hypothesis of Proposition \ref{prop:sinkhorn}. Let $P = \diag(\one_n+u)^{-1/2} Q \diag(\one_n+v)^{-1/2}$. Since $\norm u_\infty + \norm v_\infty=O(n^{-1})$, the estimate shows that both of the diagonal factors in $P$ differ in spectral norm from $I$ by order $O(n^{-1})$. Moreover, every row and column sum of $Q$ is at most $1$, so $\norm Q_\op\le 1$. Consequently, 
    \begin{align*}
        \norm{P-Q}_\op&\le\norm{\diag(\one_n+u)^{-1/2}-I}_\op\cdot\norm Q_\op\cdot\norm{\diag(\one_n+v)^{-1/2}}_\op \\&\phantom{aa}+ \norm Q_\op\cdot \norm{\diag(\one_n+v)^{-1/2}-I}_\op = O(n^{-1}).
    \end{align*}
    Since $J_n$ vanishes on $\one_n^\perp$, we have $Qz = (K-I/n)z$ for every $z\perp\one_n$. Therefore,
    \begin{align*}
        \sigma_2(P)&\le \sup_{\substack{z\perp\one_n\\\norm z_2=1}}\norm{Pz}_2 \le \norm K_\op + \frac1n + \norm{P-Q}_\op \\&\le \frac1n\cot\frac\pi{2n}+O(n^{-1})=  \frac2\pi+o(1),
    \end{align*}
    where the final inequality follows by applying Lemma \ref{lem:Seidel norm bound} to $nK$.

    Since $\norm Q_\mx\le 2/n$, all the hypotheses of Proposition \ref{prop:sinkhorn} are satisfied. Thus there exist vectors $x,y\in\mathbb R^n$ such that $D = \diag(e^x)Q\diag(e^y)$ is doubly stochastic and $\norm x_2+\norm y_2=O(n^{-1/2})$. In particular, $\norm x_\infty + \norm y_\infty=o(1)$, so in the forthcoming computation we may Taylor expand $e^{x_i+y_j}$ uniformly over all $i,j$.

    We first determine how this scaling affects the ratio between $\cyc Q$ and $\cyc D$. Let $s = \sum_{i=1}^n(x_i+y_i)$. The sum of the entries of $Q$ is $n-1$, whereas the sum of the entries of $D$ is $n$. Expanding $e^{x_i+y_j}$ and using the Euclidean norm bounds above, we have
    \[
        n = \sum_{1\le i,j\le n} Q_{ij}e^{x_i+y_j} = n-1 + s + O(n^{-1}).
    \]
    Indeed, the linear term is $s + \langle u,x\rangle+\langle v,y\rangle$, and the additional inner products and quadratic remainder are all of order $O(n^{-1})$, by the Euclidean norm bounds established earlier. Therefore, $s = 1 + O(n^{-1})$.

    We next verify the hypotheses of Proposition \ref{prop:cyc bound}. Let $B = D - J_n$. As in the proof of Proposition \ref{prop:modifications}, the bounds $Q_{ij}=O(n^{-1})$ and $\max_{1\le i,j\le n}|e^{x_i+y_j}-1|=o(1)$ imply that $\norm{D-Q}_F = o(1)$. Since $Q - J_n=K-I/n$, it follows that $\norm{B-K}_F=o(1)$. Consequently, $B$ has bounded Frobenius norm and $\norm B_\op\le2/\pi+o(1)$. Moreover, $D$ has zero diagonal and entries of order $O(n^{-1})$, so Proposition \ref{prop:cyc bound} applies.

    Since $\tr B=-1$ and $\tr K=0$, the two determinant comparisons at the end of the proof of Proposition \ref{prop:modifications} apply without change. Together with Lemma \ref{lem:skew-sym det}, this implies
    \[
        \frac{n^n\cyc D}{(n-1)!}=(1+o(1))\frac{\det(I-B)}{\sqrt{\det(I-B^TB)}}=(1+o(1))eZ(K).
    \]
    Finally, every term contributing to $\cyc Q$ contains exactly one entry from each row and column. Hence $\cyc(Q)=e^{-s}\cyc D$. Since $A = (n/2)Q$, we conclude that
    \[
        \frac{C(R_n)}{(n-1)!/2^n} = e^{-s}\frac{n^n\cyc D}{(n-1)!} = (1+o(1))e^{1-s}Z(K)\longrightarrow\frac{\cosh1}{\cos 1},
    \]
    where the limit follows from $s = 1+O(n^{-1})$ and Lemma \ref{lem:9.1}. This completes the proof of Theorem \ref{thm:lower}.
\end{proof}

\section{Hamiltonian paths}\label{sec:Hamiltonian paths}

In this section, we will prove Theorem \ref{thm:Hamiltonian paths}. The upper bound follows from $P(n) \le 4C(n+1)$, which in turn can be seen by adding a vertex with uniformly random adjacencies, and then counting Hamiltonian cycles in the resulting tournament on $n + 1$ vertices. Paralleling this, our strategy to count Hamiltonian paths in $R_n$ will be to add a new vertex which is adjacent in \textit{both directions} to every vertex in $R_n$, and then count the Hamiltonian cycles in the resulting directed graph. Despite not being a tournament, the normalized adjacency matrix of this directed graph can still be weighted to become doubly stochastic, so we will still be able to apply our Hamiltonian-cycle counting machinery.

\begin{proof}[Proof of Theorem \ref{thm:Hamiltonian paths}]
    The upper bound follows immediately from $P(n)\le4C(n+1)$ and Theorem \ref{thm:upper}. For the lower bound, we will show that $P(R_n)=(2n+o(n))C(R_n)$.

    If $M$ is an $n\times n$ matrix, we will define $\mathrm{path}(M)$ to be the weighted sum
    \[
        \mathrm{path}(M)=\sum_{\sigma\in S_n}\prod_{j=1}^{n-1}M_{\sigma(j),\sigma(j+1)}.
    \]
    Let $A,Q,D,x,y$ be defined as in the proof of Theorem \ref{thm:lower}. Thus $A = (n/2)Q$, $D = \diag(e^x)Q\diag(e^y)$ is doubly stochastic, and $\norm x_\infty+\norm y_\infty=o(1)$. Recall also that $D$ satisfies the hypotheses of Proposition \ref{prop:cyc bound}.

    We first compare $\mathrm{path}(D)$ with $\cyc D$. In effect, this will be done by introducing an auxiliary vertex $v$ as described before, then counting Hamiltonian cycles in the new directed graph. To this end, consider the $(n+1)\times(n+1)$ matrix
    \[
        \widehat D = \begin{bmatrix}0&\one_n^T/n\\\one_n/n&(1-1/n)D\end{bmatrix}.
    \]
    This matrix is doubly stochastic: the new row and column each sum to one, while the contribution of the new vertex $v$ (corresponding to the first row and column) is exactly compensated by the factor $1-1/n$. We will abuse notation and refer to Hamiltonian paths and cycles in $D$ and $\widehat D$, which really describe the Hamiltonian paths and cycles in the directed graphs these matrices encode. 
    
    Every Hamiltonian path in $D$ gives a unique Hamiltonian cycle in $\widehat D$, formed by closing the Hamiltonian path through this new vertex. The two directed edges going into and out of $v$ contribute a factor of $n^{-2}$ to $\cyc\widehat D$, while each of the $n-1$ original edges contributes an additional factor of $1-1/n$. Consequently,
    \[
        \cyc\widehat D = \frac1{n^2}\left(1-\frac1n\right)^{n-1}\mathrm{path}(D).
    \]

    We now apply Proposition \ref{prop:cyc bound} to both $D$ and $\widehat D$. Let $B = D - J_n$ and $\widehat B = \widehat D - J_{n+1}$. A direct computation of the blocks of these matrices yields
    \[
        \left\|\widehat B - \begin{bmatrix}0&0\\0&B\end{bmatrix}\right\|
        _F=O(n^{-1}).
    \]
    Since $B$ has bounded Frobenius norm and spectral norm bounded away from one, the same holds for $\widehat B$. Both matrices also have trace equal to $-1$, since $D$ and $\widehat D$ have zero diagonal. Thus, we may apply the same determinant comparisons as those used at the end of the proof of Proposition \ref{prop:modifications}, which show that the determinant factors associated with $B$ and $\widehat B$ differ by a factor $1+o(1)$. By Proposition \ref{prop:cyc bound}, we now have
    \[
        \frac{\cyc\widehat D}{\cyc D}=(1+o(1))\frac{n!\cdot n^n}{(n+1)^{n+1}(n-1)!} = (1+o(1))\left(\frac n{n+1}\right)^{n+1}.
    \]
    Since both $(n/(n+1))^{n+1}$ and $(1-1/n)^{n-1}$ tend to $e^{-1}$ as $n\to\infty$, the preceding identity yields
    \[
        \mathrm{path}(D) = (1+o(1))n^2\cyc D.
    \]

    It only remains to undo the Sinkhorn scaling between $D$ and $Q$. Let $s = \sum_{i=1}^n(x_i+y_i)$, as in the proof of Theorem \ref{thm:lower}. Every Hamiltonian cycle counted by $\cyc D$ acquires a factor of $e^s$ under this scaling. A Hamiltonian path beginning at $v_1$ and ending at $v_n$, on the other hand, acquires the factor $e^{s-x_{v_n}-y_{v_1}}$. Since $\norm x_\infty+\norm y_\infty=o(1)$, this is $(1+o(1))e^{s}$ uniformly over all the Hamiltonian paths. It follows that
    \[
        \mathrm{path}(Q)=(1+o(1))n^2\cyc Q.
    \]

    Finally, using $A = (n/2)Q$, we obtain
    \[
        P(R_n)=\left(\frac n2\right)^{n-1}\mathrm{path}(Q) = (2n+o(n))\left(\frac n2\right)^n\cyc Q = (2n+o(n))C(R_n).
    \]
    By Theorems \ref{thm:lower} and \ref{thm:upper}, we already know that
    \[
        C(R_n) = \left(\frac{\cosh1}{\cos1}+o(1)\right)\frac{(n-1)!}{2^n},
    \]
    and therefore
    \[
        P(n)\ge P(R_n) = \left(\frac{\cosh1}{\cos1}+o(1)\right)\frac{n!}{2^{n-1}},
    \]
    which completes the proof.
\end{proof}

\subsection*{Acknowledgments} We thank Noga Alon for many helpful discussions and comments. We would also like to thank Jay Pantone, who used an impressive computer program to guess based on numerics that the leading constant for $C(n)$ could be $\cosh 1/\cos 1$, and communicated these findings with Alon. 

This project began as an attempt to realize this constant in an improved lower bound for $C(n)$, a problem which seemed promising for AI models to tackle since it essentially consists of counting Hamiltonian cycles in the explicit tournament $T_n$. A proof was found by GPT-5.6 following some strategies supplied by the authors (the relevant one being a suggestion to search hard for suitable permanent estimates). In a surprising turn of events, after being input this proof and a few hopeful ideas, a separate instance of GPT-5.6 was able to generalize its techniques into an argument for the matching upper bound.

The purpose of this paper is to provide an exposition of this proof, particularly to explain the high-level strategies and intuition for the new key ideas. Many of the arguments were eventually replaced or simplified by the authors, and others originally came from the literature. The content of the present paper was written by the authors.

\bibliographystyle{amsplain}
\bibliography{references}

@article {Szele1943,
    AUTHOR = {Szele, Tibor},
     TITLE = {Kombinatorische Untersuchungen \"uber den gerichteten vollst\"andigen Graphen},
   JOURNAL = {Mat. Fiz. Lapok},
  FJOURNAL = {Publicationes Mathematicae Debrecen},
    VOLUME = {50},
      YEAR = {1943},
     PAGES = {223--256},
}

@article {Alon1990,
    AUTHOR = {Alon, Noga},
     TITLE = {The maximum number of {H}amiltonian paths in tournaments},
   JOURNAL = {Combinatorica},
  FJOURNAL = {Combinatorica. An International Journal on Combinatorics and
              the Theory of Computing},
    VOLUME = {10},
      YEAR = {1990},
    NUMBER = {4},
     PAGES = {319--324},
      ISSN = {0209-9683},
   MRCLASS = {05C20 (05C45 15A15)},
  MRNUMBER = {1099246},
MRREVIEWER = {Thomas\ H.\ Foregger},
       DOI = {10.1007/BF02128667},
       URL = {https://doi.org/10.1007/BF02128667},
}

@article {AdlerAlonRoss2001,
    AUTHOR = {Adler, Ilan and Alon, Noga and Ross, Sheldon M.},
     TITLE = {On the maximum number of {H}amiltonian paths in tournaments},
   JOURNAL = {Random Structures Algorithms},
  FJOURNAL = {Random Structures \& Algorithms},
    VOLUME = {18},
      YEAR = {2001},
    NUMBER = {3},
     PAGES = {291--296},
      ISSN = {1042-9832,1098-2418},
   MRCLASS = {05C45 (05C20)},
  MRNUMBER = {1824278},
       DOI = {10.1002/rsa.1010},
       URL = {https://doi.org/10.1002/rsa.1010},
}

@article {FriedgutKahn2005,
    AUTHOR = {Friedgut, Ehud and Kahn, Jeff},
     TITLE = {On the number of {H}amiltonian cycles in a tournament},
   JOURNAL = {Combin. Probab. Comput.},
  FJOURNAL = {Combinatorics, Probability and Computing},
    VOLUME = {14},
      YEAR = {2005},
    NUMBER = {5-6},
     PAGES = {769--781},
      ISSN = {0963-5483,1469-2163},
   MRCLASS = {05C45 (05C20)},
  MRNUMBER = {2174655},
MRREVIEWER = {Brian\ Alspach},
       DOI = {10.1017/S0963548305006863},
       URL = {https://doi.org/10.1017/S0963548305006863},
}

@article {Cuckler2007,
    AUTHOR = {Cuckler, Bill},
     TITLE = {Hamiltonian cycles in regular tournaments},
   JOURNAL = {Combin. Probab. Comput.},
  FJOURNAL = {Combinatorics, Probability and Computing},
    VOLUME = {16},
      YEAR = {2007},
    NUMBER = {2},
     PAGES = {239--249},
      ISSN = {0963-5483,1469-2163},
   MRCLASS = {05C45 (05C20)},
  MRNUMBER = {2298811},
MRREVIEWER = {J.\ W.\ Moon},
       DOI = {10.1017/S0963548306008170},
       URL = {https://doi.org/10.1017/S0963548306008170},
}

@misc{Wormald,
      title={Tournaments with many Hamilton cycles}, 
      author={N. C. Wormald},
      year={2006},
      eprint={2610.16245},
      archivePrefix={arXiv},
      primaryClass={math.CO},
      url={https://arxiv.org/abs/2610.16245}, 
}

@article {Egorychev1981,
    AUTHOR = {Egorychev, G. P.},
     TITLE = {Proof of the van der {W}aerden conjecture for permanents},
   JOURNAL = {Sibirsk. Mat. Zh.},
  FJOURNAL = {Akademiya Nauk SSSR. Sibirskoe Otdelenie. Sibirski\u i\
              Matematicheski\u i\ Zhurnal},
    VOLUME = {22},
      YEAR = {1981},
    NUMBER = {6},
     PAGES = {65--71, 225},
      ISSN = {0037-4474},
   MRCLASS = {15A15},
  MRNUMBER = {638007},
MRREVIEWER = {S.\ Lajos},
}

@article {Falikman1981,
    AUTHOR = {Falikman, D. I.},
     TITLE = {Proof of the van der {W}aerden conjecture on the permanent of
              a doubly stochastic matrix},
   JOURNAL = {Mat. Zametki},
  FJOURNAL = {Akademiya Nauk SSSR. Matematicheskie Zametki},
    VOLUME = {29},
      YEAR = {1981},
    NUMBER = {6},
     PAGES = {931--938, 957},
      ISSN = {0025-567X},
   MRCLASS = {15A15 (15A51)},
  MRNUMBER = {625097},
MRREVIEWER = {Anna\ Lee},
}

@article {Bregman1973,
    AUTHOR = {Br\`egman, L. M.},
     TITLE = {Certain properties of nonnegative matrices and their
              permanents},
   JOURNAL = {Dokl. Akad. Nauk SSSR},
  FJOURNAL = {Doklady Akademii Nauk SSSR},
    VOLUME = {211},
      YEAR = {1973},
     PAGES = {27--30},
      ISSN = {0002-3264},
   MRCLASS = {15A15},
  MRNUMBER = {327788},
MRREVIEWER = {E.\ Seneta},
}

@article {McCullagh2014,
    AUTHOR = {McCullagh, Peter},
     TITLE = {An asymptotic approximation for the permanent of a doubly
              stochastic matrix},
   JOURNAL = {J. Stat. Comput. Simul.},
  FJOURNAL = {Journal of Statistical Computation and Simulation},
    VOLUME = {84},
      YEAR = {2014},
    NUMBER = {2},
     PAGES = {404--414},
      ISSN = {0094-9655,1563-5163},
   MRCLASS = {15B51 (15A15 60J22 65F30)},
  MRNUMBER = {3169336},
MRREVIEWER = {Gi-Sang\ Cheon},
       DOI = {10.1080/00949655.2012.712122},
       URL = {https://doi.org/10.1080/00949655.2012.712122},
}

@article {Gurvits2008,
    AUTHOR = {Gurvits, Leonid},
     TITLE = {Van der {W}aerden/{S}chrijver-{V}aliant like conjectures and
              stable (aka hyperbolic) homogeneous polynomials: one theorem
              for all},
      NOTE = {With a corrigendum},
   JOURNAL = {Electron. J. Combin.},
  FJOURNAL = {Electronic Journal of Combinatorics},
    VOLUME = {15},
      YEAR = {2008},
    NUMBER = {1},
     PAGES = {Research Paper 66, 26},
      ISSN = {1077-8926},
   MRCLASS = {15A15 (05A15 05C70)},
  MRNUMBER = {2411443},
MRREVIEWER = {Peter\ M.\ Gibson},
       DOI = {10.37236/790},
       URL = {https://doi.org/10.37236/790},
}

@book {AlonSpencer2016,
    AUTHOR = {Alon, Noga and Spencer, Joel H.},
     TITLE = {The probabilistic method},
    SERIES = {Wiley Series in Discrete Mathematics and Optimization},
   EDITION = {Fourth},
 PUBLISHER = {John Wiley \& Sons, Inc., Hoboken, NJ},
      YEAR = {2016},
     PAGES = {xiv+375},
      ISBN = {978-1-119-06195-3},
   MRCLASS = {60-02 (05C80 05D40 60C05 60F10 60G42)},
  MRNUMBER = {3524748},
}

@book {ErdosSpencer1974,
    AUTHOR = {Erd\H{o}s, Paul and Spencer, Joel},
     TITLE = {Probabilistic methods in combinatorics},
    SERIES = {Probability and Mathematical Statistics},
    VOLUME = {Vol. 17},
 PUBLISHER = {Academic Press [Harcourt Brace Jovanovich, Publishers], New
              York-London},
      YEAR = {1974},
     PAGES = {106},
   MRCLASS = {05-02},
  MRNUMBER = {382007},
MRREVIEWER = {P.\ E.\ O'Neil},
}

@book {Spencer1994,
    AUTHOR = {Spencer, Joel},
     TITLE = {Ten lectures on the probabilistic method},
    SERIES = {CBMS-NSF Regional Conference Series in Applied Mathematics},
    VOLUME = {64},
   EDITION = {Second},
 PUBLISHER = {Society for Industrial and Applied Mathematics (SIAM),
              Philadelphia, PA},
      YEAR = {1994},
     PAGES = {vi+88},
      ISBN = {0-89871-325-0},
   MRCLASS = {05C80 (05C55 60C05)},
  MRNUMBER = {1249485},
MRREVIEWER = {Andrzej\ Ruci\'nski},
       DOI = {10.1137/1.9781611970074},
       URL = {https://doi.org/10.1137/1.9781611970074},
}

@incollection {Thomassen1982,
    AUTHOR = {Thomassen, Carsten},
     TITLE = {Hamilton circuits in regular tournaments},
 BOOKTITLE = {Cycles in graphs ({B}urnaby, {B}.{C}., 1982)},
    SERIES = {North-Holland Math. Stud.},
    VOLUME = {115},
     PAGES = {159--162},
 PUBLISHER = {North-Holland, Amsterdam},
      YEAR = {1985},
      ISBN = {0-444-87803-3},
   MRCLASS = {05C45 (05C20)},
  MRNUMBER = {821514},
MRREVIEWER = {R.\ C.\ Entringer},
       DOI = {10.1016/S0304-0208(08)73005-6},
       URL = {https://doi.org/10.1016/S0304-0208(08)73005-6},
}

@article {FerberKrivelevichSudakov2017,
    AUTHOR = {Ferber, Asaf and Krivelevich, Michael and Sudakov, Benny},
     TITLE = {Counting and packing {H}amilton cycles in dense graphs and
              oriented graphs},
   JOURNAL = {J. Combin. Theory Ser. B},
  FJOURNAL = {Journal of Combinatorial Theory. Series B},
    VOLUME = {122},
      YEAR = {2017},
     PAGES = {196--220},
      ISSN = {0095-8956,1096-0902},
   MRCLASS = {05C45 (05C20 05C30 05C70)},
  MRNUMBER = {3575203},
MRREVIEWER = {Andrew\ Clark\ Treglown},
       DOI = {10.1016/j.jctb.2016.06.001},
       URL = {https://doi.org/10.1016/j.jctb.2016.06.001},
}

@misc{FangHuang2026,
      title={Maximizing directed cycles in tournaments}, 
      author={Yijia Fang and Hao Huang},
      year={2026},
      eprint={2608.06238},
      archivePrefix={arXiv},
      primaryClass={math.CO},
      url={https://arxiv.org/abs/2608.06238}, 
}

@article {KendallSmith1940,
    AUTHOR = {Kendall, M. G. and Babington Smith, B.},
     TITLE = {On the method of paired comparisons},
   JOURNAL = {Biometrika},
  FJOURNAL = {Biometrika},
    VOLUME = {31},
      YEAR = {1940},
     PAGES = {324--345},
      ISSN = {0006-3444,1464-3510},
   MRCLASS = {62.0X},
  MRNUMBER = {2761},
MRREVIEWER = {S.\ S.\ Wilks},
       DOI = {10.1093/biomet/31.3-4.324},
       URL = {https://doi.org/10.1093/biomet/31.3-4.324},
}

@article {Colombo1964,
    AUTHOR = {Colombo, Umberto},
     TITLE = {Sui circuiti nei grafi completi},
   JOURNAL = {Boll. Un. Mat. Ital. (3)},
  FJOURNAL = {Boll. Un. Mat. Ital. (3)},
    VOLUME = {19},
      YEAR = {1964},
     PAGES = {153--170},
   MRCLASS = {55.10 (05.60)},
  MRNUMBER = {172262},
MRREVIEWER = {R.\ C.\ Read},
}

@article {BeinekeHarary1965,
    AUTHOR = {Beineke, Lowell W. and Harary, Frank},
     TITLE = {The maximum number of strongly connected subtournaments},
   JOURNAL = {Canad. Math. Bull.},
  FJOURNAL = {Canadian Mathematical Bulletin. Bulletin Canadien de
              Math\'ematiques},
    VOLUME = {8},
      YEAR = {1965},
     PAGES = {491--498},
      ISSN = {0008-4395,1496-4287},
   MRCLASS = {05.65},
  MRNUMBER = {181581},
MRREVIEWER = {J.\ W.\ Moon},
       DOI = {10.4153/CMB-1965-035-x},
       URL = {https://doi.org/10.4153/CMB-1965-035-x},
}

@article {Savchenko2016,
    AUTHOR = {Savchenko, S. V.},
     TITLE = {On 5-cycles and 6-cycles in regular {$n$}-tournaments},
   JOURNAL = {J. Graph Theory},
  FJOURNAL = {Journal of Graph Theory},
    VOLUME = {83},
      YEAR = {2016},
    NUMBER = {1},
     PAGES = {44--77},
      ISSN = {0364-9024,1097-0118},
   MRCLASS = {05C20 (05C38 05C50)},
  MRNUMBER = {3529947},
MRREVIEWER = {Xiaogen\ Chen},
       DOI = {10.1002/jgt.21914},
       URL = {https://doi.org/10.1002/jgt.21914},
}

@article {KomarovMackey2017,
    AUTHOR = {Komarov, Natasha and Mackey, John},
     TITLE = {On the number of 5-cycles in a tournament},
   JOURNAL = {J. Graph Theory},
  FJOURNAL = {Journal of Graph Theory},
    VOLUME = {86},
      YEAR = {2017},
    NUMBER = {3},
     PAGES = {341--356},
      ISSN = {0364-9024,1097-0118},
   MRCLASS = {05C35 (05C20 05C38)},
  MRNUMBER = {3697928},
MRREVIEWER = {Guantao\ Chen},
       DOI = {10.1002/jgt.22130},
       URL = {https://doi.org/10.1002/jgt.22130},
}

@article {Grzesik2023,
    AUTHOR = {Grzesik, Andrzej and Kr\'al', Daniel and Lov\'asz,
              L\'aszl\'o{} M. and Volec, Jan},
     TITLE = {Cycles of a given length in tournaments},
   JOURNAL = {J. Combin. Theory Ser. B},
  FJOURNAL = {Journal of Combinatorial Theory. Series B},
    VOLUME = {158},
      YEAR = {2023},
     PAGES = {117--145},
      ISSN = {0095-8956,1096-0902},
   MRCLASS = {05C35 (05C20 05C38)},
  MRNUMBER = {4513820},
MRREVIEWER = {J.\ W.\ Moon},
       DOI = {10.1016/j.jctb.2022.07.007},
       URL = {https://doi.org/10.1016/j.jctb.2022.07.007},
}

@article {Bartley2018,
    AUTHOR = {Bartley, Jack P.},
    TITLE = {Topics in Extremal and Probabilistic Combinatorics, PhD Thesis},
    JOURNAL = {Queen Mary University of London},
    YEAR = {2018},
    URL = {https://qmro.qmul.ac.uk/xmlui/handle/123456789/56858},
}

@phdthesis {Day2017,
    AUTHOR = {Day, Alan N.},
    TITLE = {A Collection of Problems in Extremal Combinatorics},
    JOURNAL = {Queen Mary University of London},
    YEAR = {2017},
    URL = {https://qmro.qmul.ac.uk/xmlui/handle/123456789/36669}
}

@article {FoxHimwichManiZhou2025,
    AUTHOR = {Fox, Jacob and Himwich, Zoe and Mani, Nitya and Zhou, Yunkun},
     TITLE = {A note on directed analogues of the {S}idorenko and forcing
              conjectures},
   JOURNAL = {Electron. J. Combin.},
  FJOURNAL = {Electronic Journal of Combinatorics},
    VOLUME = {32},
      YEAR = {2025},
    NUMBER = {3},
     PAGES = {Paper No. 3.38, 14},
      ISSN = {1077-8926},
   MRCLASS = {05C20 (05D40)},
  MRNUMBER = {4958048},
       DOI = {10.37236/11675},
       URL = {https://doi.org/10.37236/11675},
}

@article {FoxHimwichManiZhou2026,
    AUTHOR = {Fox, Jacob and Himwich, Zoe and Mani, Nitya and Zhou, Yunkun},
     TITLE = {Variations on Sidorenko's conjecture in tournaments},
   JOURNAL = {J. Graph Theory, to appear},
   eprint={2402.08418}
}

@article {ZhaoZhou2020,
    AUTHOR = {Zhao, Yufei and Zhou, Yunkun},
     TITLE = {Impartial digraphs},
   JOURNAL = {Combinatorica},
  FJOURNAL = {Combinatorica. An International Journal on Combinatorics and
              the Theory of Computing},
    VOLUME = {40},
      YEAR = {2020},
    NUMBER = {6},
     PAGES = {875--896},
      ISSN = {0209-9683,1439-6912},
   MRCLASS = {05C20 (05C75)},
  MRNUMBER = {4208101},
MRREVIEWER = {C\'esar\ Hern\'andez-Cruz},
       DOI = {10.1007/s00493-020-4280-0},
       URL = {https://doi.org/10.1007/s00493-020-4280-0},
}

@article {SahSawhneyZhao2023,
    AUTHOR = {Sah, Ashwin and Sawhney, Mehtaab and Zhao, Yufei},
     TITLE = {Paths of given length in tournaments},
   JOURNAL = {Comb. Theory},
  FJOURNAL = {Combinatorial Theory},
    VOLUME = {3},
      YEAR = {2023},
    NUMBER = {2},
     PAGES = {Paper No. 5, 6},
      ISSN = {2766-1334},
   MRCLASS = {05C38 (05C20)},
  MRNUMBER = {4646086},
MRREVIEWER = {Qiaoping\ Guo},
       DOI = {10.5070/c63261981},
       URL = {https://doi.org/10.5070/c63261981},
}

@misc{HeManiNieTungWei2025,
      title={New Sidorenko-type inequalities in tournaments}, 
      author={Xiaoyu He and Nitya Mani and Jiaxi Nie and Nathan Tung and Fan Wei},
      year={2025},
      eprint={2512.11222},
      archivePrefix={arXiv},
      primaryClass={math.CO},
      url={https://arxiv.org/abs/2512.11222}, 
}

@misc{ChenClemenNoel2026,
      title={On Tournament Anti-Sidorenko Orientations of Trees}, 
      author={Hao Chen and Felix Christian Clemen and Jonathan A. Noel},
      year={2026},
      eprint={2605.14138},
      archivePrefix={arXiv},
      primaryClass={math.CO},
      url={https://arxiv.org/abs/2605.14138}, 
}

@article {Liu1991,
    AUTHOR = {Liu, C. J.},
     TITLE = {Enumeration of {H}amiltonian cycles and paths in a graph},
   JOURNAL = {Proc. Amer. Math. Soc.},
  FJOURNAL = {Proceedings of the American Mathematical Society},
    VOLUME = {111},
      YEAR = {1991},
    NUMBER = {1},
     PAGES = {289--296},
      ISSN = {0002-9939,1088-6826},
   MRCLASS = {05C30 (05C38 05C45 15A15 15A18)},
  MRNUMBER = {1028288},
MRREVIEWER = {Yong\ Bing\ Shi},
       DOI = {10.2307/2047891},
       URL = {https://doi.org/10.2307/2047891},
}

@misc{SawczukGnang2025,
      title={Two Proofs of the Hamiltonian Cycle Identity}, 
      author={Hamilton Sawczuk and Edinah Gnang},
      year={2025},
      eprint={2510.02473},
      archivePrefix={arXiv},
      primaryClass={math.CO},
      url={https://arxiv.org/abs/2510.02473}, 
}

@book {HornJohnson1991,
    AUTHOR = {Horn, Roger A. and Johnson, Charles R.},
     TITLE = {Topics in matrix analysis},
      NOTE = {Corrected reprint of the 1991 original},
 PUBLISHER = {Cambridge University Press, Cambridge},
      YEAR = {1994},
     PAGES = {viii+607},
      ISBN = {0-521-46713-6},
   MRCLASS = {15-02 (15-01 65-02 65Fxx)},
  MRNUMBER = {1288752},
}

@article {Barvinok2008,
    AUTHOR = {Barvinok, Alexander},
     TITLE = {Enumerating contingency tables via random permanents},
   JOURNAL = {Combin. Probab. Comput.},
  FJOURNAL = {Combinatorics, Probability and Computing},
    VOLUME = {17},
      YEAR = {2008},
    NUMBER = {1},
     PAGES = {1--19},
      ISSN = {0963-5483,1469-2163},
   MRCLASS = {05A15},
  MRNUMBER = {2376421},
MRREVIEWER = {Ruriko\ Yoshida},
       DOI = {10.1017/S0963548307008668},
       URL = {https://doi.org/10.1017/S0963548307008668},
}

@inproceedings {Kwok18,
    AUTHOR = {Kwok, Tsz Chiu and Lau, Lap Chi and Lee, Yin Tat and
              Ramachandran, Akshay},
     TITLE = {The {P}aulsen problem, continuous operator scaling, and
              smoothed analysis},
 BOOKTITLE = {S{TOC}'18---{P}roceedings of the 50th {A}nnual {ACM} {SIGACT}
              {S}ymposium on {T}heory of {C}omputing},
     PAGES = {182--189},
 PUBLISHER = {ACM, New York},
      YEAR = {2018},
      ISBN = {978-1-4503-5559-9},
   MRCLASS = {42C15},
  MRNUMBER = {3826245},
}

@article {Wigderson2020,
    AUTHOR = {Garg, Ankit and Gurvits, Leonid and Oliveira, Rafael and
              Wigderson, Avi},
     TITLE = {Operator scaling: theory and applications},
   JOURNAL = {Found. Comput. Math.},
  FJOURNAL = {Foundations of Computational Mathematics. The Journal of the
              Society for the Foundations of Computational Mathematics},
    VOLUME = {20},
      YEAR = {2020},
    NUMBER = {2},
     PAGES = {223--290},
      ISSN = {1615-3375,1615-3383},
   MRCLASS = {68W30 (15A03 15A09 16R10 68Q25 68W20)},
  MRNUMBER = {4081171},
       DOI = {10.1007/s10208-019-09417-z},
       URL = {https://doi.org/10.1007/s10208-019-09417-z},
}

@article{Pick1922,
author = {Pick, Georg},
title = {Über die Wurzeln der charakteristischen Gleichungen von Schwingungsproblemen},
journal = {ZAMM - Journal of Applied Mathematics and Mechanics / Zeitschrift für Angewandte Mathematik und Mechanik},
volume = {2},
number = {5},
pages = {353-357},
doi = {https://doi.org/10.1002/zamm.19220020503},
url = {https://onlinelibrary.wiley.com/doi/abs/10.1002/zamm.19220020503},
eprint = {https://onlinelibrary.wiley.com/doi/pdf/10.1002/zamm.19220020503},
year = {1922}
}

\end{document}